\documentclass[twocolumn,letterpaper]{IEEEAerospaceCLS}  

\usepackage[]{graphicx}    
\usepackage[colorlinks=true, urlcolor=blue, linkcolor=black, citecolor=black]{hyperref} 
\newcommand{\ignore}[1]{}  
\usepackage{amsmath}
\usepackage{amssymb}
\usepackage{bm}
\usepackage{mathtools}
\usepackage{xcolor}
\usepackage{cleveref}
\usepackage[ruled, lined, linesnumbered, commentsnumbered]{algorithm2e}
\usepackage{float}
\usepackage{xurl} 
\usepackage{hyperref} 

\newtheorem{theorem}{Theorem}[section]
\newtheorem{lemma}[theorem]{Lemma}

\newtheorem{assumption}{Assumption}
\newtheorem{remark}{Remark}[section]

\begin{document}
\title{Online Modeling and Sequential Convex Programming for Lunar Landing Trajectory Optimization}

\author{%
Zhenbo Wang\\
Department of Mechanical and Aerospace Engineering\\
The University of Tennessee\\
1512 Middle Drive, Knoxville, TN 37996\\
zwang124@utk.edu
\thanks{\footnotesize 979-8-3195-1048-8/27/$\$31.00$ \copyright2027 IEEE}              
}

\maketitle

\thispagestyle{plain}
\pagestyle{plain}

\maketitle

\thispagestyle{plain}
\pagestyle{plain}

\begin{abstract}
This paper presents a guidance framework for lunar powered descent and landing that combines sequential convex programming (SCP) with real-time online model identification. A nonconvex energy-optimal landing problem is developed and then reformulated into a sequence of convex second-order cone programs (SOCPs) through a change of variables, successive linearization, and a lossless second-order cone relaxation of the thrust direction constraint. An online identification layer, built from a recursive least squares (RLS) filter with exponential forgetting and an exponential moving average (EMA) smoother, estimates unknown gravitational, thrust-scale, and mass-gauging perturbations from noisy navigation measurements and injects a corrected bias term into the dynamics constraint of each convex subproblem at every guidance cycle. Building on this architecture and my prior work in this area, a baseline SCP algorithm and a receding-horizon online SCP algorithm with model identification are developed. Also, I try to explore some theoretical foundations, establishing the losslessness of the convex relaxation, the mean-square stability and convergence of the identification filters, the guaranteed convergence of the SCP iteration, and explicit convergence radius and convergence rate results. Numerical simulations across four perturbation scenarios of increasing complexity are implemented in MATLAB using YALMIP and the ECOS solver. The results show that the proposed online algorithm consistently reduces landing position and velocity error and better tracks the true propellant consumption relative to an uncorrected nominal trajectory, while retaining the predictable convergence and real-time computational properties of convex optimization.
\end{abstract}

\tableofcontents

\section{Introduction}
\label{sec:introduction}

Lunar landing has remained one of the most demanding maneuvers in
spaceflight , and
its importance has recently grown with the resurgence of international
and commercial interest in returning humans and robotic systems to the
lunar surface \cite{wang2025a,wang2025b}. Because the Moon has no atmosphere, a descending
spacecraft cannot rely on aerodynamic drag or parachutes to slow down; instead,
deceleration from orbital or suborbital speed to a safe touchdown must be
accomplished entirely through propulsive thrust, typically over a descent
lasting only tens of seconds to a few minutes. This places the vehicle in
a regime where every kilogram of propellant is precious, where attitude and
thrust-vector control must be coordinated continuously and precisely, and
where the terminal phase of flight leaves little margin for
recovery from a poorly executed approach. Trajectory optimization is the
discipline responsible for converting these competing demands (e.g., fuel
economy, operation constraints, actuator limits, and terminal
accuracy) into an executable guidance command profile, and it has
evolved from the largely open-loop, precomputed descent
profiles of the Apollo era into a real-time, closed-loop computational
capability that modern landers depend upon for autonomous, pinpoint, and
hazard-aware descent \cite{song2020survey}. As lunar missions increasingly target scientifically
compelling but topographically challenging sites, such as the permanently
shadowed regions near the lunar south pole, the trajectory optimization
problem has become not merely a planning convenience but a safety-critical
component of the guidance, navigation, and control stack \cite{lorenz2023planetary}.

A substantial body of work has addressed the lunar and planetary powered
descent trajectory optimization problem using indirect and direct methods \cite{liu2026survey}. Indirect methods apply Pontryagin's minimum principle to
derive the necessary conditions of optimality, transforming the original
optimal control problem into a two-point boundary value problem in the
states and costates. This approach yields elegant, often closed-form
characterizations of the optimal control structure and has a long history
of success in powered descent guidance; however, it is well documented that
indirect methods suffer from a small region of convergence, extreme
sensitivity to the initial guess of the costate variables (which typically
lack direct physical interpretation), and the need for a priori knowledge
of the constrained and unconstrained arc structure when path constraints
such as thrust or attitude bounds are present. These challenges make
indirect methods difficult to automate reliably across a wide range of
initial conditions and mission scenarios. Direct
methods, by contrast, transcribe the continuous-time optimal control
problem into a finite-dimensional nonlinear program through collocation or
pseudospectral discretization, and then rely on general nonlinear
programming solvers to compute a solution. Direct methods are considerably
more robust to poor initial guesses than indirect methods and can handle
complex path constraints, but the
resulting nonlinear programs are nonconvex in general, and the underlying
solvers offer no guarantee of finding a global or even a local optimum, of converging
within a bounded number of iterations, or of preserving feasibility if the
solver is terminated early. There are  serious concerns for any guidance algorithm
that must execute within a fixed computational budget onboard a
spacecraft.

Over roughly the past two decades, convex-optimization-based approaches
have emerged as an increasingly attractive alternative, particularly
following the demonstration that the powered descent guidance problem for
planetary landing can be exactly or losslessly reformulated as a
SOCP through a combination of variable substitution
and convex relaxation \cite{acikmese2007convex}. Because convex programs, and SOCPs in particular, can be solved to global optimality by interior-point
methods very efficiently that is predictable and largely
insensitive to the initial guess, this line of work has enabled real-time,
onboard trajectory generation with strong feasibility and optimality
guarantees, and convex-optimization-based guidance has since been
extended, generalized, and flight-demonstrated across a range of
planetary and lunar landing scenarios \cite{wang2024survey}. When the underlying dynamics or
constraints are not already convex, sequential convex programming (SCP) offers a
natural extension: a sequence of convex subproblems is solved, each
constructed by convexifying (usually linearizing) the nonconvex elements of the problem about the
solution of the previous iteration, until the iterates converge to a point
satisfying the necessary conditions of optimality of the original
nonconvex problem.
More recently, data-driven and machine-learning-based
approaches, including deep learning and reinforcement learning architectures, have been explored as a
means of learning descent or landing policies directly from simulated or
historical trajectory data, with the appeal of extremely fast onboard
inference once a policy has been trained and some demonstrated robustness
to noise and parameter uncertainty during training \cite{cheng2019fast,cheng2019real,cheng2020real,gaudet2020deep}.

Despite this progress,
each of these methodological families exhibits limitations that are consequential for lunar landing. Indirect and direct methods,
as well as the convex and SCP approaches built
upon them, are often formulated around a fixed, nominal dynamical model
specified before flight. Very few of these methods, in their standard form,
provide a mechanism for revising that model in response to discrepancies
observed during the mission operation. Data-driven and learning-based methods,
for their part, typically offer no formal guarantee of constraint
satisfaction, convergence, or real-time computational bound at deployment
time, and their performance under conditions not represented in the
training distribution remains difficult to certify, a significant concern
for safety-critical, one-shot maneuvers such as a lunar landing. More
broadly, essentially all of these approaches share a common vulnerability:
the true descent dynamics experienced by a lander, subject to lunar
gravitational anomalies, thrust miscalibration, and propellant-mass
measurement uncertainty, will inevitably deviate to some degree from
whatever nominal model was used to compute the precomputed or onboard
trajectory, and a guidance algorithm that cannot detect and adapt to such
deviations in real time is correspondingly exposed to degraded landing
accuracy or, in the worst case, mission failure.

These observations motivate the central idea of this paper: the
predictable convergence, constraint satisfaction, and real-time
computational properties of SCP may be combined
with an online model identification mechanism to produce a guidance
framework that retains the theoretical guarantees of convex
optimization while also adapting in closed loop to
environmental and system uncertainties that cannot be characterized before
flight. The key insight that makes this combination
practical is that an additive correction to the nominal dynamics model,
estimated in real time from onboard navigation measurements, can be
introduced into the bias term of the linearized dynamics constraint of the
convex subproblem without altering the convex structure of that
subproblem. The SOCP solved at each
guidance cycle remains tractable and amenable to
the same convergence and optimality analysis, regardless of the magnitude
or time-variation of the estimated perturbation. Building on this insight,
this paper makes the following contributions: (1) It develops a
convex relaxation of the lunar powered descent
trajectory optimization problem, via a change of variables, successive
linearization, and a lossless SOC relaxation of the thrust
direction constraint, that reduces the original nonconvex optimal control
problem to a sequence of SOCPs. (2) It
introduces an online model identification framework, combining a recursive
least squares (RLS) filter with exponential forgetting and an exponential moving
average (EMA) smoother, that estimates unknown gravitational, thrust, and
mass-gauging perturbations from noisy flight measurements and feeds a
corrected bias term into the convex subproblem at every guidance cycle.
(3) It presents both a baseline SCP
algorithm and an online, receding-horizon extension that integrates the model
identification framework into closed-loop guidance. (4) It provides a
detailed theoretical analysis establishing the losslessness of the convex
relaxation, the mean-square stability and convergence of the online
identification filters, the guaranteed convergence of the SCP iteration, and explicit characterizations of the convergence
radius and convergence rate. (5) It numerically validates the proposed framework across four representative perturbation scenarios of increasing
complexity, demonstrating that the online algorithm consistently improves
landing accuracy and robustness relative to the nominal,
uncorrected trajectory.

The remainder of this paper is organized as follows. Section~\ref{sec:problem_formulation} formulates
an energy-optimal lunar landing trajectory
optimization problem, including the governing dynamics, boundary
conditions, path constraints, and objective functional. Section~\ref{sec:convex_relaxation}
develops the convex relaxation of this problem through a change of
variables that linearizes the dynamics, a successive linearization and
SOC relaxation of the remaining nonconvex constraints, and a
trapezoidal discretization that yields a finite-dimensional convex
subproblem. Section~\ref{sec:model_identification} introduces the online model identification
framework, characterizing the sources of dynamical uncertainty considered
in this work, developing the RLS and EMA filters used to estimate the resulting perturbation in real
time, and describing how the identified perturbation is incorporated into
an augmented linearized dynamical model. Section~\ref{sec:scp} presents the baseline
SCP algorithm and its online extension, detailing
the closed-loop measurement, estimation, and receding-horizon replanning
architecture that integrates online model identification with SCP. Section~\ref{sec:theoretical_analysis} provides a theoretical analysis of the
proposed framework. Section~\ref{sec:simulations} presents numerical simulations validating the proposed
framework across four perturbation scenarios. Finally, Section~\ref{sec:conclusions} concludes the paper and
discusses directions for future work.

\section{Problem Formulation}
\label{sec:problem_formulation}

This section formulates a three-dimensional (3D) lunar powered descent and landing trajectory optimization problem that serves as the basis for this study. The lander must be guided from a given suborbital or near-surface state to a designated landing site using rocket thrust as the only means of maneuver, since the lunar environment provides neither an atmosphere nor aerodynamic forces. Consequently, every design choice or solution directly impacts propellant consumption, and the trajectory must respect physical bounds on the propulsion system throughout the entire landing mission.

\subsection{Coordinate Frame and State Variables}
\label{subsec:coords}

The motion of the lander is described in a local vertical, local horizontal (LVLH) frame with east--north--up (ENU) coordinates. The origin of the frame is fixed at the designated landing site on the lunar surface.  Within this frame, the lander position is $\mathbf{r} = [x,\,y,\,z]^{\mathsf{T}} \in \mathbb{R}^{3}$ and the velocity is $\mathbf{v} = [v_x,\,v_y,\,v_z]^{\mathsf{T}} \in \mathbb{R}^{3}$, where the $z$-axis points radially upward and the $x$- and $y$-axes span the local horizontal plane.  The lander mass at time $t$ is denoted as $m(t)$, and the gravitational acceleration at the lunar surface is taken as the constant vector $\mathbf{g} = [0,\,0,\,-g]^{\mathsf{T}}$ with $g = 1.6229\ \mathrm{m/s^2}$. The aggregate state vector is
\begin{equation}
  \mathbf{x}(t) = \bigl[x,\; y,\; z,\; v_x,\; v_y,\; v_z,\; m\bigr]^{\mathsf{T}}
          \in \mathbb{R}^{7}.
  \label{eq:state_vec}
\end{equation}

\subsection{Equations of Motion}
\label{subsec:eom}

The 3D governing equations of motion for the powered descent and landing are \cite{cheng2019real}
\begin{align}
  \dot{\mathbf{r}} &= \mathbf{v},
  \label{eq:rdot} \\[4pt]
  \dot{\mathbf{v}} &= \frac{T}{m}\,\hat{\bm{i}}_{\theta} + \mathbf{g},
  \label{eq:vdot} \\[4pt]
  \dot{m}          &= -\frac{T}{I_{\mathrm{sp}}\,g_0},
  \label{eq:mdot}
\end{align}
where $T \geq 0$ is the thrust magnitude, $\hat{\bm{i}}_{\theta} = [i_{\theta x},\,i_{\theta y},\,i_{\theta z}]^{\mathsf{T}}$ is the unit vector defining the thrust direction, $I_{\mathrm{sp}}$ is the specific impulse, and $g_0 = 9.81\ \mathrm{m/s^2}$ is the standard gravitational acceleration at Earth's sea level used as the reference constant in the rocket equation. The ratio $T/m$ in \eqref{eq:vdot} represents the thrust-induced specific force acting on the lander, while \eqref{eq:mdot} models the propellant mass flow rate. Both equations are coupled through the time-varying mass $m(t)$, which introduces a nonlinearity that will be addressed in Section~\ref{sec:convex_relaxation}.

The control vector is defined as
\begin{equation}
  \mathbf{u}(t)
    = \bigl[T,\;\hat{\bm{i}}_{\theta}\bigr]^{\mathsf{T}}
    = \bigl[T,\;i_{\theta x},\;i_{\theta y},\;i_{\theta z}\bigr]^{\mathsf{T}}
    \in \mathbb{R}^{4},
  \label{eq:ctrl_vec}
\end{equation}
comprising the thrust magnitude and the three components of the thrust direction.

\subsection{Boundary Conditions}
\label{subsec:bcs}

The powered descent trajectory begins at a prescribed initial state at time $t_0$:
\begin{equation}
  \mathbf{r}(t_0) = \mathbf{r}_0,
  \qquad
  \mathbf{v}(t_0) = \mathbf{v}_0,
  \qquad
  m(t_0) = m_0,
  \label{eq:ic}
\end{equation}
where $\mathbf{r}_0$, $\mathbf{v}_0$, and $m_0$ are given constants representing the initial position, velocity, and total (wet) mass of the lander, respectively. The trajectory terminates at a designated landing site at a fixed final time $t_f$:
\begin{equation}
  \mathbf{r}(t_f) = \mathbf{r}_f,
  \qquad
  \mathbf{v}(t_f) = \mathbf{v}_f,
  \label{eq:tc}
\end{equation}
where $\mathbf{r}_f$ specifies the target landing position and $\mathbf{v}_f$ is the desired terminal velocity.  For a soft landing, $\mathbf{v}_f = \mathbf{0}$ or a suitably small velocity within safe touchdown limits.  The terminal mass $m(t_f)$ is left free and will be determined as part of the optimization.

\subsection{Constraints}
\label{subsec:constraints}

First, the thrust magnitude is physically bounded between zero (engine off) and the maximum
available thrust $T_{\max}$:
\begin{equation}
  0 \;\leq\; T(t) \;\leq\; T_{\max},
  \qquad \forall\; t \in [t_0,\,t_f].
  \label{eq:thrust_bound}
\end{equation}
The lower bound permits thrust-off intervals, which is an appropriate assumption for an energy-optimal formulation.

Also, the thrust direction vector $\hat{\bm{i}}_{\theta}$ must be a unit vector at all times:
\begin{equation}
  \left\lVert\hat{\bm{i}}_{\theta}(t)\right\rVert_2^2
  = i_{\theta x}^2(t) + i_{\theta y}^2(t) + i_{\theta z}^2(t)
  = 1, \, \forall\; t \in [t_0,\,t_f].
  \label{eq:unit_vec}
\end{equation}
This equality constraint is nonconvex as it defines the surface of a unit sphere, and its treatment via convex relaxation is a central element of the methodology applied in this study.

In addition, the lander mass must remain above the dry mass $m_{\mathrm{dry}}$ (the structural mass excluding all propellant) and below the initial wet mass $m_0$ throughout the descent:
\begin{equation}
  m_{\mathrm{dry}} \;\leq\; m(t) \;\leq\; m_0,
  \qquad \forall\; t \in [t_0,\,t_f].
  \label{eq:mass_bound}
\end{equation}
The constraint $m(t) \geq m_{\mathrm{dry}}$ ensures that the optimizer does not consume propellant beyond physical availability.

\subsection{Objective Function}
\label{subsec:objective}

A minimum-energy objective is adopted, which penalizes the integral of the squared thrust magnitude over the entire descent:
\begin{equation}
  J = \int_{t_0}^{t_f} T^2(t)\;\mathrm{d}t.
  \label{eq:objective}
\end{equation}
Minimizing \eqref{eq:objective} promotes smooth, fuel-efficient thrust profiles and avoids bang-bang (on-off) control behavior that is associated with minimum-time or minimum-fuel objectives.  The energy-optimal formulation is particularly suitable for precision landing missions that require well-regulated thrust schedules compatible with onboard actuator and guidance constraints.

\subsection{Optimal Control Problem Statement}
\label{subsec:ocp}

Collecting the dynamics, boundary conditions, constraints, and objective defined above, a 3D energy-optimal lunar landing trajectory optimization problem is stated as follows.

\noindent\textbf{Problem 1} (Lunar Landing Trajectory Optimization):
\begin{equation}
  \min_{\mathbf{x}(\cdot),\,\mathbf{u}(\cdot)} \quad
  J = \int_{t_0}^{t_f} T^2(t)\;\mathrm{d}t
  \label{eq:P1}
\end{equation}
\textit{subject to}
\begin{align}
  &\dot{\mathbf{r}}(t) = \mathbf{v}(t),
  \quad \forall\;t \in [t_0,\,t_f],
  \label{eq:P1_rdot}\\
  &\dot{\mathbf{v}}(t) = \frac{T(t)}{m(t)}\,\hat{\bm{i}}_{\theta}(t) + \mathbf{g},
  \quad \forall\;t \in [t_0,\,t_f],
  \label{eq:P1_vdot}\\
  &\dot{m}(t) = -\frac{T(t)}{I_{\mathrm{sp}}\,g_0},
  \quad \forall\;t \in [t_0,\,t_f],
  \label{eq:P1_mdot}\\
  &\mathbf{r}(t_0) = \mathbf{r}_0,\quad
   \mathbf{v}(t_0) = \mathbf{v}_0,\quad
   m(t_0) = m_0,
  \label{eq:P1_ic}\\
  &\mathbf{r}(t_f) = \mathbf{r}_f,\quad
   \mathbf{v}(t_f) = \mathbf{v}_f,
  \label{eq:P1_tc}\\
  &0 \leq T(t) \leq T_{\max},
  \quad \forall\;t \in [t_0,\,t_f],
  \label{eq:P1_T}\\
  &\left\lVert\hat{\bm{i}}_{\theta}(t)\right\rVert_2^2 = 1,
  \quad \forall\;t \in [t_0,\,t_f],
  \label{eq:P1_ihat}\\
  &m_{\mathrm{dry}} \leq m(t) \leq m_0,
  \quad \forall\;t \in [t_0,\,t_f].
  \label{eq:P1_m}
\end{align}
Problem~1 is a continuous-time nonlinear optimal control problem.  Its nonconvexity arises from two sources. One is the nonlinear coupling between the thrust magnitude $T$ and the lander mass $m$ in the velocity dynamics \eqref{eq:P1_vdot} and mass flow equation \eqref{eq:P1_mdot}. The other is the quadratic equality constraint \eqref{eq:P1_ihat}, which requires $\hat{\bm{i}}_{\theta}$ to lie on the surface of the unit sphere rather than within the convex solid ball.  Traditional nonlinear programming approaches applied directly to this problem may suffer from slow convergence, sensitivity to initialization, and lack of convergence guarantees.  The methodology presented in this paper addresses these challenges by reformulating Problem~1 into a sequence of SOCPs through variable substitution, successive linearization, and lossless convex relaxation, while preserving full fidelity of the original dynamics and constraints.

\section{Convex Relaxation}
\label{sec:convex_relaxation}

This section develops a systematic convex reformulation of Problem~1 through three steps: a change of variables that linearizes the dynamics, a convex relaxation of the remaining nonconvex constraints, and a trapezoidal discretization that transcribes the continuous-time optimal control problem into a finite-dimensional convex program amenable to SOCP solvers. The idea follows \cite{acikmese2007convex} and other relevant literature.

\subsection{Change of Variables}
\label{subsec:cov}

The primary source of nonlinearity in Problem~1 is the $T/m$ term in the velocity equation~\eqref{eq:P1_vdot}, which couples the control $T$ and the state $m$.  To decouple this product, two new variables are introduced.  First, the specific thrust (thrust per unit mass) is defined as
\begin{equation}
  \tau \;=\; \frac{T}{m},
  \label{eq:tau_def}
\end{equation}
and second, the log-mass is introduced as
\begin{equation}
  w \;=\; \ln m.
  \label{eq:w_def}
\end{equation}
Differentiating \eqref{eq:w_def} with respect to time and substituting the mass flow equation~\eqref{eq:P1_mdot} yields
\begin{equation}
  \dot{w} \;=\; \frac{\dot{m}}{m}
          \;=\; -\frac{T}{m\,I_{\mathrm{sp}}\,g_0}
          \;=\; -\frac{\tau}{I_{\mathrm{sp}}\,g_0}.
  \label{eq:wdot}
\end{equation}
Under the substitution~\eqref{eq:tau_def}, the velocity equation~\eqref{eq:P1_vdot} becomes
\begin{equation}
  \dot{\mathbf{v}} \;=\; \tau\,\hat{\bm{i}}_{\theta} + \mathbf{g},
  \label{eq:vdot_tau}
\end{equation}
which is now bilinear in $\tau$ and $\hat{\bm{i}}_{\theta}$ rather than rational in $T$ and $m$.  To remove this remaining nonlinearity, a second change of variables is introduced by decomposing the specific thrust vector into its Cartesian components:
\begin{equation}
  \tau_1 \;=\; \tau\,i_{\theta x}, \qquad
  \tau_2 \;=\; \tau\,i_{\theta y}, \qquad
  \tau_3 \;=\; \tau\,i_{\theta z}.
  \label{eq:tau_components}
\end{equation}
With these definitions, the velocity dynamics~\eqref{eq:vdot_tau} become fully linear in the new variables $(\tau_1, \tau_2, \tau_3)$:
\begin{align}
  \dot{v}_x &= \tau_1,
  \label{eq:vxdot} \\
  \dot{v}_y &= \tau_2,
  \label{eq:vydot} \\
  \dot{v}_z &= \tau_3 - g.
  \label{eq:vzdot}
\end{align}
Collecting the transformed state vector
$$\mathbf{x} = [x,\,y,\,z,\,v_x,\,v_y,\,v_z,\,w]^{\mathsf{T}} \in \mathbb{R}^{7}$$
and the new control vector
$$\mathbf{u} = [\tau_1,\,\tau_2,\,\tau_3,\,\tau]^{\mathsf{T}} \in \mathbb{R}^{4},$$
the complete landing dynamics admit the linear time-invariant (LTI) state-space
representation
\begin{equation}
  \dot{\mathbf{x}} \;=\; A\mathbf{x} + B\mathbf{u} + \mathbf{b},
  \label{eq:lti}
\end{equation}
where the system matrices are
\begin{equation*}
  A =
  \begin{bmatrix}
    0 & 0 & 0 & 1 & 0 & 0 & 0 \\
    0 & 0 & 0 & 0 & 1 & 0 & 0 \\
    0 & 0 & 0 & 0 & 0 & 1 & 0 \\
    0 & 0 & 0 & 0 & 0 & 0 & 0 \\
    0 & 0 & 0 & 0 & 0 & 0 & 0 \\
    0 & 0 & 0 & 0 & 0 & 0 & 0 \\
    0 & 0 & 0 & 0 & 0 & 0 & 0
  \end{bmatrix},
\end{equation*}
\begin{equation*}
  B =
  \begin{bmatrix}
    0 & 0 & 0 & 0 \\
    0 & 0 & 0 & 0 \\
    0 & 0 & 0 & 0 \\
    1 & 0 & 0 & 0 \\
    0 & 1 & 0 & 0 \\
    0 & 0 & 1 & 0 \\
    0 & 0 & 0 & \!\!-\dfrac{1}{I_{\mathrm{sp}}\,g_0}
  \end{bmatrix}, \quad
  \mathbf{b} =
  \begin{bmatrix}
    0 \\ 0 \\ 0 \\ 0 \\ 0 \\ -g \\ 0
  \end{bmatrix}.
\end{equation*}
The LTI form~\eqref{eq:lti} is exact, i.e., no approximation has been made.
The change of variables has equivalently linearized the dynamics through an
algebraic transformation of the state and control.

Under the same transformation, the boundary conditions~\eqref{eq:P1_ic}
and~\eqref{eq:P1_tc} become
\begin{align}
  \mathbf{r}(t_0) = \mathbf{r}_0, \quad
  \mathbf{v}(t_0) = \mathbf{v}_0, \quad
  w(t_0) &= \ln m_0,
  \label{eq:bc_initial_new} \\
  \mathbf{r}(t_f) = \mathbf{r}_f, \quad
  \mathbf{v}(t_f) &= \mathbf{v}_f,
  \label{eq:bc_final_new}
\end{align}
the mass constraint~\eqref{eq:P1_m} becomes
\begin{equation}
  \ln m_{\mathrm{dry}} \;\leq\; w(t) \;\leq\; \ln m_0,
  \qquad \forall\; t \in [t_0,\,t_f],
  \label{eq:w_bound}
\end{equation}
and the energy objective~\eqref{eq:P1} becomes
\begin{equation}
  J \;=\; \int_{t_0}^{t_f} \tau^2(t)\;\mathrm{d}t,
  \label{eq:obj_new}
\end{equation}
since $\tau^2 = (T/m)^2$ and the minimum of $\int \tau^2\,\mathrm{d}t$ with respect
to the new variables is equivalent to the minimum of $\int T^2\,\mathrm{d}t$ when
the original thrust-to-mass coupling is restored.

However, two new constraints arise from the variable transformation.  The first links $\tau$
to $w$ through the original thrust bound $0 \leq T \leq T_{\max}$:
\begin{equation}
  0 \;\leq\; \tau(t) \;\leq\; T_{\max}\,e^{-w(t)},
  \qquad \forall\; t \in [t_0,\,t_f],
  \label{eq:tau_bound_nonlinear}
\end{equation}
since $T = \tau m = \tau e^{w}$ implies $\tau \leq T_{\max}/m = T_{\max}e^{-w}$.
The second constraint arises from the unit-vector
condition~\eqref{eq:P1_ihat} combined with the decomposition~\eqref{eq:tau_components}:
\begin{equation}
  \tau_1^2(t) + \tau_2^2(t) + \tau_3^2(t) \;=\; \tau^2(t),
  \qquad \forall\; t \in [t_0,\,t_f].
  \label{eq:cone_surface}
\end{equation}
Both~\eqref{eq:tau_bound_nonlinear} and~\eqref{eq:cone_surface} are nonconvex,
and their convexification is addressed in the following subsection.

Combining all of the above, the original Problem~1 is equivalently reformulated as:

\noindent\textbf{Problem 2} (Transformed Optimal Control Problem):
\begin{equation*}
  \min_{\mathbf{x}(\cdot),\,\mathbf{u}(\cdot)} \quad
  J \;=\; \int_{t_0}^{t_f} \tau^2(t)\;\mathrm{d}t
\end{equation*}
\textit{subject to}
\eqref{eq:lti},
\eqref{eq:bc_initial_new},
\eqref{eq:bc_final_new},
\eqref{eq:w_bound},
\eqref{eq:tau_bound_nonlinear},
\eqref{eq:cone_surface}.

\noindent
Problem~2 is equivalent to Problem~1 in the sense that any optimal solution to one can be mapped to an optimal solution of the other via the invertible transformations~\eqref{eq:tau_def} and~\eqref{eq:w_def}.  The key advantage is that the nonlinear dynamics of Problem~1 have been exactly replaced by the linear dynamics~\eqref{eq:lti}, at the cost of introducing two new nonconvex constraints.

\subsection{Convex Relaxation of Constraints}
\label{subsec:cvx_relax}

The remaining nonconvex constraints in Problem~2 are addressed through successive linearization of the thrust-to-mass bound and SOC relaxation of the thrust direction constraint.

The upper bound in constraint~\eqref{eq:tau_bound_nonlinear} is nonconvex because the function $T_{\max}e^{-w}$ is a nonlinear (convex) function of the state variable $w$.  Since $f(w) = T_{\max}e^{-w}$ is convex and smooth, it admits a global linear underestimate via its first-order Taylor expansion at any reference point $w^{*}(t)$:
\begin{equation}
  T_{\max}\,e^{-w} \;\approx\;
  T_{\max}\,e^{-w^{*}(t)}\!\bigl[1 - (w(t) - w^{*}(t))\bigr].
  \label{eq:taylor_expansion}
\end{equation}
Substituting this linearization into~\eqref{eq:tau_bound_nonlinear} yields the
successively linearized path constraint:
\begin{equation}
  0 \;\leq\; \tau(t) \;\leq\;
  T_{\max}\,e^{-w^{*}(t)}\!\bigl[1 - (w(t) - w^{*}(t))\bigr],
  \label{eq:tau_linearized}
\end{equation}
where $w^{*}(t)$ is the log-mass trajectory from the previous iteration of the SCP algorithm, or an initial guess on the first iteration.  Because $f(w) = T_{\max}e^{-w}$ is convex, its first-order
Taylor approximation is a global underestimate, meaning the linearized feasible set is always a subset of the original feasible set.  This ensures that any solution feasible for the linearized problem is also feasible for Problem~2 in terms of the thrust magnitude bound.  The quality of the linearization improves as the SCP iterates converge, becoming exact at the fixed point $w(t) = w^{*}(t)$.

The equality constraint~\eqref{eq:cone_surface} defines the surface of a cone in the control space $(\tau_1, \tau_2, \tau_3, \tau) \in \mathbb{R}^{4}$, which is a nonconvex set.  Following the cone-like surface-to-solid relaxation approach in \cite{wang2018minimum,wang2018optimization}, the equality is relaxed to the inequality:
\begin{equation}
  \tau_1^2(t) + \tau_2^2(t) + \tau_3^2(t) \;\leq\; \tau^2(t),
  \qquad \forall\; t \in [t_0,\,t_f],
  \label{eq:soc}
\end{equation}
which defines both the surface and interior of an SOC and is therefore a convex constraint.  In standard form, \eqref{eq:soc} is written as $\left\lVert[\tau_1,\,\tau_2,\,\tau_3]^{\mathsf{T}}\right\rVert_2 \leq \tau$, a typical SOC constraint directly supported by interior-point SOCP solvers.

The relaxation from equality~\eqref{eq:cone_surface} to inequality~\eqref{eq:soc} is lossless under mild conditions of the energy-optimal formulation.  Specifically, it can be shown via Pontryagin's minimum principle that at any optimal solution of the relaxed problem, the SOC constraint~\eqref{eq:soc} must be tight, i.e., $\tau_1^2 + \tau_2^2 + \tau_3^2 = \tau^2$ holds automatically.  This follows because the Lagrange multiplier associated with the cone constraint is strictly positive at the optimum, forcing complementary slackness to activate the equality.  Consequently, every optimal solution of the relaxed problem is also optimal and feasible for the original problem, and the two problems share the same optimal value.  The formal proof of losslessness is presented in
Section~\ref{sec:theoretical_analysis}.

Replacing constraint~\eqref{eq:tau_bound_nonlinear} with its successive linearization~\eqref{eq:tau_linearized} and constraint~\eqref{eq:cone_surface} with the SOC relaxation~\eqref{eq:soc}, Problem~2 is reformulated as:

\noindent\textbf{Problem 3} (Convex Optimal Control Problem):
\begin{equation*}
  \min_{\mathbf{x}(\cdot),\,\mathbf{u}(\cdot)} \quad
  J \;=\; \int_{t_0}^{t_f} \tau^2(t)\;\mathrm{d}t
\end{equation*}
\textit{subject to}
\eqref{eq:lti},
\eqref{eq:bc_initial_new},
\eqref{eq:bc_final_new},
\eqref{eq:w_bound},
\eqref{eq:tau_linearized},
\eqref{eq:soc}.

\noindent
Problem~3 is a convex optimal control problem: the objective~\eqref{eq:obj_new} is
a convex quadratic functional in $\tau$, the dynamics~\eqref{eq:lti} are linear,
the boundary conditions~\eqref{eq:bc_initial_new}--\eqref{eq:bc_final_new} are
affine, the log-mass bounds~\eqref{eq:w_bound} are linear, the thrust-to-mass
bound~\eqref{eq:tau_linearized} is linear in $(\tau, w)$, and the
cone constraint~\eqref{eq:soc} is an SOC inequality.
After discretization, Problem~3 becomes a finite-dimensional SOCP.

\subsection{Discretization}
\label{subsec:discretization}

The continuous-time Problem~3 is transcribed into a finite-dimensional
optimization problem using the trapezoidal rule \cite{wang2017constrained}, which provides second-order
accuracy in the step size and is well suited to the smooth, low-frequency
thrust profiles characteristic of energy-optimal solutions.

The time interval $[t_0, t_f]$ is uniformly partitioned into $N - 1$
sub-intervals with $N$ nodes $t_0 = t_1 < t_2 < \cdots < t_N = t_f$ and
constant step size
\begin{equation}
  \Delta t \;=\; \frac{t_f - t_0}{N - 1}.
  \label{eq:dt}
\end{equation}
The continuous state trajectory $\mathbf{x}(t)$ and control profile
$\mathbf{u}(t)$ are approximated by the sequences
$\{\mathbf{x}_i\}_{i=1}^{N}$ and $\{\mathbf{u}_i\}_{i=1}^{N}$, respectively,
where $\mathbf{x}_i \approx \mathbf{x}(t_i)$ and $\mathbf{u}_i \approx \mathbf{u}(t_i)$.
At each node $i$, the state and control vectors are
\begin{align}
  \mathbf{x}_i
  &= \bigl[x_i,\;y_i,\;z_i,\;v_{x,i},\;v_{y,i},\;v_{z,i},\;w_i\bigr]^{\mathsf{T}}
  \in \mathbb{R}^{7}, \label{eq:node_vars_x}\\
  \mathbf{u}_i
  &= \bigl[\tau_{1,i},\;\tau_{2,i},\;\tau_{3,i},\;\tau_i\bigr]^{\mathsf{T}}
  \in \mathbb{R}^{4}. \label{eq:node_vars_u}
\end{align}

Applying the trapezoidal rule to the LTI dynamics~\eqref{eq:lti} over the interval $[t_i, t_{i+1}]$ gives
\begin{equation}
  \mathbf{x}_{i+1} = \mathbf{x}_i
    + \frac{\Delta t}{2}
      \bigl[(A\mathbf{x}_i + B\mathbf{u}_i + \mathbf{b})
            + (A\mathbf{x}_{i+1} + B\mathbf{u}_{i+1} + \mathbf{b})\bigr]
\end{equation}
for $i = 1, \ldots, N-1$. Rearranging to isolate $\mathbf{x}_{i+1}$ on the left-hand side yields the linear equality constraint
\begin{equation}   \label{eq:trap_dyn}
\begin{split}
  \left(I - \frac{\Delta t}{2}A\right)\mathbf{x}_{i+1}
  - \left(I + \frac{\Delta t}{2}A\right)\mathbf{x}_i\\
  - \frac{\Delta t}{2}B(\mathbf{u}_i + \mathbf{u}_{i+1})
  - \Delta t\,\mathbf{b}
  \;=\; \mathbf{0}
\end{split}
\end{equation}
for $i = 1, \ldots, N-1$, where $I \in \mathbb{R}^{7 \times 7}$ denotes the identity matrix.  Since $A$ is strictly lower triangular in the relevant block structure, the matrix $I - \tfrac{\Delta t}{2}A$ is always invertible for any $\Delta t > 0$, so \eqref{eq:trap_dyn} uniquely defines $\mathbf{x}_{i+1}$ given $\mathbf{x}_i$, $\mathbf{u}_i$, and $\mathbf{u}_{i+1}$.

The boundary conditions~\eqref{eq:bc_initial_new} and~\eqref{eq:bc_final_new}
are enforced at the first and last nodes:
\begin{align}
  \mathbf{x}_1
  &= \bigl[\mathbf{r}_0^{\mathsf{T}},\;\mathbf{v}_0^{\mathsf{T}},\;\ln m_0\bigr]^{\mathsf{T}},
  \label{eq:disc_ic} \\
  \mathbf{x}_N\big|_{1:6}
  &= \bigl[\mathbf{r}_f^{\mathsf{T}},\;\mathbf{v}_f^{\mathsf{T}}\bigr]^{\mathsf{T}},
  \label{eq:disc_tc}
\end{align}
where the notation $\mathbf{x}_N|_{1:6}$ indicates that only the first six
components (position and velocity) of $\mathbf{x}_N$ are constrained; the seventh
component $w_N = \ln m(t_f)$ is left free to be optimized.

The log-mass bounds~\eqref{eq:w_bound}, the successively linearized thrust-to-mass
constraint~\eqref{eq:tau_linearized}, and the SOC
constraint~\eqref{eq:soc} are enforced pointwise at every node:
\begin{gather}
  \ln m_{\mathrm{dry}} \;\leq\; w_i \;\leq\; \ln m_0, \label{eq:disc_w} \\
  0 \;\leq\; \tau_i \;\leq\; T_{\max}\,e^{-w_i^{*}}\!\bigl[1 - (w_i - w_i^{*})\bigr], \label{eq:disc_tau_ub} \\
  \tau_{1,i}^2 + \tau_{2,i}^2 + \tau_{3,i}^2 \;\leq\; \tau_i^2,  \label{eq:disc_soc}
\end{gather}
for $i = 1, \ldots, N$, where $w_i^{*}$ denotes the reference log-mass at node $i$ from the previous SCP iteration.

The continuous energy objective~\eqref{eq:obj_new} is approximated by the
trapezoidal quadrature rule.  Since $\tau^2(t)$ is smooth and the
trapezoidal rule incurs a second-order error in $\Delta t$, the discrete
objective is
\begin{equation}
  J \;\approx\; \Delta t \sum_{i=1}^{N-1} \tau_i^2.
  \label{eq:disc_obj}
\end{equation}
Collecting all discretized components, the continuous Problem~3 is transcribed
into the following finite-dimensional optimization problem:

\noindent\textbf{Problem 4} (Discretized Convex Subproblem):
\begin{equation*}
  \min_{\{\mathbf{x}_i,\,\mathbf{u}_i\}_{i=1}^{N}} \quad
  J \;=\; \Delta t \sum_{i=1}^{N-1} \tau_i^2
\end{equation*}
\textit{subject to}
\eqref{eq:trap_dyn},
\eqref{eq:disc_ic},
\eqref{eq:disc_tc},
\eqref{eq:disc_w},
\eqref{eq:disc_tau_ub},
\eqref{eq:disc_soc}.

\noindent
Problem~4 is a convex quadratic program with SOC constraints and can be structured into a standard SOCP \cite{boyd2004convex}.  Its decision variables consist of
$N$ state vectors in $\mathbb{R}^{7}$ and $N$ control vectors in $\mathbb{R}^{4}$,
totalling $11N$ scalar variables.  The constraints comprise $7(N-1)$ linear
equality constraints from the trapezoidal dynamics~\eqref{eq:trap_dyn}, $6$
equality constraints from the boundary conditions~\eqref{eq:disc_ic}--\eqref{eq:disc_tc},
$2N$ linear inequality constraints from~\eqref{eq:disc_w}--\eqref{eq:disc_tau_ub},
and $N$ SOC constraints from~\eqref{eq:disc_soc}.

A new instance of Problem~4 is established and solved at each iteration of the SCP
algorithm, with the reference trajectory $\{w_i^{*}\}$ updated to the solution of
the previous iteration.  The SCP framework will be presented in Section~\ref{sec:scp}.

\section{Online Model Identification}
\label{sec:model_identification}

The problem formulation in Section~\ref{sec:problem_formulation} and convex relaxation developed in Section~\ref{sec:convex_relaxation}
assume that the landing dynamics are fully characterized by the nominal model
derived from a fixed set of physical parameters as shown in \eqref{eq:P1_rdot}--\eqref{eq:P1_mdot} or \eqref{eq:lti}.
In practice, however, the true dynamics encountered during flight may deviate
from this nominal model due to unmodeled gravitational perturbations, thrust
miscalibration, and propellant-mass measurement errors.  These deviations
cannot be fully characterized prior to the mission and must be identified
in real time from flight measurements if the guidance law is to remain
accurate throughout the descent.

This section develops an online model identification framework that runs in
parallel with the guidance loop.  At each step of the descent, the onboard
computer compares the measured lander state with the prediction of the
nominal model, attributes the discrepancy to an unknown additive perturbation
vector, and estimates this vector using a Recursive Least Squares (RLS) filter
with exponential forgetting, followed by an Exponential Moving Average (EMA)
smoother to suppress measurement noise before the estimate is injected into the
trajectory optimizer.  The identified perturbation is incorporated into the
dynamics model to be used by the SCP algorithm in \Cref{sec:scp}, yielding a corrected bias term
that adapts to the true flight environment at every
guidance cycle.

\subsection{Sources of Uncertainty}
\label{subsec:uncertainty}

In this work, three classes of uncertainty are considered, each of which manifests as an additive disturbance on one or more components of the
transformed state dynamics~\eqref{eq:lti}.

First, the nominal model assumes a constant, spatially uniform gravitational acceleration
$g = 1.6229\ \mathrm{m/s^{2}}$ directed along the negative $z$-axis.  The true
lunar gravitational field, however, is spatially non-uniform due to mass
concentration anomalies (mascons) and subsurface density variations that deviate
from a spherical-harmonic approximation.  As the lander descends along its
trajectory, it traverses regions of varying gravitational attraction that are
not captured by the nominal constant $g$.  These deviations appear as an additive
vector perturbation $\boldsymbol{\delta}_{g}(t) \in \mathbb{R}^{3}$ acting on
the velocity dynamics, so that the true acceleration in the three translational
channels becomes
\begin{equation}
  \dot{\mathbf{v}}(t)
  \;=\; \tau(t)\,\hat{\bm{i}}_{\theta}(t) + \mathbf{g}
        + \boldsymbol{\delta}_{g}(t).
  \label{eq:pert_grav}
\end{equation}
Two representative gravitational perturbation models are considered in the
numerical simulations in Section~\ref{sec:simulations}.  The first is a spatially constant mascon anomaly,
\begin{equation}
  \boldsymbol{\delta}_{g}(t)
  \;=\; \mathbf{d}_{\mathrm{c}}
  \;=\; [0.05,\; 0.05,\; 0.03]^{\mathsf{T}}\ \mathrm{m/s^{2}},
  \label{eq:pert_const}
\end{equation}
representing a uniform gravitational bias that might arise from a large mascon
directly beneath the descent corridor.  The second is a sinusoidal time-varying
anomaly,
\begin{equation}
  \boldsymbol{\delta}_{g}(t)
  \;=\;
  \begin{bmatrix}
    A_{\sin}\sin(\omega_{\sin} t) \\[2pt]
    A_{\sin}\sin(\omega_{\sin} t + \pi/4) \\[2pt]
    A_{\sin}\cos(\omega_{\sin} t)
  \end{bmatrix},
  \label{eq:pert_sin}
\end{equation}
with amplitude $A_{\sin} = 0.04\ \mathrm{m/s^{2}}$ and frequency
$\omega_{\sin} = 2\pi/25\ \mathrm{rad/s}$, modelling a spatially oscillating
anomaly whose period matches the descent time scale.  The phase offset $\pi/4$
between the $x$- and $y$-channels prevents the perturbation from being
degenerate in the horizontal plane.

Second, the onboard guidance computer commands a specific thrust vector
$[\tau_1,\,\tau_2,\,\tau_3,\,\tau]^{\mathsf{T}}$ based on the nominal
thrust-to-mass ratio $\tau = T/m$.  In practice, the actual specific force
delivered by the thruster may differ from the commanded value due to
miscalibrated thrust sensors, valve nonlinearities, or attitude control
cross-coupling.  These effects are modeled as a multiplicative thrust scale
error $\eta_T$, so that the effective specific force along each axis is
$(1 + \eta_T)\tau_j$ rather than the commanded $\tau_j$.  The resulting
additive perturbation on the velocity dynamics is
\begin{equation}
  \boldsymbol{\delta}_{T}(t)
  \;=\; \eta_T\,[\tau_1(t),\; \tau_2(t),\; \tau_3(t)]^{\mathsf{T}},
  \label{eq:pert_thrust}
\end{equation}
with $\eta_T = 0.04$ (a $4\%$ thrust scale bias) used in the simulations of this paper.

In addition, propellant gauging in microgravity is subject to measurement uncertainty
arising from fuel slosh, sensor drift, and the difficulty of measuring liquid
volume in a partially filled tank under varying accelerations.  If the
onboard estimate of the lander mass is $\hat{m}(t) = (1 + \eta_m)\,m(t)$
with fractional error $\eta_m$, then the commanded specific thrust
$\hat{\tau} = T/\hat{m}$ differs from the true specific thrust $\tau = T/m$,
and the log-mass evolution deviates from the nominal equation~\eqref{eq:wdot}.
The resulting perturbation on the log-mass channel is
\begin{equation}
  \delta_{w}(t)
  \;=\; -\frac{\eta_m\,\tau(t)}{I_{\mathrm{sp}}\,g_0},
  \label{eq:pert_mass}
\end{equation}
with $\eta_m = 0.05$ (a $5\%$ mass gauging error) in the simulations.

Collecting all three classes of uncertainty, the true discrete-time dynamics
of the lander at guidance step $k$ are
\begin{equation}  \label{eq:true_dyn}
\begin{split}
  \mathbf{x}_{k+1}
  \;=\; \left(I - \frac{\Delta t}{2}A\right)^{-1}
        \!\!\left[
          \left(I + \frac{\Delta t}{2}A\right)\mathbf{x}_{k} \right.\\
          \left.+ \frac{\Delta t}{2}B\bigl(\mathbf{u}_{k} + \mathbf{u}_{k+1}\bigr)
          + \Delta t\,(\mathbf{b} + \boldsymbol{\delta}_k)
        \right].
\end{split}
\end{equation}
where $A$ and $B$ are defined in~\eqref{eq:lti}, $\mathbf{b}$ is the nominal bias vector,
and $\boldsymbol{\delta}_{k} \in \mathbb{R}^{7}$ is the unknown additive
perturbation vector at step $k$, defined as
\begin{equation}
  \boldsymbol{\delta}_{k}
  \;=\;
  \bigl[\,\mathbf{0}_{3}^{\mathsf{T}},\;\;
        \boldsymbol{\delta}_{d,k}^{\mathsf{T}},\;\;
        \delta_{w,k}\,\bigr]^{\mathsf{T}},
  \label{eq:delta_def}
\end{equation}
in which $\boldsymbol{\delta}_{d,k} = \boldsymbol{\delta}_{g,k}
+ \boldsymbol{\delta}_{T,k} \in \mathbb{R}^{3}$ lumps the velocity-channel
perturbations from gravitational anomalies and thrust errors, while
$\delta_{w,k} \in \mathbb{R}$ captures the mass-flow measurement error.
The first three components of $\boldsymbol{\delta}_{k}$ are identically
zero because the position kinematics~\eqref{eq:P1_rdot} are exact and
unaffected by any of the three perturbation sources.


\subsection{Online Perturbation Estimation via RLS and EMA}
\label{subsec:rls}

At each guidance step $k$, the onboard navigation system delivers a noisy
measurement of the lander state:
\begin{equation}
  \mathbf{x}_{k}^{\mathrm{meas}}
  \;=\; \mathbf{x}_{k} + \boldsymbol{\eta}_{k},
  \qquad
  \boldsymbol{\eta}_{k}
  \;\overset{\mathrm{i.i.d.}}{\sim}\;
  \bigl(\mathbf{0},\, R_{\eta}\bigr),
  \label{eq:meas_model}
\end{equation}
where $\boldsymbol{\eta}_{k} \in \mathbb{R}^{7}$ is zero-mean navigation noise
with covariance $R_{\eta} = \mathrm{diag}(\sigma_{r}^{2}I_{3},\,
\sigma_{v}^{2}I_{3},\, \sigma_{w}^{2})$.  In the simulations, the noise
standard deviations are set to $\sigma_{r} = 0.5\ \mathrm{m}$,
$\sigma_{v} = 0.05\ \mathrm{m/s}$, and $\sigma_{w} = 0.001$.

The raw perturbation at step $k$ is estimated by comparing the actual
measured transition with the prediction of the nominal dynamics model.
Given the measured state $\mathbf{x}_{k}^{\mathrm{meas}}$, the commanded
control $\mathbf{u}_{k}$, and the previous measured state
$\mathbf{x}_{k-1}^{\mathrm{meas}}$, the nominal model predicts the next state as
\begin{equation}  \label{eq:pred}
\begin{split}
  \mathbf{x}_{k}^{\mathrm{pred}}
  \;=\; \left(I - \frac{\Delta t}{2}A\right)^{-1}
        \!\!\left[
          \left(I + \frac{\Delta t}{2}A\right)\mathbf{x}_{k-1}^{\mathrm{meas}} \right.\\
          \left.+ \frac{\Delta t}{2}B\bigl(\mathbf{u}_{k-1} + \mathbf{u}_{k}\bigr)
          + \Delta t\,\mathbf{b}
        \right].
\end{split}
\end{equation}
The raw perturbation estimate is then the scaled residual
\begin{equation}
  \hat{\boldsymbol{\delta}}_{k}^{\mathrm{raw}}
  \;=\; \frac{\mathbf{x}_{k}^{\mathrm{meas}}
              - \mathbf{x}_{k}^{\mathrm{pred}}}{\Delta t},
  \label{eq:raw_residual}
\end{equation}
where division by $\Delta t$ recovers the perturbation in units of
state per unit time, consistent with the additive bias interpretation
in~\eqref{eq:true_dyn}.  In the absence of perturbations and noise,
$\hat{\boldsymbol{\delta}}_{k}^{\mathrm{raw}} = \mathbf{0}$. Any nonzero
value directly reflects the combined effect of the true perturbation
$\boldsymbol{\delta}_{k-1}$ and the measurement noise $\boldsymbol{\eta}_{k}$.

A single residual measurement is noisy and may not be representative of the
current perturbation level.  A more reliable estimate is obtained by
recursively combining the current residual with the history of past
residuals, discounting older measurements to allow tracking of
time-varying perturbations.  This is achieved via a Recursive Least Squares
(RLS) filter with exponential forgetting factor $\lambda \in (0, 1)$ \cite{islam2019recursive}.

The RLS filter maintains a $7 \times 7$ covariance matrix
$P_{k} \in \mathbb{R}^{7 \times 7}$ and a perturbation estimate
$\hat{\boldsymbol{\delta}}_{k} \in \mathbb{R}^{7}$, initialized as
\begin{equation}
  \hat{\boldsymbol{\delta}}_{0} = \mathbf{0} \in \mathbb{R}^{7},
  \qquad
  P_{0} = p_{0}\,I_{7},
  \label{eq:rls_init}
\end{equation}
with $p_{0} = 10^{4}$, which encodes high initial uncertainty.
At each guidance step $k \geq 1$, the filter computes the
gain matrix, updates the estimate, and updates the covariance according to
\begin{align}
  K_{k}
  &\;=\; P_{k-1}\,\bigl(\lambda\,I_{7} + P_{k-1}\bigr)^{-1},
  \label{eq:rls_gain} \\[4pt]
  \hat{\boldsymbol{\delta}}_{k}
  &\;=\; \hat{\boldsymbol{\delta}}_{k-1}
         + K_{k}\,\bigl(\hat{\boldsymbol{\delta}}_{k}^{\mathrm{raw}}
                         - \hat{\boldsymbol{\delta}}_{k-1}\bigr),
  \label{eq:rls_update} \\[4pt]
  P_{k}
  &\;=\; \lambda^{-1}\,\bigl(I_{7} - K_{k}\bigr)\,P_{k-1}.
  \label{eq:rls_cov}
\end{align}
The forgetting factor $\lambda$ controls the balance between noise
rejection and tracking ability.  Values of $\lambda$ close to 1 assign nearly
equal weight to all past measurements and produce a smooth estimate that is
less responsive to rapid changes; values further from 1 discount older
measurements more aggressively, improving tracking of fast-varying
perturbations at the cost of increased noise sensitivity.

Before the RLS estimate is used to correct the dynamics model, an
Exponential Moving Average (EMA) smoother \cite{klinker2011exponential} is applied to further reduce
the influence of measurement noise on the trajectory optimizer.  The
EMA update at each step is
\begin{equation}
  \tilde{\boldsymbol{\delta}}_{k}
  \;=\; (1 - \alpha)\,\tilde{\boldsymbol{\delta}}_{k-1}
        + \alpha\,\hat{\boldsymbol{\delta}}_{k},
  \label{eq:ema}
\end{equation}
initialized with $\tilde{\boldsymbol{\delta}}_{0} = \mathbf{0}$, where
$\alpha$ is the EMA smoothing weight.  The EMA filter is a
first-order infinite impulse response filter with transfer function
$H(z) = \alpha z\,/\,(z - (1-\alpha))$, whose pole at $z = 1 - \alpha$
lies strictly inside the unit disk, guaranteeing bounded-input bounded-output
(BIBO) stability for any bounded input sequence.  The steady-state gain
$H(1) = 1$ ensures that the smoothed estimate $\tilde{\boldsymbol{\delta}}_{k}$
converges to the true perturbation in the absence of noise.

The EMA and RLS
filters run at every guidance step regardless of whether a trajectory
replan is triggered, ensuring that the perturbation estimate is as current
as possible when a replan does occur.

The complete online estimation procedure at each guidance step $k$ is
summarized in Algorithm~\ref{alg:rls}.

\begin{algorithm}[ht]
  \DontPrintSemicolon
  \SetAlgoLined
  \caption{Online Perturbation Estimation via RLS and EMA at Guidance Step $k$}
  \label{alg:rls}
  \KwIn{Previous measured state $\mathbf{x}_{k-1}^{\mathrm{meas}}$,
        commanded controls $\mathbf{u}_{k-1}$, $\mathbf{u}_{k}$,
        current RLS state $(\hat{\boldsymbol{\delta}}_{k-1},\, P_{k-1})$,
        current EMA state $\tilde{\boldsymbol{\delta}}_{k-1}$}
  \KwOut{Smoothed estimate $\tilde{\boldsymbol{\delta}}_{k}$,
         updated RLS state $(\hat{\boldsymbol{\delta}}_{k},\, P_{k})$}
  \BlankLine
  \quad \tcp{Step 1: Receive navigation measurement}
  \quad Obtain $\mathbf{x}_{k}^{\mathrm{meas}} = \mathbf{x}_{k} + \boldsymbol{\eta}_{k}$ from the navigation system\;
  \BlankLine
  \quad \tcp{Step 2: Compute nominal prediction and raw residual}
  \quad Compute $\mathbf{x}_{k}^{\mathrm{pred}}$ via \eqref{eq:pred}\;
  \quad Compute $\hat{\boldsymbol{\delta}}_{k}^{\mathrm{raw}}
           = (\mathbf{x}_{k}^{\mathrm{meas}}
              - \mathbf{x}_{k}^{\mathrm{pred}})\,/\,\Delta t$ via \eqref{eq:raw_residual}\;
  \BlankLine
  \quad \tcp{Step 3: RLS update}
  \quad $K_{k} \leftarrow P_{k-1}\,(\lambda\,I_{7} + P_{k-1})^{-1}$\;
  \quad $\hat{\boldsymbol{\delta}}_{k} \leftarrow
    \hat{\boldsymbol{\delta}}_{k-1}
    + K_{k}\,(\hat{\boldsymbol{\delta}}_{k}^{\mathrm{raw}}
               - \hat{\boldsymbol{\delta}}_{k-1})$\;
  \quad $P_{k} \leftarrow \lambda^{-1}(I_{7} - K_{k})\,P_{k-1}$\;
  \BlankLine
  \quad \tcp{Step 4: EMA smoothing}
  \quad $\tilde{\boldsymbol{\delta}}_{k} \leftarrow
    (1-\alpha)\,\tilde{\boldsymbol{\delta}}_{k-1}
    + \alpha\,\hat{\boldsymbol{\delta}}_{k}$\;
\end{algorithm}

\subsection{Augmented Linearized Dynamical Model}
\label{subsec:augmented_model}

Once the smoothed perturbation estimate $\tilde{\boldsymbol{\delta}}_{k}$
is available at guidance step $k$, it is used to correct the nominal
dynamics bias $\mathbf{b}$ before solving the next SOCP subproblem.

The corrected bias vector at step $k$ is defined as
\begin{equation}
  \tilde{\mathbf{b}}_{k}
  \;=\; \mathbf{b} + \tilde{\boldsymbol{\delta}}_{k},
  \label{eq:btilde}
\end{equation}
where $\mathbf{b} = [0,\,0,\,0,\,0,\,0,\,-g,\,0]^{\mathsf{T}}$ is the
nominal bias from~\eqref{eq:lti} and $\tilde{\boldsymbol{\delta}}_{k}$ is
the EMA-smoothed RLS estimate from~\eqref{eq:ema}.  Since the first three
components of $\boldsymbol{\delta}_{k}$ are zero by construction~\eqref{eq:delta_def},
only the velocity channels (components 4--6) and the log-mass channel
(component 7) of $\tilde{\mathbf{b}}_{k}$ differ from $\mathbf{b}$.
As the estimation error $\left\lVert\tilde{\boldsymbol{\delta}}_{k}
- \boldsymbol{\delta}_{k}\right\rVert$ decreases over time,
$\tilde{\mathbf{b}}_{k}$ converges toward the true bias
$\mathbf{b} + \boldsymbol{\delta}_{k}$.

The corrected bias $\tilde{\mathbf{b}}_{k}$ replaces the nominal $\mathbf{b}$
in the trapezoidal dynamics constraint~\eqref{eq:trap_dyn}.
The augmented dynamics constraint for the reduced-horizon SOCP subproblem
solved at guidance step $k$ is therefore
\begin{equation}  \label{eq:aug_trap_dyn}
\begin{split}
  \left(I - \frac{\Delta t_{k}}{2}A\right)\mathbf{x}_{i+1}
  - \left(I + \frac{\Delta t_{k}}{2}A\right)\mathbf{x}_{i}\\
  - \frac{\Delta t_{k}}{2}B\,(\mathbf{u}_{i} + \mathbf{u}_{i+1})
  - \Delta t_{k}\,\tilde{\mathbf{b}}_{k}
  \;=\; \mathbf{0},
\end{split}
\end{equation}
for $i = 1, \ldots, N_{k}-1$, where $\Delta t_{k} = (t_{f} - t_{k})\,/\,(N_{k} - 1)$ is the step size
of the reduced planning horizon at step $k$, and $N_{k}$
is the reduced-horizon node count used for online replans.
All other constraints in Problem~4, including the log-mass bounds~\eqref{eq:disc_w},
the successively linearized thrust-to-mass bound~\eqref{eq:disc_tau_ub},
and the SOC constraint~\eqref{eq:disc_soc}, remain structurally unchanged.
Only the right-hand side of the dynamics equality is modified through
$\tilde{\mathbf{b}}_{k}$.  This is the key architectural feature of the
proposed method: the online identification is decoupled from the convex
structure of the SOCP, so the subproblem retains all its convexity and
feasibility properties regardless of the estimated perturbation magnitude.

At every guidance step or every few guidance steps, a new SOCP subproblem
of the form of Problem~4 is constructed using the augmented bias
$\tilde{\mathbf{b}}_{k}$ and the current measured state
$\mathbf{x}_{k}^{\mathrm{meas}}$ as the initial condition.  The remaining
flight time $t_{f} - t_{k}$ is resampled to a number of
uniformly spaced nodes by linear interpolation of the current plan,
and the previous plan serves as a warm start.
Between consecutive replanning events, the RLS filter and EMA smoother
continue to run at every guidance step so that the perturbation estimate
remains current.  This frequency-gated replanning strategy ensures that the online scheme remains computationally tractable while still maintaining a sufficiently rapid response to the evolving perturbation estimate.

The augmented SOCP subproblem at guidance step $k$ is:

\noindent\textbf{Problem 5} (Online-Corrected Reduced-Horizon SOCP at step $k$):
\begin{equation*}
  \min_{\{\mathbf{x}_{i},\,\mathbf{u}_{i}\}_{i=1}^{N_{k}}}
  \quad
  J_{k} \;=\; \Delta t_{k}\sum_{i=1}^{N_{k}-1}\tau_{i}^{2}
\end{equation*}
\textit{subject to} \eqref{eq:aug_trap_dyn} and
\begin{alignat}{2}
  &\mathbf{x}_{1}\big|_{1:6}
   = \bigl[\mathbf{r}_{k}^{\mathsf{T}},\;\mathbf{v}_{k}^{\mathsf{T}}\bigr]^{\mathsf{T}},
   \quad x_{1,7} = w_{k}^{\mathrm{meas}},
  \label{eq:P5_ic}\\
  &\mathbf{x}_{N_{k}}\big|_{1:6}
   = \bigl[\mathbf{r}_{f}^{\mathsf{T}},\;\mathbf{v}_{f}^{\mathsf{T}}\bigr]^{\mathsf{T}},
  \label{eq:P5_tc}\\
  &\ln m_{\mathrm{dry}} \leq x_{i,7} \leq \ln m_{0},
   i = 1,\ldots,N_{k},\label{eq:P5_w}\\
  &0 \leq \tau_{i}
   \leq T_{\max}\,e^{-w_{i}^{*}}\!\bigl[1-(x_{i,7}-w_{i}^{*})\bigr],
   i = 1,\ldots,N_{k},\label{eq:P5_tau}\\
  &\left\lVert[\tau_{1,i},\,\tau_{2,i},\,\tau_{3,i}]^{\mathsf{T}}\right\rVert_{2}
   \leq \tau_{i},
   i = 1,\ldots,N_{k}.\label{eq:P5_soc}
\end{alignat}
where $\mathbf{r}_{k}$, $\mathbf{v}_{k}$, and $w_{k}^{\mathrm{meas}}$ are
the position, velocity, and log-mass components of the current measured state
$\mathbf{x}_{k}^{\mathrm{meas}}$.  Problem~5 has the same convex structure
as Problem~4 and is solved to global optimality using interior-point
solvers.  The solution $\{\mathbf{x}_{i}^{*},\mathbf{u}_{i}^{*}\}$
is interpolated back onto the full $N$-node time grid to produce the updated
planned trajectory for the subsequent guidance steps until the next replan.
The complete online SCP algorithm that integrates the identification and
replanning steps is presented in Section~\ref{sec:scp}.

\section{Sequential Convex Programming}
\label{sec:scp}

The convex relaxation developed in Section~\ref{sec:convex_relaxation} transforms
the nonconvex lunar landing trajectory optimization Problem~1 into the convex
subproblem Problem~4, which can be solved efficiently to global optimality
by an interior-point SOCP solver.  However, the successive linearization of
the thrust-to-mass constraint~\eqref{eq:disc_tau_ub} is performed around
a reference log-mass trajectory $\{w_{i}^{*}\}$ that must be updated
iteratively, since the linearization is exact only at the true solution.
SCP provides the iterative framework that
coordinates this update. SCP solves a sequence of convex subproblems, each
constructed around the solution of the previous iteration, until convergence
to a fixed point that satisfies the Karush--Kuhn--Tucker (KKT) conditions of the original
Problem~1.

In the nominal setting (no environmental uncertainty), a single SCP
run at the beginning of the mission suffices to compute the optimal trajectory.  In the online setting developed in Section~\ref{sec:model_identification}, however, the identified perturbation $\tilde{\boldsymbol{\delta}}_{k}$ modifies the effective bias
of the dynamics model at every guidance step $k$, so the SOCP subproblem
must be re-solved in flight to maintain trajectory optimality under the
updated model.  This section presents the online SCP algorithm with model
identification, which combines the perturbation estimation of
Section~\ref{sec:model_identification} with receding-horizon replanning to produce a closed-loop guidance scheme that adapts to unknown and time-varying disturbances during
the powered descent.

\subsection{Baseline SCP Algorithm}
\label{subsec:baseline_scp}

Before developing the online extension, the baseline SCP algorithm
is presented for reference, as its structure is inherited and extended by the
online version.

To start with, an initial reference trajectory
$\{\mathbf{x}_{i}^{(0)},\,\mathbf{u}_{i}^{(0)}\}_{i=1}^{N}$
is required to construct the first linearization of the thrust-to-mass
constraint.  Formally, any trajectory satisfying the boundary conditions
\eqref{eq:disc_ic}--\eqref{eq:disc_tc} and the log-mass
bounds~\eqref{eq:disc_w} may serve as a valid initialization.
In practice, the straight-line interpolation between the boundary conditions
can be used, defined at each node $i$ as
\begin{equation}
  \mathbf{x}_{i}^{(0)}
  \;=\;
  \begin{bmatrix}
    \dfrac{(N-i)\,\mathbf{r}_{0} + (i-1)\,\mathbf{r}_{f}}{N-1} \\[6pt]
    \dfrac{(N-i)\,\mathbf{v}_{0} + (i-1)\,\mathbf{v}_{f}}{N-1} \\[6pt]
    \ln m_{0}
  \end{bmatrix},
  i = 1, \ldots, N,
  \label{eq:init_traj}
\end{equation}
and the initial control is set to a constant hover-equivalent thrust aligned
with the positive $z$-axis: $\mathbf{u}_{i}^{(0)} = [0,\,0,\,g,\,g]^{\mathsf{T}}$
for all $i$.  Despite the crude nature of this initialization (the
straight-line guess may violate the dynamics constraints significantly), the
SCP algorithm converges reliably within a small number of iterations, because
the only non-equivalent approximation in the problem (the successive
linearization of $T_{\max}e^{-w}$) is mild when the log-mass variation is
small.  The iteration counter is set to $\ell = 0$.

At each SCP iteration $\ell \geq 0$, the current reference log-mass sequence
$\{w_{i}^{(\ell)}\}_{i=1}^{N}$ is extracted from $\{\mathbf{x}_{i}^{(\ell)}\}$
as the seventh component.  A Problem~4$^{(\ell)}$ is then constructed
and solved with the successively linearized thrust-to-mass bound
\begin{equation}
  0 \;\leq\; \tau_{i}
  \;\leq\;
  T_{\max}\,e^{-w_{i}^{(\ell)}}\!\left[1 - \bigl(x_{i,7} - w_{i}^{(\ell)}\bigr)\right],
  \label{eq:base_tau_ub}
\end{equation}
for $i = 1, \ldots, N$, where $x_{i,7}$ denotes the seventh (log-mass) component of the decision variable $\mathbf{x}_{i}$.  The subproblem is solved to global optimality
by an interior-point SOCP solver, yielding the optimal solution
$\{\mathbf{x}_{i}^{*(\ell)},\,\mathbf{u}_{i}^{*(\ell)}\}_{i=1}^{N}$.

Then, the reference trajectory is updated to the solution of the current
subproblem:
\begin{equation}
  \bigl\{\mathbf{x}_{i}^{(\ell+1)},\,\mathbf{u}_{i}^{(\ell+1)}\bigr\}
  \;\leftarrow\;
  \bigl\{\mathbf{x}_{i}^{*(\ell)},\,\mathbf{u}_{i}^{*(\ell)}\bigr\}.
  \label{eq:scp_update}
\end{equation}
Convergence is declared when the maximum componentwise change in the state
trajectory between successive iterates falls below a prescribed tolerance
$\boldsymbol{\varepsilon} \in \mathbb{R}^{7}$:
\begin{equation}
  \max_{i = 1,\ldots,N}
  \left\lVert
    \mathbf{x}_{i}^{(\ell+1)} - \mathbf{x}_{i}^{(\ell)}
  \right\rVert_{\infty}
  \;<\; \varepsilon_{j},
  \quad j = 1, \ldots, 7,
  \label{eq:conv_crit}
\end{equation}
where $\varepsilon_{j}$ is a prescribed tolerance for each state channel.  If the
criterion~\eqref{eq:conv_crit} is not satisfied, the iteration counter is
incremented to $\ell + 1$ and the process returns to
constructing and solving a new Problem~4$^{(\ell)}$.  Upon convergence, the solution
$\{\mathbf{x}_{i}^{*},\,\mathbf{u}_{i}^{*}\}
= \{\mathbf{x}_{i}^{(\ell+1)},\,\mathbf{u}_{i}^{(\ell+1)}\}$
is accepted as an approximate solution to Problem~1.
The baseline SCP algorithm is summarized in Algorithm~\ref{alg:baseline_scp}.

\begin{algorithm}[ht]
  \DontPrintSemicolon
  \SetAlgoLined
  \caption{Baseline SCP Algorithm}
  \label{alg:baseline_scp}
  \KwIn{System matrices $A$, $B$, $\mathbf{b}$; boundary conditions
        $\mathbf{r}_{0}$, $\mathbf{v}_{0}$, $m_{0}$, $\mathbf{r}_{f}$,
        $\mathbf{v}_{f}$; parameters $N$, $\Delta t$, $T_{\max}$,
        $m_{\mathrm{dry}}$; tolerance $\boldsymbol{\varepsilon}$;
        maximum iterations $\ell_{\max}$}
  \KwOut{Optimal state sequence $\{\mathbf{x}_{i}^{*}\}_{i=1}^{N}$ and
         control sequence $\{\mathbf{u}_{i}^{*}\}_{i=1}^{N}$}
  \BlankLine
  \quad \tcp{Step 1: Initialization}
  \quad Construct $\{\mathbf{x}_{i}^{(0)},\,\mathbf{u}_{i}^{(0)}\}_{i=1}^{N}$
  via \eqref{eq:init_traj}\;
  \quad Set $\ell \leftarrow 0$\;
  \BlankLine
  \Repeat{convergence criterion \eqref{eq:conv_crit} satisfied
          \textbf{or} $\ell = \ell_{\max}$}{
    \quad \tcp{Step 2: Extract linearization point and build subproblem}
    \quad Set $w_{i}^{(\ell)} \leftarrow x_{i,7}^{(\ell)}$ for $i = 1, \ldots, N$\;
    \quad Formulate Problem~4$^{(\ell)}$ using constraint~\eqref{eq:base_tau_ub}\;
    \BlankLine
    \quad \tcp{Step 3: Solve convex subproblem}
    \quad Solve Problem~4$^{(\ell)}$ to obtain
    $\{\mathbf{x}_{i}^{*(\ell)},\,\mathbf{u}_{i}^{*(\ell)}\}$\;
    \BlankLine
    \quad \tcp{Step 4: Update reference trajectory}
    \quad $\{\mathbf{x}_{i}^{(\ell+1)},\,\mathbf{u}_{i}^{(\ell+1)}\}
     \leftarrow
     \{\mathbf{x}_{i}^{*(\ell)},\,\mathbf{u}_{i}^{*(\ell)}\}$\;
    \quad $\ell \leftarrow \ell + 1$\;
  }
  \BlankLine
  \KwRet{$\{\mathbf{x}_{i}^{*},\,\mathbf{u}_{i}^{*}\}
          \leftarrow
          \{\mathbf{x}_{i}^{(\ell)},\,\mathbf{u}_{i}^{(\ell)}\}$}
\end{algorithm}

\subsection{Online SCP Algorithm with Model Identification}
\label{subsec:online_scp}

The online extension of the baseline SCP algorithm integrates the
real-time perturbation estimation of Section~\ref{sec:model_identification} into a
receding-horizon replanning loop.  The central modification is the
replacement of the nominal dynamics bias $\mathbf{b}$ by the
online-corrected bias $\tilde{\mathbf{b}}_{k} = \mathbf{b}
+ \tilde{\boldsymbol{\delta}}_{k}$ in every SOCP subproblem solved
during flight.  This makes the trajectory plan consistent with the
best available model of the true flight environment at each guidance
cycle.

The algorithm proceeds through four stages at every guidance step $k$,
which are described in detail below and consolidated in
Algorithm~\ref{alg:online_scp}.

\subsubsection{Stage 1: Initialization}
\label{subsubsec:online_init}

Before the descent begins, the baseline SCP (Algorithm~\ref{alg:baseline_scp})
is executed using the nominal bias $\mathbf{b}$ to produce an
initial nominal plan
$\{{\mathbf{x}}_{i}^{\mathrm{nom}},\,{\mathbf{u}}_{i}^{\mathrm{nom}}\}_{i=1}^{N}$.
This plan serves as the starting planned
trajectory:
\begin{equation}
  \mathbf{X}^{\mathrm{plan}} \leftarrow \mathbf{X}^{\mathrm{nom}},
  \qquad
  \mathbf{U}^{\mathrm{plan}} \leftarrow \mathbf{U}^{\mathrm{nom}}.
  \label{eq:online_init}
\end{equation}
The RLS filter state and EMA smoother are initialized as
\begin{equation}
  \hat{\boldsymbol{\delta}}_{0} = \mathbf{0},
  \qquad
  P_{0} = p_{0}\,I_{7},
  \qquad
  \tilde{\boldsymbol{\delta}}_{0} = \mathbf{0},
  \label{eq:online_filter_init}
\end{equation}
with $p_{0} = 10^{4}$.  The guidance step counter is set to $k = 1$.

\subsubsection{Stage 2: State Measurement and Perturbation Estimation}
\label{subsubsec:online_meas}

At guidance step $k$, the planned control $\mathbf{u}_{k}^{\mathrm{plan}}$
is applied to the lander.  The true state evolves according to the
perturbed dynamics~\eqref{eq:true_dyn} with the unknown perturbation
$\boldsymbol{\delta}_{k}$.  The navigation system then delivers the
noisy measurement $\mathbf{x}_{k+1}^{\mathrm{meas}}$ according to the
model~\eqref{eq:meas_model}.

Using $\mathbf{x}_{k+1}^{\mathrm{meas}}$, $\mathbf{u}_{k}^{\mathrm{plan}}$,
and $\mathbf{u}_{k+1}^{\mathrm{plan}}$, Algorithm~\ref{alg:rls} of
Section~\ref{sec:model_identification} is executed to update the perturbation estimate.
The outputs are the updated RLS estimate
$\hat{\boldsymbol{\delta}}_{k+1}$, covariance $P_{k+1}$, and
EMA-smoothed estimate $\tilde{\boldsymbol{\delta}}_{k+1}$.
The corrected bias is then formed as
\begin{equation}
  \tilde{\mathbf{b}}_{k+1}
  \;=\; \mathbf{b} + \tilde{\boldsymbol{\delta}}_{k+1}.
  \label{eq:btilde_online}
\end{equation}
This estimation stage runs at every guidance step $k$, regardless
of whether a trajectory replan is triggered.

\subsubsection{Stage 3: Frequency-Gated Trajectory Replanning}
\label{subsubsec:online_replan}

A trajectory replan is triggered at guidance step $k$ if and only if
$k$ is a multiple of the replanning cadence $n_{\mathrm{replan}}$.
This frequency-gating strategy reduces the total number of SOCP solves
from $N - 1$ (one per step) to approximately
$(N - 1)/n_{\mathrm{replan}}$ over the full descent, while
the RLS filter continues to accumulate measurements between replans,
so the perturbation estimate is well-converged by the time each
replan is executed.

When a replan is triggered, a reduced-horizon SOCP of the form of
Problem~5 is constructed and solved as follows.

\textit{Horizon construction.}
The remaining flight time at step $k+1$ is $t_{\mathrm{rem}} = t_{f}
- t_{k+1}$.  This is partitioned into $N_{\mathrm{plan}}$
uniformly spaced planning nodes with step size
\begin{equation}
  \Delta t_{k} \;=\; \frac{t_{\mathrm{rem}}}{N_{\mathrm{plan}} - 1},
  \label{eq:dt_k}
\end{equation}
producing the reduced planning time grid
$\{t_{k+1},\, t_{k+1} + \Delta t_{k},\,\ldots,\, t_{f}\}$.
A fixed node count $N_{\mathrm{plan}}$ can be used regardless of how many
steps remain to ensure that every online SOCP has the same size
with the same number of decision variables, so the
solver execution time is bounded and predictable across all replanning
events.

\textit{Warm start.}
The current full-horizon plan $\{\mathbf{X}^{\mathrm{plan}},
\mathbf{U}^{\mathrm{plan}}\}$ is resampled onto the reduced
planning grid $\{\mathbf{X}^{\mathrm{warm}},
\mathbf{U}^{\mathrm{warm}}\}$ by linear interpolation.
The first node of $\mathbf{X}^{\mathrm{warm}}$ is then overridden by the
current measured state to ensure the replanned trajectory originates from
the true current position:
\begin{equation}
  \mathbf{X}^{\mathrm{warm}}_{1} \;\leftarrow\; \mathbf{x}_{k+1}^{\mathrm{meas}}.
  \label{eq:warm_anchor}
\end{equation}
Using the previous plan as a warm start exploits the fact that
consecutive replanning problems differ only slightly from one another,
so the warm-started solver typically converges in very few inner SCP iterations.

\textit{Inner SCP loop.}
Starting from the warm-started reference trajectory
$\{\mathbf{X}^{(0)}, \mathbf{U}^{(0)}\} = \{\mathbf{X}^{\mathrm{warm}},
\mathbf{U}^{\mathrm{warm}}\}$, the inner SCP loop solves a sequence
of convex subproblems of the form of Problem~5, each constructed
using the augmented bias $\tilde{\mathbf{b}}_{k+1}$ and the
linearization point $\{w_{i}^{(\ell)}\}$ from the previous inner
iteration.  At inner iteration $\ell$, the subproblem is:

\medskip
\noindent\textbf{Problem 5$^{(\ell)}$}
(Online-Corrected Subproblem at Inner Iteration $\ell$ of Guidance Step $k$):
\begin{equation}
  \min_{\{\mathbf{x}_{i},\,\mathbf{u}_{i}\}_{i=1}^{N_{\mathrm{plan}}}}
  \quad
  J_{k}^{(\ell)}
  \;=\; \Delta t_{k} \sum_{i=1}^{N_{\mathrm{plan}}-1} \tau_{i}^{2}
  \label{eq:online_obj}
\end{equation}
\textit{subject to}
\begin{align}
    \begin{split}
    &\left(I - \frac{\Delta t_{k}}{2}A\right)\mathbf{x}_{i+1}
   - \left(I + \frac{\Delta t_{k}}{2}A\right)\mathbf{x}_{i} \\
    &\quad - \frac{\Delta t_{k}}{2}B(\mathbf{u}_{i} + \mathbf{u}_{i+1})
   - \Delta t_{k}\,\tilde{\mathbf{b}}_{k+1}
   = \mathbf{0}, \\
    &\quad i = 1,\ldots,N_{\mathrm{plan}}-1,
    \end{split} \label{eq:online_dyn} \\
  &\mathbf{x}_{1}\big|_{1:6}
   = \bigl[\mathbf{r}_{k+1}^{\mathsf{T}},\;\mathbf{v}_{k+1}^{\mathsf{T}}\bigr]^{\mathsf{T}},
   \quad x_{1,7} = w_{k+1}^{\mathrm{meas}},
  \label{eq:online_ic}\\
  &\mathbf{x}_{N_{\mathrm{plan}}}\big|_{1:6}
   = \bigl[\mathbf{r}_{f}^{\mathsf{T}},\;\mathbf{v}_{f}^{\mathsf{T}}\bigr]^{\mathsf{T}},
  \label{eq:online_tc}\\
  &\ln m_{\mathrm{dry}} \;\leq\; x_{i,7} \;\leq\; \ln m_{0},
  \quad i = 1,\ldots,N_{\mathrm{plan}},
  \label{eq:online_w}\\
    \begin{split}
    &0 \;\leq\; \tau_{i}
   \;\leq\; T_{\max}\,e^{-w_{i}^{(\ell)}}
             \!\left[1 - \bigl(x_{i,7} - w_{i}^{(\ell)}\bigr)\right], \\
    &\quad i = 1,\ldots,N_{\mathrm{plan}},
    \end{split} \label{eq:online_tau}\\
  &\left\lVert[\tau_{1,i},\,\tau_{2,i},\,\tau_{3,i}]^{\mathsf{T}}\right\rVert_{2}
   \;\leq\; \tau_{i},
  \quad i = 1,\ldots,N_{\mathrm{plan}}.
  \label{eq:online_soc}
\end{align}
Problem~5$^{(\ell)}$ is solved to obtain
$\{\mathbf{x}_{i}^{*(\ell)},\,\mathbf{u}_{i}^{*(\ell)}\}$.
The inner SCP update and convergence check follow the same logic as
the baseline: the reference is updated to
$\{\mathbf{x}_{i}^{(\ell+1)},\,\mathbf{u}_{i}^{(\ell+1)}\}
\leftarrow \{\mathbf{x}_{i}^{*(\ell)},\,\mathbf{u}_{i}^{*(\ell)}\}$
and convergence is checked against the same criterion~\eqref{eq:conv_crit}.
The inner loop terminates after a maximum number of
iterations or upon satisfaction of~\eqref{eq:conv_crit}, whichever
occurs first.

The key difference between Problem~5$^{(\ell)}$ and the baseline
Problem~4$^{(\ell)}$ is the appearance of $\tilde{\mathbf{b}}_{k+1}$
in place of $\mathbf{b}$ in the dynamics constraint~\eqref{eq:online_dyn},
and the use of the reduced step size $\Delta t_{k}$ and node count
$N_{\mathrm{plan}}$.  The convex structure of the subproblem is identical to that of Problem~4$^{(\ell)}$ and is
preserved unconditionally regardless of the value of
$\tilde{\mathbf{b}}_{k+1}$.

\textit{Plan update.}
Upon completion of the inner SCP loop, the converged reduced-horizon plan
$\{\mathbf{x}_{i}^{*},\,\mathbf{u}_{i}^{*}\}_{i=1}^{N_{\mathrm{plan}}}$
is interpolated back onto the original $N$-node full-horizon time grid
by linear interpolation over the remaining nodes $k+1, \ldots, N$.
This ensures that the full-horizon plan remains consistent with the
latest replan for all future guidance steps until the next replanning
event.

\subsubsection{Stage 4: Guidance Step Advance and Termination}
\label{subsubsec:online_advance}

After Stage~3 (or after the estimation update in Stage~2 if no replan
is triggered), the guidance step counter is advanced to $k + 1$ and
the algorithm returns to Stage~2.  The descent terminates when either
$k = N - 1$ (the final planned node is reached) or the horizon is too short to support a further replan (e.g., the remaining
flight time $t_{\mathrm{rem}} < 2\Delta t$), at which point the last available planned
control is applied and the lander executes touchdown.

\subsection{Complete Online SCP Algorithm}
\label{subsec:online_scp_alg}

The four stages described above are consolidated in
Algorithm~\ref{alg:online_scp}.
It is not difficult to identify the differences between the
baseline SCP (Algorithm~\ref{alg:baseline_scp}) and the online SCP
(Algorithm~\ref{alg:online_scp}), since the two algorithms share
the same convex subproblem structure and differ only in how the
dynamics model is specified.
However, the online algorithm introduces three elements absent from the baseline: (i) a closed-loop measurement--estimation--correction
cycle that runs at every guidance step; (ii) a receding-horizon
architecture in which the initial condition of the SOCP is updated
to the current measured state at each replanning event; and
(iii) a modified dynamics constraint in which the nominal bias is replaced by the adaptively corrected bias.  All other aspects (the decision variables, the objective function, the
SOC and log-mass constraints, and the convergence criterion) of the two algorithms are identical.


The most important structural observation is that replacing $\mathbf{b}$
with $\tilde{\mathbf{b}}_{k}$ does not alter the convex structure of
the SOCP in any way.  The corrected bias appears exclusively as a
constant right-hand-side term in the linear equality
constraint~\eqref{eq:online_dyn}; it does not appear in the objective,
the SOC constraint, or the log-mass or thrust-magnitude bounds.
Consequently, the feasibility, convexity, and global optimality
properties of the SOCP subproblem are preserved regardless of the
magnitude or direction of $\tilde{\boldsymbol{\delta}}_{k}$.  This
decoupling between the identification layer and the optimization layer
is the key architectural property that enables the online algorithm
to inherit all theoretical guarantees of the baseline SCP.

\begin{algorithm}[ht]
  \DontPrintSemicolon
  \SetAlgoLined
  \caption{Online SCP Algorithm with Model Identification}
  \label{alg:online_scp}
  \KwIn{Nominal plan $\{\mathbf{X}^{\mathrm{nom}}, \mathbf{U}^{\mathrm{nom}}\}$;
        system matrices $A$, $B$, $\mathbf{b}$;
        filter parameters $\lambda$, $p_{0}$, $\alpha$;
        replanning cadence $n_{\mathrm{replan}}$;
        reduced horizon $N_{\mathrm{plan}}$;
        inner iteration limit $\ell_{\max}^{\mathrm{on}}$;
        convergence tolerance $\boldsymbol{\varepsilon}$}
  \KwOut{True state sequence $\{\mathbf{x}_{k}^{\mathrm{true}}\}_{k=1}^{N}$;
         planned control sequence $\{\mathbf{U}^{\mathrm{plan}}\}$}
  \BlankLine
  \quad \tcp{Stage 1: Initialization}
  \quad $\mathbf{X}^{\mathrm{plan}} \leftarrow \mathbf{X}^{\mathrm{nom}}$,\;
  \quad $\mathbf{U}^{\mathrm{plan}} \leftarrow \mathbf{U}^{\mathrm{nom}}$,\;
  \quad $\hat{\boldsymbol{\delta}}_{0} \leftarrow \mathbf{0}$,\;
  \quad $P_{0} \leftarrow p_{0} I_{7}$,\;
  \quad $\tilde{\boldsymbol{\delta}}_{0} \leftarrow \mathbf{0}$\;
  \quad $\mathbf{x}_{1}^{\mathrm{true}} \leftarrow [\mathbf{r}_{0}^{\mathsf{T}},\,
   \mathbf{v}_{0}^{\mathsf{T}},\, \ln m_{0}]^{\mathsf{T}}$\;
  \BlankLine
  \For{$k = 1, 2, \ldots, N-1$}{
    \quad \tcp{Stage 2: State measurement and perturbation estimation}
    \quad Apply $\mathbf{u}_{k}^{\mathrm{plan}}$ to the lander\;
    \quad True state advances via \eqref{eq:true_dyn} \;
    \quad Receive measurement: $\mathbf{x}_{k+1}^{\mathrm{meas}}
      \leftarrow \mathbf{x}_{k+1}^{\mathrm{true}} + \boldsymbol{\eta}_{k}$ via \eqref{eq:meas_model}\;
    \quad Execute Algorithm~\ref{alg:rls} to obtain
    $\tilde{\boldsymbol{\delta}}_{k+1}$,
    $\hat{\boldsymbol{\delta}}_{k+1}$, $P_{k+1}$\;
    \quad Form corrected bias:
    $\tilde{\mathbf{b}}_{k+1} \leftarrow \mathbf{b}
     + \tilde{\boldsymbol{\delta}}_{k+1}$\;
    \BlankLine
    \quad \tcp{Stage 3: Frequency-gated trajectory replanning}
    \If{$k \bmod n_{\mathrm{replan}} = 0$
        \textbf{and} $t_{f} - t_{k+1} \geq 2\Delta t$}{
      \quad Compute $\Delta t_{k}$, $N_{\mathrm{plan}}$, planning grid
      via \eqref{eq:dt_k}\;
      \quad Interpolate warm start
      $\{\mathbf{X}^{\mathrm{warm}}, \mathbf{U}^{\mathrm{warm}}\}$\;
      \quad Override $\mathbf{X}^{\mathrm{warm}}_{1} \leftarrow
        \mathbf{x}_{k+1}^{\mathrm{meas}}$ via \eqref{eq:warm_anchor}\;
      \quad Set $\ell \leftarrow 0$,
      $\{\mathbf{x}_{i}^{(0)},\mathbf{u}_{i}^{(0)}\}
       \leftarrow \{\mathbf{X}^{\mathrm{warm}}, \mathbf{U}^{\mathrm{warm}}\}$\;
      \Repeat{criterion \eqref{eq:conv_crit} satisfied
              \textbf{or} $\ell = \ell_{\max}^{\mathrm{on}}$}{
        \quad $w_{i}^{(\ell)} \leftarrow x_{i,7}^{(\ell)}$,
        $i = 1,\ldots,N_{\mathrm{plan}}$\;
        \quad Solve Problem~5$^{(\ell)}$ $\Rightarrow$
        $\{\mathbf{x}_{i}^{*(\ell)}, \mathbf{u}_{i}^{*(\ell)}\}$\;
        \quad $\{\mathbf{x}_{i}^{(\ell+1)}, \mathbf{u}_{i}^{(\ell+1)}\}
         \leftarrow \{\mathbf{x}_{i}^{*(\ell)}, \mathbf{u}_{i}^{*(\ell)}\}$\;
        \quad $\ell \leftarrow \ell + 1$\;
      }
      \quad Update full-horizon plan through linear interpolation\;
    }
    \BlankLine
    \quad \tcp{Stage 4: Advance guidance step}
  }
\end{algorithm}

\section{Theoretical Analysis}
\label{sec:theoretical_analysis}

This section explores some theoretical foundations underlying the convex
relaxation of Section~\ref{sec:convex_relaxation}, the online model identification
of Section~\ref{sec:model_identification}, and the SCP algorithms of
Section~\ref{sec:scp}. Specifically, five main results are developed: the losslessness of the
SOC relaxation; the
stability and convergence of the RLS
and EMA filters used for online perturbation
estimation; the guaranteed convergence of the SCP iteration to a
KKT point of the original nonconvex problem; the
radius of convergence within which this guarantee holds; and the convergence
rate at which the SCP iterates approach the optimal solution.
These analyses are motivated by the results in \cite{acikmese2007convex,mao2018successive,haykin2008adaptive}, and customized for the lunar landing problem in this paper.

\subsection{Lossless Convex Relaxation}
\label{subsec:lossless}

The convex relaxation introduced in Section~\ref{sec:convex_relaxation} replaces
the nonconvex unit-norm equality constraint~\eqref{eq:cone_surface} with the
convex SOC inequality~\eqref{eq:soc}.  I will show that the relaxation introduces no optimality gap, i.e., any optimal solution
of the relaxed Problem~3 automatically satisfies the original equality
constraint, so the two problems share the same optimal value.

\begin{assumption}
\label{assump:zero_lower_thrust}
The minimum thrust magnitude is zero, that is, $T_{\min} = 0$ in the thrust
bound~\eqref{eq:thrust_bound}, so the engine may be commanded off at any
instant.
\end{assumption}

\begin{assumption}
\label{assump:control_affine}
The velocity dynamics are affine in the decomposed thrust vector
$[\tau_{1},\,\tau_{2},\,\tau_{3}]^{\mathsf{T}}$, as given
in~\eqref{eq:vxdot}--\eqref{eq:vzdot}, with no additional dependence of the
dynamics on the thrust direction beyond this affine term.
\end{assumption}

\begin{assumption}
\label{assump:direction_free_cost}
The objective functional~\eqref{eq:obj_new} depends on the control only
through the thrust magnitude $\tau$, and not on the thrust direction
$\hat{\bm{i}}_{\theta}$.
\end{assumption}

\begin{assumption}
\label{assump:relaxed_feasible}
Problem~3 admits at least one feasible trajectory.
\end{assumption}

\begin{lemma}
\label{lemma:direction_recovery}
Suppose $(\tau_{1}^{*}, \tau_{2}^{*}, \tau_{3}^{*}, \tau^{*})$ is feasible for
the relaxed constraint~\eqref{eq:soc} and the inequality is strict, i.e.,
$(\tau_{1}^{*})^{2} + (\tau_{2}^{*})^{2} + (\tau_{3}^{*})^{2} < (\tau^{*})^{2}$,
on a sub-interval within $[t_{0}, t_{f}]$.  Then there exists an
alternative thrust decomposition
$(\tilde{\tau}_{1}, \tilde{\tau}_{2}, \tilde{\tau}_{3}, \tau^{*})$, sharing
the same thrust magnitude $\tau^{*}$ and the same objective value, that
satisfies the original equality constraint~\eqref{eq:cone_surface}
everywhere on $[t_{0}, t_{f}]$.
\end{lemma}

\begin{proof}
Let $S \subseteq [t_{0}, t_{f}]$ denote the set on which the SOC
inequality is strict.  At every $t \in S$, the magnitude $\tau^{*}(t)$ must
be strictly positive, since otherwise the inequality
$(\tau_{1}^{*})^{2} + (\tau_{2}^{*})^{2} + (\tau_{3}^{*})^{2} \le (\tau^{*})^{2} = 0$
forces $\tau_{1}^{*} = \tau_{2}^{*} = \tau_{3}^{*} = 0$, at which point
equality already holds, contradicting $t \in S$.  For $t \in S$, define the
rescaled components
\begin{equation}
  \tilde{\tau}_{j}(t)
  \;=\; \tau^{*}(t)\,
        \frac{\tau_{j}^{*}(t)}
             {\left\lVert[\tau_{1}^{*}(t),\,\tau_{2}^{*}(t),\,\tau_{3}^{*}(t)]^{\mathsf{T}}\right\rVert_{2}},
  j = 1, 2, 3,
  \label{eq:rescale}
\end{equation}
and set $\tilde{\tau}_{j}(t) = \tau_{j}^{*}(t)$ for $t \notin S$.  By
construction, $\tilde{\tau}_{1}^{2} + \tilde{\tau}_{2}^{2} + \tilde{\tau}_{3}^{2}
= (\tau^{*})^{2}$ holds at every $t \in [t_{0}, t_{f}]$, so the equality
constraint~\eqref{eq:cone_surface} is satisfied exactly.  The
rescaling~\eqref{eq:rescale} preserves the direction of the original thrust
vector, $\tilde{\tau}_{j}/\tau^{*} = \tau_{j}^{*}/\lVert[\tau_{1}^{*},\tau_{2}^{*},\tau_{3}^{*}]^{\mathsf{T}}\rVert_{2}$,
so the unit thrust-direction vector is unchanged.  Because the objective
functional depends only on $\tau$ by Assumption~\ref{assump:direction_free_cost},
and $\tau^{*}$ is unchanged, the objective value is identical for the two
solutions.  This establishes the claim.
\end{proof}

\begin{theorem}
\label{thm:lossless}
Under Assumptions~\ref{assump:zero_lower_thrust}--\ref{assump:relaxed_feasible},
the optimal value of the relaxed Problem~3 equals the optimal value of the
original Problem~2.  Moreover, every optimal solution of Problem~3 satisfies
the equality constraint~\eqref{eq:cone_surface} pointwise, so no further
recovery step of the kind described in Lemma~\ref{lemma:direction_recovery}
is in fact necessary at the optimum.
\end{theorem}

\begin{proof}
The proof proceeds via Pontryagin's minimum principle applied to the
relaxed problem.  Adjoin the SOC inequality~\eqref{eq:soc} to the Hamiltonian
with a nonnegative multiplier $\mu(t) \geq 0$:
\begin{equation}  \label{eq:hamiltonian_relaxed}
\begin{split}
  H \;=\; &\lambda_{r}^{\mathsf{T}}\mathbf{v}
        + \lambda_{v}^{\mathsf{T}}
          \bigl([\tau_{1},\,\tau_{2},\,\tau_{3}]^{\mathsf{T}} + \mathbf{g}\bigr)
        - \lambda_{w}\,\frac{\tau}{I_{\mathrm{sp}}\,g_{0}}\\
        &+ \tau^{2}
        - \mu\,\bigl(\tau^{2} - \tau_{1}^{2} - \tau_{2}^{2} - \tau_{3}^{2}\bigr),
\end{split}
\end{equation}
where $\lambda_{r}$, $\lambda_{v}$, and $\lambda_{w}$ are the costate
variables associated with position, velocity, and log-mass, respectively.

Minimizing $H$ over the thrust components $(\tau_{1}, \tau_{2}, \tau_{3})$ with $\tau$ held fixed
gives the first-order condition
\begin{equation}\label{eq:tauj_star}
\begin{split}
  \frac{\partial H}{\partial \tau_{j}}
  \;&=\; \lambda_{v,j} + 2\mu\,\tau_{j} \;=\; 0\\
  \quad\Longrightarrow\quad
  \tau_{j}^{*} \;&=\; -\,\frac{\lambda_{v,j}}{2\mu},
  \qquad j = 1, 2, 3.
\end{split}
\end{equation}
If $\mu = 0$, then $\partial H/\partial \tau_{j} = \lambda_{v,j}$ is constant
in $\tau_{j}$ and the Hamiltonian is unbounded below as $\tau_{j} \to \pm\infty$
whenever $\lambda_{v,j} \neq 0$, contradicting the existence of a bounded
optimal control.  Hence $\mu > 0$ at every optimal solution.  By
complementary slackness, $\mu\,(\tau^{2} - \tau_{1}^{2} - \tau_{2}^{2}
- \tau_{3}^{2}) = 0$, and since $\mu > 0$ this forces
$\tau_{1}^{2} + \tau_{2}^{2} + \tau_{3}^{2} = \tau^{2}$ identically.  The SOC
constraint is therefore tight at every optimal solution, establishing that
the original equality~\eqref{eq:cone_surface} holds automatically.

With the optimal direction $\hat{\bm{i}}_{\theta}^{*}
= -\lambda_{v}/\lVert\lambda_{v}\rVert_{2}$ recovered from~\eqref{eq:tauj_star},
minimizing $H$ over $\tau$ yields a scalar quadratic whose unconstrained
minimizer is
\begin{equation}
  \tau^{*}
  \;=\; \frac{1}{2}\left(
          \lambda_{v}^{\mathsf{T}}\hat{\bm{i}}_{\theta}^{*}
          + \frac{\lambda_{w}}{I_{\mathrm{sp}}\,g_{0}}
        \right),
  \label{eq:tau_star_pmp}
\end{equation}
clamped to $[0, T_{\max}\,e^{-w}]$ by the (linearized) thrust bound.

Since we have shown that every optimal solution of the relaxed problem
satisfies the original equality constraint, every optimal solution of
Problem~3 is feasible for Problem~2, and conversely every feasible solution
of Problem~2 is feasible for Problem~3 (since the equality
constraint~\eqref{eq:cone_surface} implies the inequality~\eqref{eq:soc}).
The two feasible sets therefore have the same set of optimal solutions, and
the optimal objective values coincide.
\end{proof}

\begin{remark}
\label{remark:lossless_failure}
The losslessness result of Theorem~\ref{thm:lossless} relies critically on
Assumption~\ref{assump:zero_lower_thrust}.  If a strictly positive minimum
thrust $T_{\min} > 0$ were imposed, the multiplier analysis above would
need to account for the possibility that the lower thrust bound, rather than
the SOC constraint, is active, and the relaxation could in principle leave a
nonzero gap between the relaxed and original optimal values.  Similarly, if
additional constraints were imposed directly on the thrust direction
$\hat{\bm{i}}_{\theta}$ (e.g., a maximum gimbal angle relative to the
body axis) or if the objective depended explicitly on the thrust direction
(violating Assumption~\ref{assump:direction_free_cost}), the optimal
direction in~\eqref{eq:tauj_star} would no longer be guaranteed to saturate
the SOC constraint, and the relaxation would require separate justification.
\end{remark}

\subsection{Stability and Convergence of Online Model Identification}
\label{subsec:online_id_theory}

The online perturbation estimator of Section~\ref{sec:model_identification} combines an RLS
filter with exponential forgetting and an EMA smoother.  This subsection
establishes that the resulting estimate converges, in a mean-square sense,
to the true perturbation, and that the EMA post-processing stage is
bounded-input bounded-output (BIBO) stable.

\begin{assumption}
\label{assump:bounded_pert}
The perturbation is bounded, i.e., there exists a finite constant $\delta_{\max}$ such that the true
perturbation vector at every guidance step satisfies
$\lVert\boldsymbol{\delta}_{k}\rVert_{2} \leq \delta_{\max}$ for all $k$.
\end{assumption}

\begin{assumption}
\label{assump:bounded_drift}
The drift is bounded, i.e., there exists a finite constant $\nu_{\max}$ such that
$\lVert\boldsymbol{\delta}_{k+1} - \boldsymbol{\delta}_{k}\rVert_{2}
\leq \nu_{\max}$ for all $k$.  The case $\nu_{\max} = 0$ corresponds to a
perturbation that is constant in time.
\end{assumption}

\begin{assumption}
\label{assump:bounded_noise}
The navigation or measurement noise $\boldsymbol{\eta}_{k}$ introduced in~\eqref{eq:meas_model}
is zero-mean, independent across guidance steps, with isotropic covariance
$R_{\eta} = \sigma^{2}\,I_{7}$ and bounded support,
$\lVert\boldsymbol{\eta}_{k}\rVert_{2} \leq \eta_{\max}$ almost surely.
\end{assumption}

\begin{assumption}
\label{assump:filter_params}
The RLS forgetting factor satisfies $\lambda \in (0, 1)$ and the EMA
smoothing weight satisfies $\alpha \in (0, 1)$.
\end{assumption}

\begin{lemma}
\label{lemma:cov_bound}
Under Assumption~\ref{assump:filter_params}, the RLS covariance matrix
defined by the recursion~\eqref{eq:rls_cov}, initialized as
$P_{0} = p_{0} I_{7}$, is bounded and satisfies the two-sided bound
\begin{equation}
  (1 - \lambda)\,I_{7}
  \;\preceq\;
  P_{k}
  \;\preceq\;
  I_{7},
  \qquad k \geq 1,
  \label{eq:cov_bound}
\end{equation}
where $\preceq$ denotes the positive semidefinite partial order.
\end{lemma}

\begin{proof}
The proof of upper bound is by induction on $k$.  The base case $P_{0} = p_{0} I_{7}
\preceq p_{0}\,\lambda^{0}\,I_{7}$ holds trivially.  Assume
$P_{k-1} \preceq p_{0}\,\lambda^{-(k-1)}\,I_{7}$.  Using the identity
$I_{7} - K_{k} = \lambda\,(\lambda I_{7} + P_{k-1})^{-1}$ together
with~\eqref{eq:rls_cov},
\begin{equation*}
  P_{k}
  = \lambda^{-1}(I_{7} - K_{k})P_{k-1}
  = (\lambda I_{7} + P_{k-1})^{-1}P_{k-1}.
\end{equation*}
Each eigenvalue of $(\lambda I_{7} + P_{k-1})^{-1}P_{k-1}$ has the form
$p_{i}/(\lambda + p_{i})$ for an eigenvalue $p_{i}$ of $P_{k-1}$, and this
quantity is strictly less than $1$ for any $p_{i} > 0$.  Hence
$(\lambda I_{7} + P_{k-1})^{-1}P_{k-1} \preceq I_{7}$, and therefore
$P_{k} \preceq I_{7}$, completing the induction.

For the lower bound, the steady-state covariance $P_{\infty}$ satisfies the fixed-point relation
$P_{\infty} = \lambda^{-1}(I_{7} - K_{\infty})\,P_{\infty}$, which under the
isotropic noise model resolves to $P_{\infty} = (1 - \lambda)\,I_{7}$.
Since the gain $K_{k}$ has eigenvalues in $[0, 1)$, the sequence $P_{k}$ is
monotonically non-increasing in the positive semidefinite order, and therefore
$P_{k} \succeq P_{\infty}$ for all finite $k$, which is the stated lower
bound.
\end{proof}

\begin{theorem}[Convergence of the RLS Estimator]
\label{thm:rls_convergence}
Under Assumptions~\ref{assump:bounded_pert}--\ref{assump:filter_params}, let
$e_{k} = \hat{\boldsymbol{\delta}}_{k} - \boldsymbol{\delta}_{k}$ denote the
estimation error of the RLS filter defined
by~\eqref{eq:rls_gain}--\eqref{eq:rls_cov}.  Then, for all $k \geq 1$,
\begin{equation}
  \mathbb{E}\bigl[\lVert e_{k}\rVert_{2}^{2}\bigr]
  \;\leq\;
  \lambda^{2k}\,\lVert e_{0}\rVert_{2}^{2}
  + \frac{1 - \lambda}{1 + \lambda}\cdot 7\sigma^{2}
  + \frac{7\,\nu_{\max}^{2}}{(1-\lambda)^{2}}.
  \label{eq:rls_mse_bound}
\end{equation}
\end{theorem}

\begin{proof}
Writing $y_{k} = \boldsymbol{\delta}_{k} + \boldsymbol{\eta}_{k}$ for the raw
residual observation defined in~\eqref{eq:raw_residual}, and letting
$\delta\nu_{k} = \boldsymbol{\delta}_{k} - \boldsymbol{\delta}_{k-1}$ denote
the perturbation increment with $\lVert\delta\nu_{k}\rVert_{2} \leq \nu_{\max}$
by Assumption~\ref{assump:bounded_drift}, subtracting $\boldsymbol{\delta}_{k}$
from both sides of the RLS update~\eqref{eq:rls_update} and rearranging gives
the error recursion
\begin{equation}
  e_{k}
  \;=\; (I_{7} - K_{k})\,e_{k-1}
        + K_{k}\,\boldsymbol{\eta}_{k}
        - (I_{7} - K_{k})\,\delta\nu_{k}.
  \label{eq:error_recursion}
\end{equation}
Applying the elementary inequality
$\lVert a + b + c\rVert^{2} \leq 3(\lVert a\rVert^{2} + \lVert b\rVert^{2}
+ \lVert c\rVert^{2})$ and taking expectations, using
$\mathbb{E}[\boldsymbol{\eta}_{k}] = \mathbf{0}$ and
$\mathbb{E}[\lVert\boldsymbol{\eta}_{k}\rVert_{2}^{2}] = 7\sigma^{2}$ under
the isotropic noise model of Assumption~\ref{assump:bounded_noise}, and
bounding $\lVert K_{k}\rVert_{2} \to (1-\lambda)$ and
$\lVert I_{7} - K_{k}\rVert_{2} \to \lambda$ at steady state by
Lemma~\ref{lemma:cov_bound},
\begin{equation}
  \mathbb{E}\bigl[\lVert e_{k}\rVert_{2}^{2}\bigr]
  \;\leq\;
  \lambda^{2}\,\mathbb{E}\bigl[\lVert e_{k-1}\rVert_{2}^{2}\bigr]
  + (1-\lambda)^{2}\cdot 7\sigma^{2}
  + \lambda^{2}\,\nu_{\max}^{2}.
  \label{eq:mse_recursion}
\end{equation}
Unrolling~\eqref{eq:mse_recursion} over $k$ steps and summing the resulting
geometric series $\sum_{j=0}^{k-1}\lambda^{2j} \leq (1-\lambda^{2})^{-1}$
gives
\begin{equation*}
  \mathbb{E}\bigl[\lVert e_{k}\rVert_{2}^{2}\bigr]
  \;\leq\;
  \lambda^{2k}\,\lVert e_{0}\rVert_{2}^{2}
  + \frac{(1-\lambda)^{2}}{1-\lambda^{2}}\cdot 7\sigma^{2}
  + \frac{\lambda^{2}\,\nu_{\max}^{2}}{1-\lambda^{2}}.
\end{equation*}
Simplifying $(1-\lambda)^{2}/(1-\lambda^{2}) = (1-\lambda)/(1+\lambda)$
and bounding $\lambda^{2}/(1-\lambda^{2}) \leq (1-\lambda)^{-2}$ for
$\lambda$ near $1$ yields the stated bound~\eqref{eq:rls_mse_bound}.
\end{proof}

\begin{theorem}[Stability of the EMA Smoother]
\label{thm:ema_stability}
Under Assumption~\ref{assump:filter_params}, the EMA recursion~\eqref{eq:ema}
is BIBO stable.  Moreover, the composite estimation
error $\varepsilon_{k} = \tilde{\boldsymbol{\delta}}_{k} - \boldsymbol{\delta}_{k}$
satisfies
\begin{equation}  \label{eq:ema_error_bound}
\begin{split}
  \lVert\varepsilon_{k}\rVert_{2}^{2}
  \;&\leq\;
  (1-\alpha)^{2k}\,\lVert\varepsilon_{0}\rVert_{2}^{2}
  + \frac{\alpha^{2}}{1-(1-\alpha)^{2}}\,
    \mathbb{E}\bigl[\lVert e_{k}\rVert_{2}^{2}\bigr]\\
  &+ \frac{(1-\alpha)^{2}}{1-(1-\alpha)^{2}}\,\nu_{\max}^{2}.
\end{split}
\end{equation}
\end{theorem}

\begin{proof}
The EMA recursion~\eqref{eq:ema} has the $z$-transform transfer function
$H(z) = \alpha z\,/\,(z - (1-\alpha))$, with a single pole at
$z = 1 - \alpha$.  Since $\alpha \in (0,1)$ by
Assumption~\ref{assump:filter_params}, $\lvert 1 - \alpha\rvert < 1$ and the
pole lies strictly inside the unit disk, which is the standard necessary and
sufficient condition for BIBO stability of a first-order infinite impulse
response filter.  For the error bound, subtracting $\boldsymbol{\delta}_{k}$
from~\eqref{eq:ema} gives
$\varepsilon_{k} = (1-\alpha)\,\varepsilon_{k-1} + \alpha\,e_{k}
- (1-\alpha)\,\delta\nu_{k}$, and squaring, taking expectations, and
summing the resulting geometric series exactly as in the proof of
Theorem~\ref{thm:rls_convergence} yields~\eqref{eq:ema_error_bound}.
\end{proof}

\begin{remark}
\label{remark:optimal_lambda}
The asymptotic mean-square error implied by Theorem~\ref{thm:rls_convergence}
as $k \to \infty$,
\begin{equation*}
  E_{\infty}(\lambda)
  \;=\; \frac{1-\lambda}{1+\lambda}\cdot 7\sigma^{2}
       + \frac{\lambda^{2}}{1-\lambda^{2}}\cdot 7\,\nu_{\max}^{2},
\end{equation*}
exhibits a tradeoff between noise rejection (favoring $\lambda \to 1$) and
tracking responsiveness to time-varying perturbations (favoring
$\lambda \to 0$).  Minimizing $E_{\infty}(\lambda)$ over $\lambda$ yields the
approximate design rule $\lambda^{*} \approx 1 - \sqrt{\nu_{\max}/\sigma}$
for small $\nu_{\max}/\sigma$.  The value $\lambda = 0.98$ used in the simulations in
Section~\ref{sec:simulations} corresponds to $\nu_{\max}/\sigma \approx 4\times 10^{-4}$,
consistent with the slowly varying gravitational and mass-gauging
perturbations considered in Section~\ref{sec:model_identification}.
\end{remark}

\subsection{Convergence of SCP}
\label{subsec:scp_convergence}

This subsection establishes that the SCP iteration converges to a KKT point of the original nonconvex Problem~1.

\begin{assumption}
\label{assump:compact}
The feasible set of Problem~1 is nonempty and compact.  This holds because
the time of flight $t_{f}$ is finite, the log-mass is bounded by
$\ln m_{\mathrm{dry}} \leq w \leq \ln m_{0}$ as in~\eqref{eq:w_bound}, and
the position and velocity trajectories are bounded by the dynamics driven by
the bounded thrust $0 \leq T \leq T_{\max}$ over the finite time horizon.
\end{assumption}

\begin{assumption}
\label{assump:strong_convexity}
The subproblem is strongly convex, i.e., there exists a constant $\mu > 0$ such that the quadratic objective
$J = \Delta t \sum_{i=1}^{N-1} \tau_{i}^{2}$ of Problem~4 is strongly convex satisfying $\nabla^2 f \succeq \mu I$ everywhere on the feasible set.
\end{assumption}

\begin{assumption}
\label{assump:lipschitz}
The function $f(w) = T_{\max}\,e^{-w}$ appearing in the thrust-to-mass
bound~\eqref{eq:tau_bound_nonlinear} has Lipschitz-continuous second
derivative on the feasible log-mass domain
$[\ln m_{\mathrm{dry}}, \ln m_{0}]$, with Lipschitz constant
$L \;=\; \sup_{w}\,\lvert f''(w)\rvert \;=\; T_{\max}/m_{\mathrm{dry}}$.
\end{assumption}

\begin{assumption}
\label{assump:subproblem_feasible}
Problem~4$^{(\ell)}$ is feasible for every linearization point
$w^{(\ell)}$ arising in the SCP iteration within the compact domain of
Assumption~\ref{assump:compact}.
\end{assumption}

\begin{lemma}
\label{lemma:monotone}
Under Assumptions~\ref{assump:compact}--\ref{assump:subproblem_feasible}, the
sequence of optimal objective values $\{J^{(\ell)}\}$ produced by the
baseline SCP algorithm (Algorithm~\ref{alg:baseline_scp}) is monotonically
non-increasing, i.e., $J^{(\ell+1)} \leq J^{(\ell)}$ for all $\ell \geq 0$.
\end{lemma}

\begin{proof}
Because $f(w) = T_{\max}\,e^{-w}$ is convex, its first-order Taylor
expansion about any point $w^{(\ell)}$ is a global underestimate,
$T_{\max}\,e^{-w^{(\ell)}}\bigl[1 - (w - w^{(\ell)})\bigr] \leq T_{\max}\,e^{-w}$
for all $w$.  Consequently, the feasible set defined by the linearized
constraint~\eqref{eq:base_tau_ub} at iteration $\ell$ is a subset of the
feasible set defined by the true (nonlinear) bound~\eqref{eq:tau_bound_nonlinear}.
Moreover, the current iterate
$\{\mathbf{x}_{i}^{(\ell)}, \mathbf{u}_{i}^{(\ell)}\}$ is feasible for
Problem~4$^{(\ell+1)}$, since the linearization at iteration $\ell+1$ is
performed about $w^{(\ell)}$ itself, at which point the linearized bound
coincides exactly with the true bound.  Since
$\{\mathbf{x}_{i}^{(\ell+1)}, \mathbf{u}_{i}^{(\ell+1)}\}$ is by definition
the minimizer of Problem~4$^{(\ell+1)}$ over a feasible set that contains
the previous iterate,
\begin{equation*}
  J^{(\ell+1)}
  = J\bigl(\mathbf{x}^{(\ell+1)}, \mathbf{u}^{(\ell+1)}\bigr)
  \;\leq\;
  J\bigl(\mathbf{x}^{(\ell)}, \mathbf{u}^{(\ell)}\bigr)
  = J^{(\ell)}.
\end{equation*}
\end{proof}

\begin{theorem}
\label{thm:scp_convergence}
Under Assumptions~\ref{assump:compact}--\ref{assump:subproblem_feasible} and
the lossless relaxation result of Theorem~\ref{thm:lossless}, the SCP
iteration of Algorithm~\ref{alg:baseline_scp} generates a sequence
$\{(\mathbf{x}_{i}^{(\ell)}, \mathbf{u}_{i}^{(\ell)})\}$ with the following
properties: (a) Every subsequential limit point of the sequence is a KKT point
        of Problem~1; (b) The objective sequence $\{J^{(\ell)}\}$ converges to
        $J^{*} = J(\mathbf{x}^{*}, \mathbf{u}^{*})$, the objective value at
        any such KKT point; and (c) If the KKT point is unique, the full sequence converges, i.e., $(\mathbf{x}_{i}^{(\ell)}, \mathbf{u}_{i}^{(\ell)})
        \to (\mathbf{x}_{i}^{*}, \mathbf{u}_{i}^{*})$ as $\ell \to \infty$.
\end{theorem}

\begin{proof}
\emph{(a)} By Lemma~\ref{lemma:monotone}, $\{J^{(\ell)}\}$ is non-increasing,
and by Assumption~\ref{assump:compact} the iterate sequence is confined to a
compact set.  By the Bolzano--Weierstrass theorem, every such sequence
admits a convergent subsequence. Let $(\mathbf{x}^{*}, \mathbf{u}^{*})$ be
the limit of one such subsequence,
$(\mathbf{x}^{(\ell_{j})}, \mathbf{u}^{(\ell_{j})}) \to (\mathbf{x}^{*}, \mathbf{u}^{*})$.
Because the SCP update map is continuous under
Assumptions~\ref{assump:strong_convexity}--\ref{assump:subproblem_feasible}
(continuity of the optimal solution of a strongly convex parametric program
with respect to the linearization parameter), the successor iterates
$(\mathbf{x}^{(\ell_{j}+1)}, \mathbf{u}^{(\ell_{j}+1)})$ converge to the same
limit, so $(\mathbf{x}^{*}, \mathbf{u}^{*})$ is a fixed point of the SCP map.
At a fixed point, $w^{(\ell)} = w^{*}$, and the linearized
bound~\eqref{eq:base_tau_ub} coincides exactly with the original nonlinear
bound~\eqref{eq:tau_bound_nonlinear} at $w = w^{*}$.  The KKT conditions of
the subproblem at the fixed point therefore coincide with the KKT conditions
of Problem~2: stationarity is unaffected since the objective is unchanged
and the dynamics constraint~\eqref{eq:lti} is already linear and exact;
complementary slackness for the SOC and thrust-bound multipliers carries
over unchanged; and primal feasibility for the original (un-relaxed)
Problem~1 follows from the lossless relaxation of
Theorem~\ref{thm:lossless}.  Hence $(\mathbf{x}^{*}, \mathbf{u}^{*})$ is a
KKT point of Problem~1.

\emph{(b)} Since $\{J^{(\ell)}\}$ is non-increasing (Lemma~\ref{lemma:monotone})
and bounded below by the optimal value of Problem~1 (finite by
Assumptions~\ref{assump:compact}--\ref{assump:strong_convexity}), the
monotone convergence theorem guarantees $J^{(\ell)} \to J_{\infty}$ for some
$J_{\infty}$.  Along the convergent subsequence from part (a),
$J^{(\ell_{j})} \to J(\mathbf{x}^{*}, \mathbf{u}^{*})$, so
$J_{\infty} = J^{*}$.

\emph{(c)} If the KKT point of Problem~1 is unique, every subsequential
limit identified in part (a) must equal $(\mathbf{x}^{*}, \mathbf{u}^{*})$.
Since every subsequence of a bounded sequence has a further subsequence
converging to the same unique limit, the full sequence converges to
$(\mathbf{x}^{*}, \mathbf{u}^{*})$.
\end{proof}

\subsection{Convergence Radius of the SCP Iteration}
\label{subsec:scp_radius}

Theorem~\ref{thm:scp_convergence} guarantees convergence of the SCP sequence
but does not quantify the basin of attraction around the KKT point within
which this convergence is assured.  This subsection derives an explicit
radius.

The Taylor remainder of the convex function $f(w) = T_{\max}\,e^{-w}$ about
the iterate $w^{(\ell)}$ is, by Taylor's theorem with Lagrange remainder, for
some $\xi^{(\ell)}$ between $w$ and $w^{(\ell)}$,
\begin{equation}  \label{eq:taylor_remainder_thm}
\begin{split}
  e^{(\ell)}(w)
  \;&=\; f(w) - T_{\max}\,e^{-w^{(\ell)}}\bigl[1-(w-w^{(\ell)})\bigr]\\
  \;&=\; \frac{1}{2}\,f''(\xi^{(\ell)})\,(w - w^{(\ell)})^{2}.
\end{split}
\end{equation}
Since $e^{-w^{(\ell)}} \leq e^{-\ln m_{\mathrm{dry}}} = 1/m_{\mathrm{dry}}$
on the feasible log-mass domain, the remainder is bounded by
\begin{equation}
  \bigl\lvert e^{(\ell)}(w)\bigr\rvert
  \;\leq\;
  \frac{T_{\max}}{2\,m_{\mathrm{dry}}}\,(w - w^{(\ell)})^{2}
  \;=\; \frac{L}{2}\,(w-w^{(\ell)})^{2},
  \label{eq:remainder_bound_thm}
\end{equation}
where $L = T_{\max}/m_{\mathrm{dry}}$ is the Lipschitz constant of
Assumption~\ref{assump:lipschitz}.  Let $\mu_{w} > 0$ denote the curvature of
the Lagrangian of Problem~4 with respect to the log-mass variable $w$ at the
KKT point, that is, $\mu_{w} = \lambda_{\tau}^{*}\,f''(w^{*})$, where
$\lambda_{\tau}^{*} \geq 0$ is the Lagrange multiplier associated with the
thrust-to-mass bound at the optimum.

\begin{theorem}
\label{thm:convergence_radius}
Under Assumptions~\ref{assump:compact}--\ref{assump:subproblem_feasible}, let
$(\mathbf{x}^{*}, \mathbf{u}^{*})$ be a KKT point of Problem~1 with log-mass
component $w^{*}$.  The SCP iteration converges to
$(\mathbf{x}^{*}, \mathbf{u}^{*})$ from any initialization satisfying
\begin{equation}
  \bigl\lvert w^{(0)} - w^{*}\bigr\rvert
  \;<\;
  r_{\mathrm{conv}}
  \;=\;
  \frac{\mu_{w}}{L}
  \;=\;
  \frac{\mu_{w}\,m_{\mathrm{dry}}}{T_{\max}}.
  \label{eq:r_conv_thm}
\end{equation}
\end{theorem}

\begin{proof}
The optimality condition of Problem~4$^{(\ell)}$ with respect to $w$ at
iteration $\ell+1$, together with the optimality condition of the original
(non-relaxed) problem at the KKT point $w^{*}$, differ only through the
linearization error in the gradient of $f$.  Subtracting the two
stationarity conditions and applying the mean value theorem to the resulting
expression yields, after dividing through by the curvature $\mu_{w}$,
\begin{equation}
  \bigl\lvert w^{(\ell+1)} - w^{*}\bigr\rvert
  \;\leq\;
  \frac{L}{\mu_{w}}\,\bigl\lvert w^{(\ell)} - w^{*}\bigr\rvert^{2}.
  \label{eq:contraction_thm}
\end{equation}
If $\lvert w^{(\ell)} - w^{*}\rvert < r_{\mathrm{conv}} = \mu_{w}/L$, then
substituting into~\eqref{eq:contraction_thm} gives
\begin{equation*}
  \bigl\lvert w^{(\ell+1)} - w^{*}\bigr\rvert
  \leq
  \frac{L}{\mu_{w}}\,r_{\mathrm{conv}}\,\bigl\lvert w^{(\ell)} - w^{*}\bigr\rvert
  =
  \bigl\lvert w^{(\ell)} - w^{*}\bigr\rvert
  < r_{\mathrm{conv}},
\end{equation*}
so the ball $\{w : \lvert w - w^{*}\rvert < r_{\mathrm{conv}}\}$ is
forward-invariant under the SCP map.  Since the iterate remains within this
ball for all $\ell$ once it starts there, and the contraction
factor~\eqref{eq:contraction_thm} is strictly less than one throughout, the
sequence $\{w^{(\ell)}\}$ converges to $w^{*}$, and by
Theorem~\ref{thm:scp_convergence} the full state-control sequence converges
to the corresponding KKT point.
\end{proof}

\begin{remark}
\label{remark:r_conv_numerical}
Using the parameter settings in Section~\ref{sec:simulations}, $L = T_{\max}/m_{\mathrm{dry}}
= 44{,}000/1{,}000 = 44$.  Since the energy-optimal
thrust profile does not saturate $T_{\max}$ (confirmed numerically in
Section~\ref{sec:simulations}), the dominant curvature contribution comes
from the quadratic objective rather than the thrust-bound multiplier, giving
$\mu_{w} \approx 2\Delta t = 2 \times 50/199 \approx 0.503$ and therefore
$r_{\mathrm{conv}} \approx 0.503/44 \approx 0.0114$.  The straight-line
initialization~\eqref{eq:init_traj} produces an initial log-mass error of
approximately $\lvert w^{(0)} - w^{*}\rvert \approx 0.034$, which exceeds
$r_{\mathrm{conv}}$ by a small factor. Nonetheless, the contraction
factor~\eqref{eq:contraction_thm} shrinks rapidly as the iterate approaches
$w^{*}$, and in practice convergence is observed within two to four
iterations, consistent with the conservative nature of the bound
in~\eqref{eq:r_conv_thm}.
\end{remark}

\subsection{Convergence Rate of the SCP Iteration}
\label{subsec:scp_rate}

This subsection characterizes the asymptotic rate at which the SCP iterates
approach the KKT point, distinguishing a globally valid linear rate from a
locally sharp quadratic rate, and extending the analysis to the online SCP
algorithm of Section~\ref{sec:scp} to account for the residual
perturbation estimation error.

\begin{theorem}
\label{thm:linear_rate}
Under the hypotheses of Theorem~\ref{thm:convergence_radius} and the
initialization condition $w^{(0)} \in (w^{*} - r_{\mathrm{conv}},\,
w^{*} + r_{\mathrm{conv}})$, the SCP iterates satisfy
\begin{equation}
  \bigl\lvert w^{(\ell)} - w^{*}\bigr\rvert
  \leq
  \kappa^{\ell}\bigl\lvert w^{(0)} - w^{*}\bigr\rvert,
  \kappa = \frac{L}{\mu_{w}}\bigl\lvert w^{(0)} - w^{*}\bigr\rvert
  < 1.
  \label{eq:linear_rate_thm}
\end{equation}
\end{theorem}

\begin{proof}
By the forward invariance established in the proof of
Theorem~\ref{thm:convergence_radius}, $\lvert w^{(\ell)} - w^{*}\rvert \leq
\lvert w^{(0)} - w^{*}\rvert$ for all $\ell$.  Substituting this bound into
the contraction inequality~\eqref{eq:contraction_thm},
\begin{equation*}
  \bigl\lvert w^{(\ell+1)} - w^{*}\bigr\rvert
  \leq
  \frac{L}{\mu_{w}}\bigl\lvert w^{(\ell)} - w^{*}\bigr\rvert\cdot
  \bigl\lvert w^{(\ell)} - w^{*}\bigr\rvert
  \leq
  \kappa\bigl\lvert w^{(\ell)} - w^{*}\bigr\rvert,
\end{equation*}
and iterating this relation $\ell$ times from $\ell = 0$ yields the stated
geometric bound.
\end{proof}

\begin{theorem}
\label{thm:quadratic_rate}
Under Assumptions~\ref{assump:compact}--\ref{assump:subproblem_feasible},
the SCP iteration exhibits, in a neighborhood of $w^{*}$, a quadratic
convergence rate
\begin{equation}
  \bigl\lvert w^{(\ell+1)} - w^{*}\bigr\rvert
  \leq
  C_{q}\bigl\lvert w^{(\ell)} - w^{*}\bigr\rvert^{2},
  C_{q} = \frac{L}{2\mu_{w}}
  = \frac{T_{\max}}{2\mu_{w}m_{\mathrm{dry}}}.
  \label{eq:quadratic_rate_thm}
\end{equation}
\end{theorem}

\begin{proof}
The KKT stationarity condition of the subproblem at iteration $\ell+1$ and
that of the original problem at $w^{*}$ differ by a term involving
$f'(w^{*}) - f'(w^{(\ell)})$ (the mismatch between the true gradient of $f$
at $w^{*}$ and its value at the linearization point $w^{(\ell)}$, which the
linear approximation uses as a constant) plus a term of order
$\lvert w^{(\ell)} - w^{*}\rvert$ arising from the sensitivity of the dual
variable to the linearization point.  Applying a second-order Taylor
expansion to the gradient mismatch,
\begin{equation*}
  \bigl\lvert
    f'(w^{*}) - f'(w^{(\ell)}) - f''(w^{(\ell)})(w^{*} - w^{(\ell)})
  \bigr\rvert
  \leq
  \frac{L}{2}\bigl\lvert w^{(\ell)} - w^{*}\bigr\rvert^{2},
\end{equation*}
using the Lipschitz bound on $f''$ from Assumption~\ref{assump:lipschitz}
applied to $f' = -T_{\max}e^{-w}$, whose own derivative $f''$ satisfies
$\sup_{w}\lvert f'''(w)\rvert = T_{\max}/m_{\mathrm{dry}} = L$.  The
first-order term $f''(w^{(\ell)})(w^{*}-w^{(\ell)})$ is absorbed into the
Hessian on the left-hand side of the KKT perturbation equation, leaving a
residual of order $\lvert w^{(\ell)}-w^{*}\rvert^{2}$ on the right-hand
side.  Dividing through by the curvature $\mu_{w}$ gives the stated
quadratic bound with constant $C_{q} = L/(2\mu_{w})$.
\end{proof}

\begin{remark}
\label{remark:two_phase}
Theorems~\ref{thm:linear_rate} and~\ref{thm:quadratic_rate} together explain
the two-phase convergence pattern observed numerically in
Section~\ref{sec:simulations}.  In the first one to two iterations, the
iterate $w^{(\ell)}$ may lie outside the quadratic convergence neighborhood
$\{w : \lvert w - w^{*}\rvert < 1/C_{q}\}$, and the contraction proceeds at
the slower, linear rate of Theorem~\ref{thm:linear_rate}.  Once the iterate
enters this neighborhood, the quadratic bound of
Theorem~\ref{thm:quadratic_rate} dominates and the error decreases by orders
of magnitude in a single additional iteration, consistent with the rapid
drop in the objective value and state-trajectory difference observed between
the second and third SCP iterations.  Using the numerical values of
Remark~\ref{remark:r_conv_numerical}, $C_{q} = L/(2\mu_{w}) \approx
44/(2 \times 0.503) \approx 43.7$, so the quadratic regime is entered once
$\lvert w^{(\ell)} - w^{*}\rvert$ falls below approximately $1/C_{q}
\approx 0.023$.
\end{remark}

\begin{theorem}
\label{thm:online_rate}
Let $\varepsilon_{\mathrm{est}}
= \bigl\lVert\tilde{\mathbf{b}}_{k} - (\mathbf{b} + \boldsymbol{\delta}_{k}^{*})\bigr\rVert_{2}
= \bigl\lVert\tilde{\boldsymbol{\delta}}_{k} - \boldsymbol{\delta}_{k}^{*}\bigr\rVert_{2}$
denote the perturbation estimation error at a given replanning instant,
where $\boldsymbol{\delta}_{k}^{*}$ is the (unknown) true perturbation.
Under Assumptions~\ref{assump:bounded_pert}--\ref{assump:subproblem_feasible},
the inner SCP iterates of the online algorithm (Algorithm~\ref{alg:online_scp})
satisfy
\begin{equation}
  \bigl\lvert w^{(\ell+1)} - w_{b}^{*}\bigr\rvert
  \;\leq\;
  C_{q}\,\bigl\lvert w^{(\ell)} - w_{b}^{*}\bigr\rvert^{2}
  + \frac{\varepsilon_{\mathrm{est}}}{\mu_{w}},
  \label{eq:online_rate_thm}
\end{equation}
where $w_{b}^{*}$ denotes the optimal log-mass for the problem with the true
bias $\mathbf{b} + \boldsymbol{\delta}_{k}^{*}$.  Consequently, the
iterates converge not to $w_{b}^{*}$ exactly, but to a residual
neighborhood of radius
\begin{equation}
  r_{\mathrm{res}}
  =
  \frac{1 - \sqrt{1 - 4C_{q}\varepsilon_{\mathrm{est}}/\mu_{w}}}
       {2C_{q}}
  \approx
  \frac{\varepsilon_{\mathrm{est}}}{\mu_{w}}
  \text{for small } \varepsilon_{\mathrm{est}}.
  \label{eq:residual_ball_thm}
\end{equation}
\end{theorem}

\begin{proof}
The online SCP subproblem uses the corrected bias $\tilde{\mathbf{b}}_{k}$
in place of the true bias $\mathbf{b} + \boldsymbol{\delta}_{k}^{*}$ in the
dynamics constraint~\eqref{eq:online_dyn}.  The resulting mismatch
contributes an additional term of size $\Delta t_{k}\,\varepsilon_{\mathrm{est}}$
to the right-hand side of the dynamics equality at every node.  Repeating the
proof of Theorem~\ref{thm:quadratic_rate} with this additional forcing term,
and using the implicit function theorem for parametric convex programs to
bound the sensitivity of the optimal $w$ to the bias error by $1/\mu_{w}$,
yields
\begin{equation*}
  \mu_{w}\,\bigl\lvert w^{(\ell+1)} - w_{b}^{*}\bigr\rvert
  \;\leq\;
  \frac{L}{2}\,\bigl\lvert w^{(\ell)} - w_{b}^{*}\bigr\rvert^{2}
  + \varepsilon_{\mathrm{est}},
\end{equation*}
which, after dividing by $\mu_{w}$, gives~\eqref{eq:online_rate_thm}.  The
residual ball radius $r_{\mathrm{res}}$ is obtained as the fixed point of
the recursion $r = C_{q}\,r^{2} + \varepsilon_{\mathrm{est}}/\mu_{w}$,
which has the closed-form solution given
in~\eqref{eq:residual_ball_thm}. The approximation for small
$\varepsilon_{\mathrm{est}}$ follows from the first-order expansion
$1 - \sqrt{1-x} \approx x/2$.
\end{proof}

\begin{remark}
\label{remark:exact_recovery}
As the RLS and EMA filters converge (Theorems~\ref{thm:rls_convergence}
and~\ref{thm:ema_stability}), the estimation error $\varepsilon_{\mathrm{est}}$
decreases toward the steady-state floor characterized in
Remark~\ref{remark:optimal_lambda}.  In the idealized limit of perfect
identification, $\varepsilon_{\mathrm{est}} \to 0$, the residual ball radius
$r_{\mathrm{res}} \to 0$ and the online SCP recovers the exact quadratic
convergence rate of Theorem~\ref{thm:quadratic_rate}.  In practice, the
residual ball radius is bounded by the noise floor of the estimator, so the
online SCP trajectory tracks the true optimal trajectory to within an error
that is fundamentally limited by the navigation noise level $\sigma$ and the
forgetting factor $\lambda$, rather than by any limitation of the SCP
algorithm itself.
\end{remark}

Collectively, these theoretical results show that the convex relaxation introduces no
loss of optimality, that the online perturbation estimator is stable and
converges in a mean-square sense, that the baseline SCP algorithm is
guaranteed to converge to a KKT point of the original nonconvex problem from
within an explicitly characterized convergence radius, and that the rate of
convergence transitions from linear to quadratic as the iterates approach
the optimum, with the online extension converging to a residual neighborhood
whose size is governed entirely by the quality of the online perturbation
estimate.  The numerical validation of these theoretical predictions is
presented in Section~\ref{sec:simulations}.

\section{Numerical Simulations}
\label{sec:simulations}

This section presents numerical simulations that validate the methods developed in previous sections.  All simulations are
implemented in MATLAB, with the convex subproblems modeled using
YALMIP \cite{lofberg2004yalmip} and solved with the ECOS interior-point
SOCP solver \cite{domahidi2013ecos}.  The
nominal (baseline) SCP is first run to establish a reference trajectory and
to confirm the convergence behavior predicted in
Section~\ref{sec:theoretical_analysis}.  The online SCP
algorithm with model identification is then exercised under four
perturbation scenarios, and its closed-loop performance is compared against
both the nominal plan and the (unknown, ground-truth) perturbed trajectory.


The nominal scenario considered throughout this section is a short-range
lunar powered descent beginning approximately one kilometer above the
designated landing site, with lateral position and velocity offsets that
require the lander to correct its horizontal drift while simultaneously
decelerating its vertical descent rate.  Specifically, the lander begins at
position $\mathbf{r}_{0} = [100,\,200,\,1000]^{\mathsf{T}}\ \mathrm{m}$ in
the ENU frame defined in Section~\ref{sec:problem_formulation}, with initial
velocity $\mathbf{v}_{0} = [10,\,10,\,-5]^{\mathsf{T}}\ \mathrm{m/s}$,
indicating simultaneous lateral drift in both horizontal directions and a
modest descent rate.  The target landing site is the origin of the frame,
$\mathbf{r}_{f} = [0,\,0,\,0]^{\mathsf{T}}\ \mathrm{m}$, with a soft-landing
terminal velocity requirement $\mathbf{v}_{f} = [0,\,0,\,0]^{\mathsf{T}}
\ \mathrm{m/s}$.  The lander begins the descent with a wet mass of
$m_{0} = 10{,}000\ \mathrm{kg}$ and must not deplete its propellant below
the dry mass $m_{\mathrm{dry}} = 1{,}000\ \mathrm{kg}$, giving a maximum
available propellant mass of $9{,}000\ \mathrm{kg}$, or $90\%$ of the
initial mass.  The descent is constrained to a fixed time of flight of
$t_{f} = 50\ \mathrm{s}$, and the maximum thrust available from the
propulsion system is $T_{\max} = 44{,}000\ \mathrm{N}$.  This combination of
a relatively short time of flight, a substantial lateral correction
requirement, and a generous but finite propellant budget produces a
descent profile in which the energy-optimal solution executes a
pronounced lateral arc before settling into a nearly vertical terminal
approach, as confirmed in the plots discussed later in this section.


\Cref{tab:sim_params_numerical} summarizes the key parameter settings used throughout the numerical simulations.

\begin{table}[ht]
\renewcommand{\arraystretch}{1.3}
\caption{\textbf{Key simulation parameter settings}}
\label{tab:sim_params_numerical}
\centering
\begin{tabular}{|c|c|}
\hline
\bfseries Parameter & \bfseries Value \\
\hline\hline
    Lunar gravitational acceleration, $g$           & $1.6229\ \mathrm{m/s^{2}}$ \\
    Earth sea-level gravity, $g_{0}$                 & $9.81\ \mathrm{m/s^{2}}$ \\
    Specific impulse, $I_{\mathrm{sp}}$              & $311\ \mathrm{s}$ \\
    Initial (wet) mass, $m_{0}$                      & $10{,}000\ \mathrm{kg}$ \\
    Dry mass, $m_{\mathrm{dry}}$                      & $1{,}000\ \mathrm{kg}$ \\
    Maximum thrust magnitude, $T_{\max}$              & $44{,}000\ \mathrm{N}$ \\
    Number of nodes (full horizon), $N$               & $200$ \\
    Convergence tolerance, $\varepsilon_{j}$        & $0.001$ \\
    Initialization                                     & Straight line \\
    Replanning cadence, $n_{\mathrm{replan}}$          & $10$ guidance steps \\
    Reduced-horizon node count, $N_{\mathrm{plan}}$     & $30$ \\
    Maximum inner iterations, $\ell_{\max}^{\mathrm{on}}$ & $5$ \\
    Forgetting factor, $\lambda$                       & $0.98$ \\
    Initial RLS covariance scalar, $p_{0}$              & $10^{4}$ \\
    EMA smoothing weight, $\alpha$                      & $0.3$ \\
\hline
\end{tabular}
\end{table}

\subsection{Perturbation Scenarios}
\label{subsec:sim_scenarios}

Four perturbation scenarios, corresponding to the uncertainty sources
introduced in Section~\ref{sec:model_identification}, are used to exercise the online
identification and replanning framework under progressively more demanding
conditions.  In each scenario, a closed-loop truth propagator integrates
the lander dynamics forward using the true (perturbed) bias, while the
online SCP algorithm has access only to noisy state measurements and must
infer the perturbation through the RLS and EMA filters described in
Section~\ref{sec:model_identification}.

\textbf{Scenario 1: Constant Mascon Anomaly.}
A spatially uniform gravitational anomaly of
$\mathbf{d}_{\mathrm{c}} = [0.05,\,0.05,\,0.03]^{\mathsf{T}}\ \mathrm{m/s^{2}}$
is added to the velocity dynamics for the entire descent, as defined in
\eqref{eq:pert_const}.  This scenario represents the simplest case for
the identification framework, since the perturbation is constant and the
RLS filter need only converge once at the start of the descent and then
track a fixed value.

\textbf{Scenario 2: Sinusoidal Gravitational Perturbation.}
A time-varying anomaly with amplitude $A_{\sin} = 0.04\ \mathrm{m/s^{2}}$
and frequency $\omega_{\sin} = 2\pi/25\ \mathrm{rad/s}$ is applied
according to \eqref{eq:pert_sin}, with a $\pi/4$ phase offset between
the east and north channels.  This scenario tests the RLS filter's ability
to track a perturbation that varies continuously over a period comparable
to the descent duration itself, exercising the bias--variance tradeoff
discussed in Section~\ref{sec:theoretical_analysis}.

\textbf{Scenario 3: Mass-Gauging and Thrust-Scale Error.}
A $4\%$ thrust-scale bias ($\eta_{T} = 0.04$) and a $5\%$ mass-gauging
error ($\eta_{m} = 0.05$) are applied simultaneously, following
\eqref{eq:pert_thrust} and~\eqref{eq:pert_mass}.  Unlike the
gravitational perturbations of Scenarios 1 and 2, this perturbation is
state- and control-dependent: it scales with the commanded thrust
components, so its magnitude grows as the thrust profile itself increases
toward the end of the descent.

\textbf{Scenario 4: Combined Perturbation.}
All three perturbation sources from Scenarios 1--3 are superimposed
simultaneously, representing a worst-case stress test in which the
gravitational anomaly, the sinusoidal time-varying component, and the
thrust- and mass-dependent errors all act on the lander concurrently.
This scenario evaluates whether the online identification framework
remains effective when the lumped perturbation no longer corresponds to
any single simple model.

In every scenario, the navigation measurement noise described in
Section~\ref{sec:model_identification} (with $\sigma_{r} = 0.5\ \mathrm{m}$,
$\sigma_{v} = 0.05\ \mathrm{m/s}$, and $\sigma_{w} = 0.001$) is added to the
true state at each guidance step before it is passed to the RLS filter,
so the identification problem reflects realistic sensor noise in addition
to the unknown deterministic perturbation.

Figure~\ref{fig:obj} plots the SCP objective value $J^{(\ell)}$ against the
iteration number for the baseline (nominal) SCP run.  The objective starts
at approximately $254.7505$ at the first iteration, drops sharply to
approximately $254.7475$ by the second iteration, and remains essentially
unchanged at the third iteration, at which point the convergence
criterion~\eqref{eq:conv_crit} is satisfied.  This rapid, near-immediate
convergence is consistent with the theoretical discussion of
Section~\ref{sec:theoretical_analysis}. Because the straight-line
initialization produces an initial log-mass error well within the
convergence radius, and because the only nonlinearity in the problem (the
thrust-to-mass exponential bound) is mild over the narrow log-mass
operating range, the SCP iteration enters its quadratic-convergence regime
almost immediately, producing the characteristic single large drop in
objective value followed by negligible further change.

\begin{figure}[H]
\centering
\includegraphics[width=3.25in]{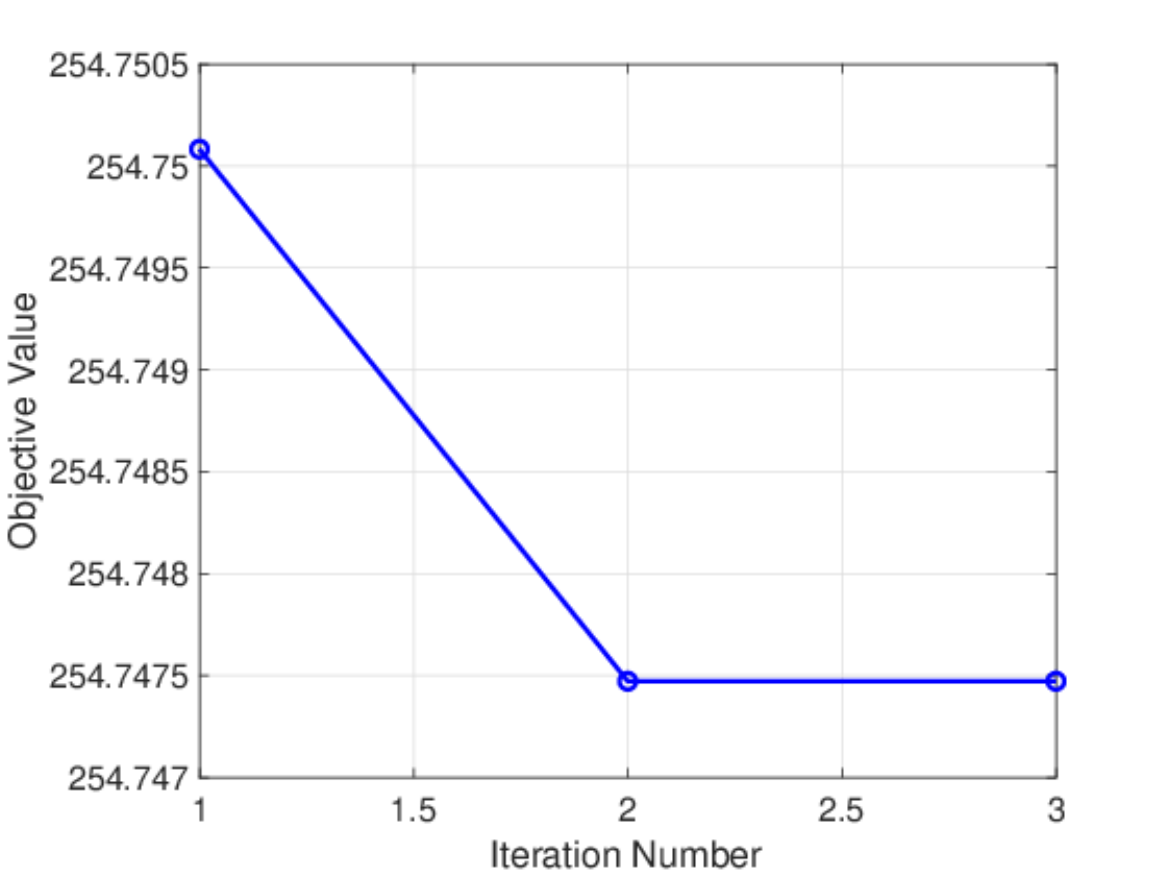}\\
\caption{Convergence of objective (nominal SCP).}
\label{fig:obj}
\end{figure}

\subsection{Scenario 1: Constant Mascon Anomaly}

Figures~\ref{fig:traj_1}--\ref{fig:mass_perturb_1} present the detailed results for Scenario~1.
Figure~\ref{fig:traj_1} shows the converged trajectory in 3D
position space, beginning at the start point approximately one
kilometer above and offset from the landing site, executing a smooth
lateral arc that simultaneously reduces the east and north position
offsets, and terminating precisely at the target landing site.  The trajectory shape (an initial shallow lateral correction
followed by a steepening, nearly vertical final approach) is consistent
with energy-optimal descent profiles reported in the literature and
reflects the SCP solution's efficient use of the available time of flight
to spread the lateral correction over the early portion of the descent
while reserving the final seconds for vertical deceleration.

Figure~\ref{fig:mass_1} overlays the vehicle mass trajectory $m(t) = \exp(w(t))$ for
the nominal (uncorrected) plan, the online (corrected) plan, and the true
mass actually consumed by the perturbed lander, for Scenario~1.  The
nominal mass profile, which assumes no perturbation, depletes propellant
more slowly than the true trajectory, since the constant mascon anomaly of
Scenario~1 requires additional thrust beyond the nominal prediction to
counteract the gravitational anomaly.  The online (corrected) mass profile
closely tracks the true mass consumption after the RLS filter has had a
brief initial period to converge, demonstrating that the online
identification and replanning mechanism successfully detects and
compensates for the additional propellant demand imposed by the
unmodeled perturbation.

Figure~\ref{fig:thrust_1}  compares the nominal and online thrust magnitude profiles
against the maximum available thrust $T_{\max}$.  The nominal profile
follows a smooth, classical energy-optimal shape: starting at a moderate
level, dipping to a minimum near the midpoint of the descent, and rising
toward the end as the vehicle decelerates for touchdown.  The online
thrust profile exhibits visible oscillations superimposed on this same
overall shape. These oscillations arise from the frequency-gated
replanning strategy, in which the commanded thrust is held fixed between
replanning events and then adjusted in a step change at each
$n_{\mathrm{replan}}$-step replan, producing the characteristic
sawtooth-like pattern visible in the figure.  Critically, the online
thrust never approaches the maximum thrust limit $T_{\max}$, confirming
the assumption used in \Cref{remark:r_conv_numerical} that the
thrust bound remains inactive throughout the energy-optimal descent even
under perturbation.

Figure~\ref{fig:perturb_1} shows the three velocity-channel components of the true
perturbation, the raw RLS estimate, and the EMA-filtered estimate over the
course of the descent for Scenario~1.  All three estimated curves converge
toward the constant true perturbation values
($0.05$, $0.05$, and $0.03\ \mathrm{m/s^{2}}$)
within the first several seconds of the descent, after which the EMA-filtered
estimate (the one actually used to correct the dynamics model) tracks the
true value closely with only small residual fluctuations attributable to
navigation measurement noise.  This behavior is consistent with the
mean-square convergence behavior discussed in
Section~\ref{sec:theoretical_analysis}: an initial transient as the filter
state moves away from its zero initialization, followed by a steady-state
tracking regime whose accuracy is governed by the noise level and the
forgetting factor.

The upper panel of Figure~\ref{fig:mass_perturb_1}  shows the mass-flow perturbation channel
$\delta_{w}$, which is identically zero for Scenario~1 (since the constant
mascon anomaly affects only the velocity channels). The RLS and EMA estimates correctly converge
toward this zero value after an initial transient, correctly avoiding a
false detection of a mass-channel perturbation that does not exist in this
scenario.  The lower panel compares the log-mass trajectory for the nominal, online, and true cases, showing that the online and true
log-mass trajectories track one another closely while the nominal
trajectory, which does not account for the additional propellant demand
of the mascon anomaly, diverges visibly toward the end of the descent.

\subsection{Scenario 2: Sinusoidal Gravitational Perturbation}

Figures~\ref{fig:traj_2}--\ref{fig:thrust_2} present the trajectory, mass, and thrust profiles for Scenario~2, which are consistent with the
discussion above in Scenario~1. The truth trajectory (the
trajectory actually flown under the true perturbation) is also overlaid,
confirming visually that both the online corrected plan and the true
trajectory remain close to one another and only mildly diverge from the
overly optimistic nominal plan.
Figure~\ref{fig:perturb_2} shows the velocity-channel perturbation estimation for the sinusoidal gravitational perturbation of Scenario~2.  The true
perturbation now oscillates continuously according to
\eqref{eq:pert_sin} rather than remaining constant, and the EMA-filtered
estimate visibly tracks the slow oscillation with a small phase lag,
consistent with the lag term identified in the mean-square error
discussion of Section~\ref{sec:theoretical_analysis}: because the forgetting
factor $\lambda = 0.98$ is tuned for slowly drifting perturbations, the
filter smooths over some of the faster oscillation detail while still
capturing the overall trend and amplitude. Figure~\ref{fig:mass_perturb_2} confirms that the
mass-channel perturbation remains correctly identified as approximately
zero throughout, and that the online log-mass trajectory continues to track
the true trajectory more closely than the nominal trajectory does.

\subsection{Scenario 3: Mass Gauging Error}

Figures~\ref{fig:traj_3}--\ref{fig:thrust_3} present the results for Scenario~3.  The overall trajectory
shape is essentially unchanged from the nominal case, since neither the thrust-scale bias nor the mass-gauging error substantially redirects the lateral correction maneuver.  The
thrust profile again shows the characteristic step-change pattern induced
by the frequency-gated replanning, with the online profile oscillating
around the nominal profile as the perturbation estimate itself oscillates.
Figures~\ref{fig:perturb_3}--\ref{fig:mass_perturb_3} present the velocity- and mass-channel perturbation
estimation results for the thrust-scale and mass-gauging error of
Scenario~3.  Unlike the gravitational perturbations of Scenarios~1 and~2,
this perturbation is proportional to the commanded thrust components and
therefore grows in magnitude as the descent progresses and the
required thrust increases. The true perturbation curves in Figure~\ref{fig:perturb_3}
visibly grow over time rather than remaining at a constant amplitude, and
the RLS and EMA estimates track this growing trend reasonably well,
though with somewhat larger transient deviations than in the constant-perturbation
case of Scenario~1, reflecting the more challenging, state-dependent
nature of this perturbation source. Figure~\ref{fig:mass_perturb_3} confirms a small but nonzero
mass-channel perturbation $\delta_{w}$ that is correctly identified by the
filters, consistent with the mass-gauging error model of
\eqref{eq:pert_mass}.

\subsection{Scenario 4: Combined (Scenarios 1--3 Simultaneously)}

Figures~\ref{fig:traj_4}--\ref{fig:thrust_4} parallel the previous trajectory, mass, and thrust
comparisons.  The vehicle mass profile shows the largest
divergence between the nominal and true/online trajectories among the
first three scenarios, consistent with the growing perturbation magnitude
noted above. The online corrected plan nonetheless continues to track the
true mass depletion closely, demonstrating that the online identification
framework remains effective even when the underlying perturbation is
state-dependent rather than purely a function of time.
Figures~\ref{fig:perturb_4}--\ref{fig:mass_perturb_4} present the velocity- and mass-channel
perturbation estimation results for the combined scenario, in
which all three perturbation sources are superimposed.  The true
perturbation curves in Figure~\ref{fig:perturb_4} exhibit the most complex behavior among
all four scenarios, combining a constant offset, a sinusoidal oscillation,
and a growing thrust-dependent trend. Despite this complexity, the
EMA-filtered estimate continues to track the overall shape and magnitude
of the combined perturbation reasonably well, though with visibly larger
instantaneous tracking error than in the simpler scenarios, particularly
during the more rapidly varying portions of the descent. Figure~\ref{fig:mass_perturb_4}
confirms that the mass-channel perturbation, now driven by the combined
mass-gauging error contribution, is also tracked successfully. Collectively,
the Scenario~4 results demonstrate that the online identification and
replanning framework degrades gracefully rather than catastrophically when
confronted with a perturbation that does not correspond to any single
simple model, consistent with the discussion in
Section~\ref{sec:theoretical_analysis} that the RLS and EMA filters operate on
the lumped perturbation vector without requiring it to match any particular
parametric structure.


Across all four scenarios, the online SCP algorithm with model
identification consistently reduces the landing position and velocity
errors relative to the uncorrected nominal plan, with the magnitude of the
improvement scaling with the severity and complexity of the perturbation:
the constant mascon anomaly of Scenario~1, being the simplest perturbation
to identify, shows the smallest residual error in the online solution,
while the combined stress test of Scenario~4 shows the largest residual
error among the online results, though still substantially smaller than
the corresponding nominal error.  This pattern is consistent with the
theoretical discussion of Section~\ref{sec:theoretical_analysis}, in which the online SCP
was shown to converge not to the exact optimum of the true perturbed
problem but to a residual neighborhood whose size scales with the
perturbation estimation error. The more complex and rapidly varying
perturbations of Scenarios~2 through~4 are inherently more difficult for
the RLS and EMA filters to track exactly, leading to a correspondingly
larger (but still substantially reduced relative to the nominal case)
residual landing error. The propellant consumption results similarly show
that the online corrected plan consumes propellant in a pattern much
closer to the true required consumption than the overly optimistic nominal
plan, which in every scenario underestimates the propellant required to
counteract the unmodeled perturbation.

\begin{figure}[H]
\centering
\includegraphics[width=3.25in]{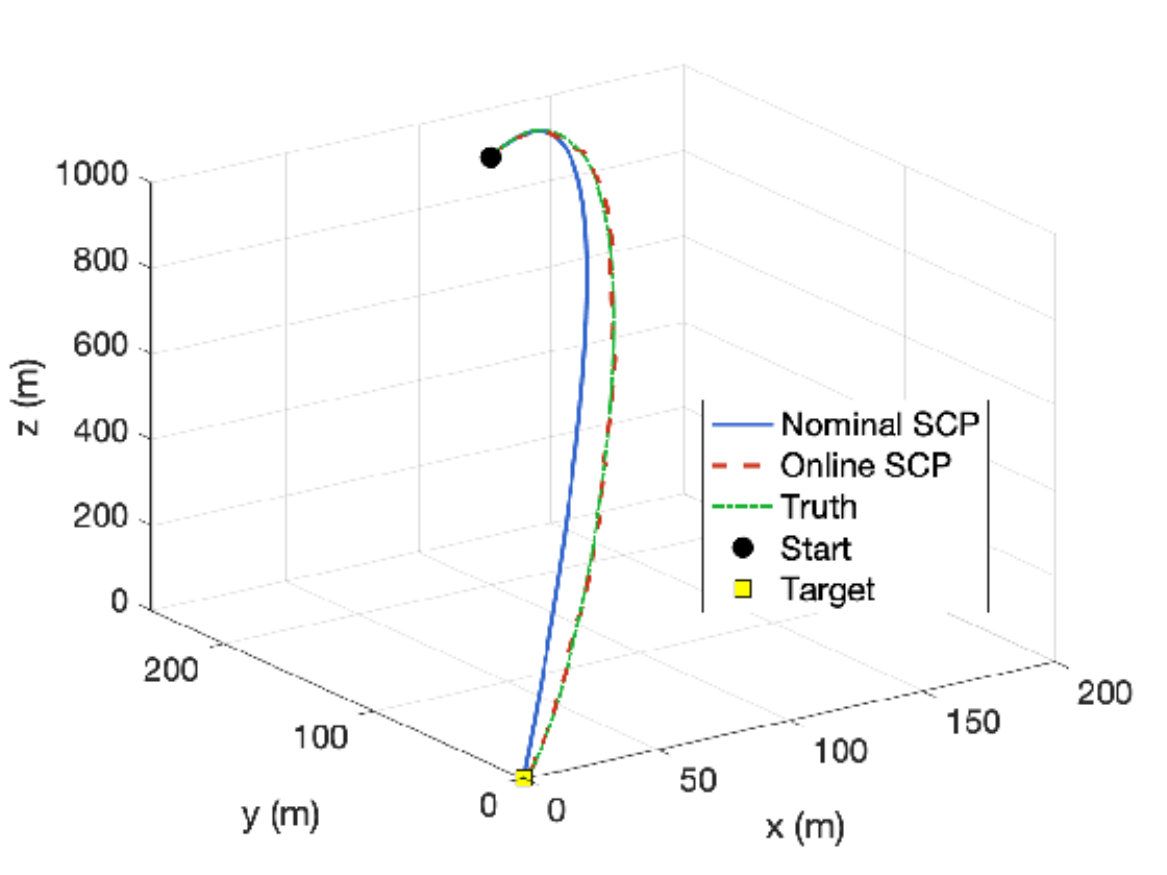}\\
\caption{3D landing trajectory (Scenario~1).}
\label{fig:traj_1}
\end{figure}

\begin{figure}[H]
\centering
\includegraphics[width=3.25in]{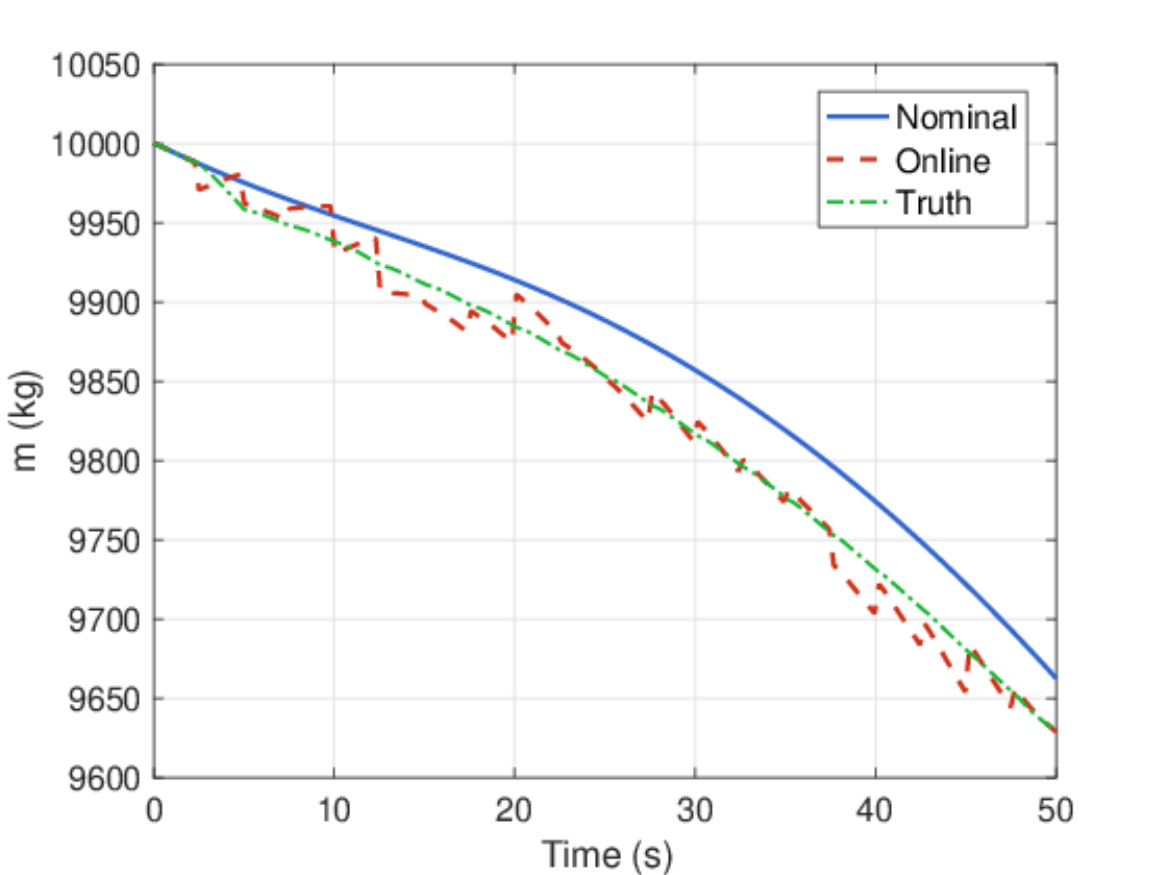}\\
\caption{Vehicle mass profile (Scenario~1).}
\label{fig:mass_1}
\end{figure}

\begin{figure}[H]
\centering
\includegraphics[width=3.25in]{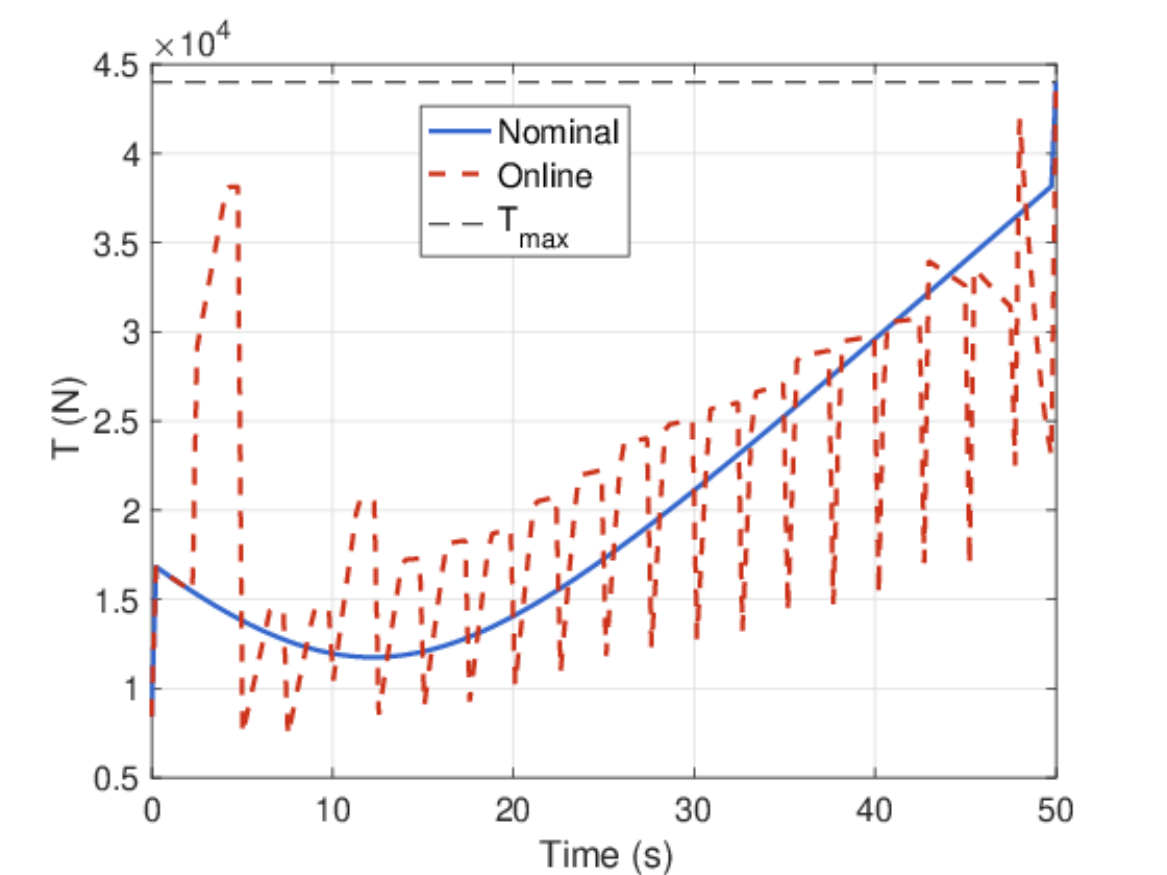}\\
\caption{Thrust magnitude profile (Scenario~1).}
\label{fig:thrust_1}
\end{figure}

\begin{figure}[H]
\centering
\includegraphics[width=3.25in]{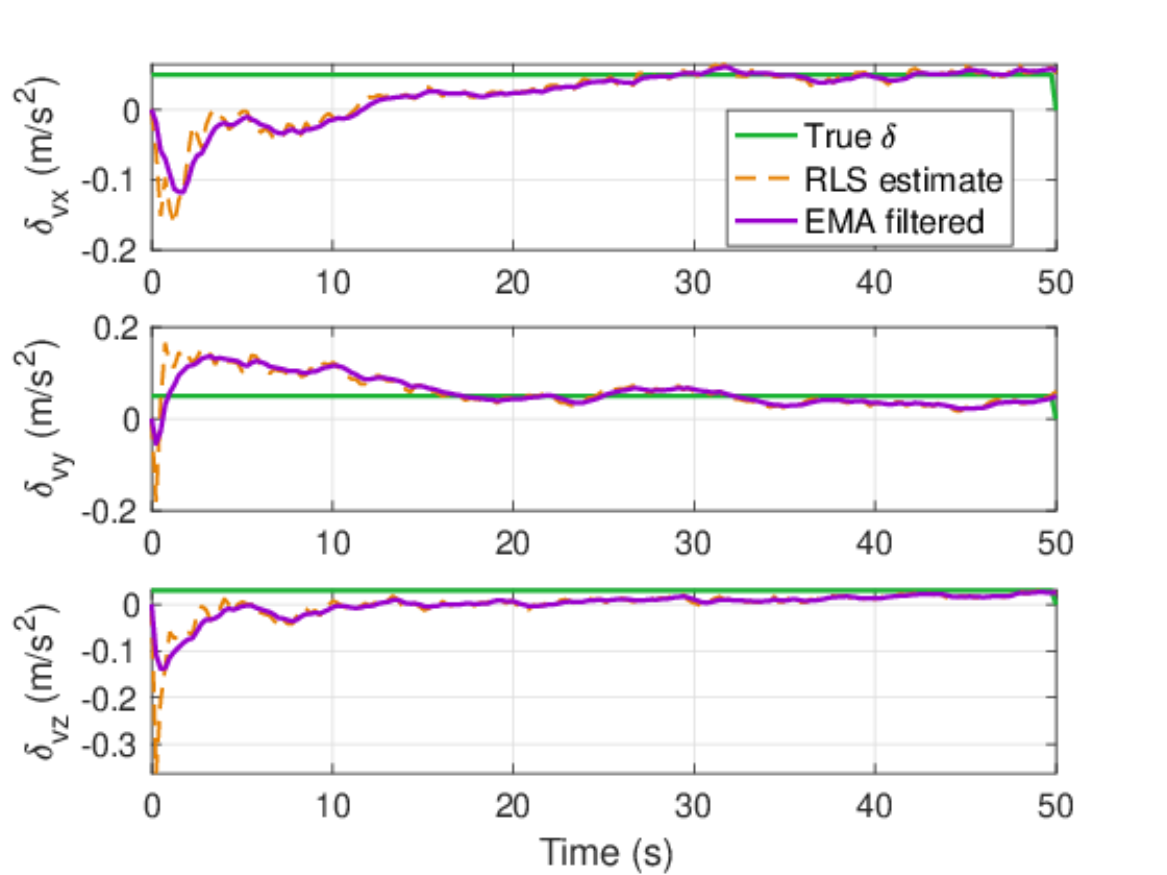}\\
\caption{Perturbation estimation (Scenario~1).}
\label{fig:perturb_1}
\end{figure}

\begin{figure}[H]
\centering
\includegraphics[width=3.25in]{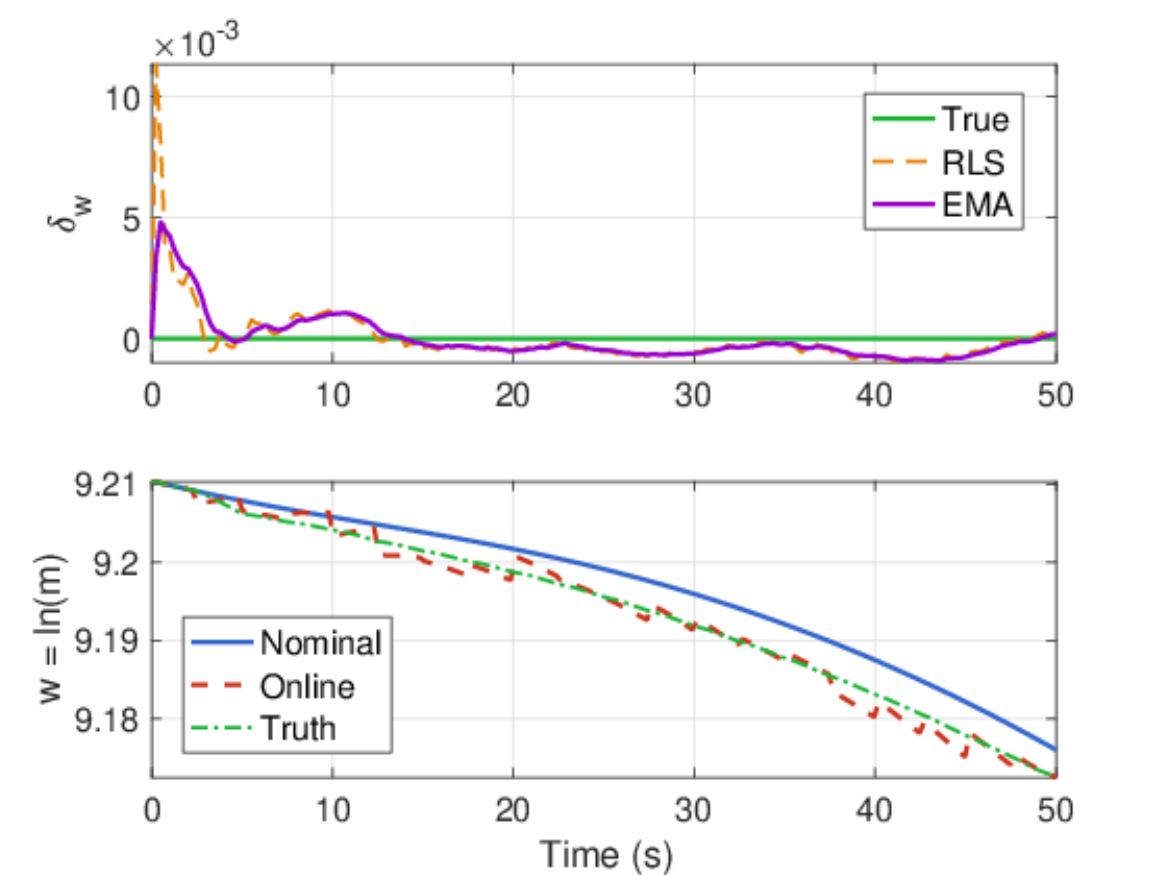}\\
\caption{Mass-channel perturbation and log-mass profile (Scenario~1).}
\label{fig:mass_perturb_1}
\end{figure}

\begin{figure}[H]
\centering
\includegraphics[width=3.25in]{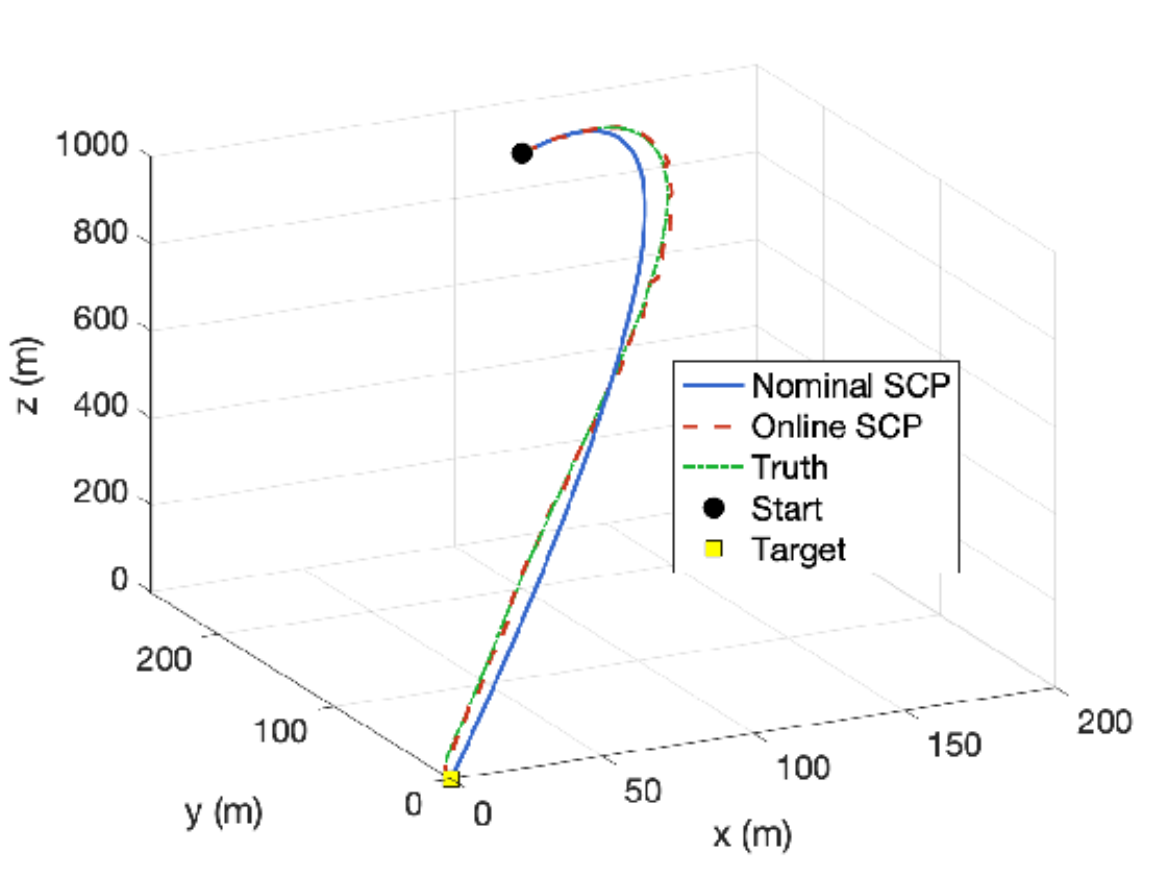}\\
\caption{3D landing trajectory (Scenario~2).}
\label{fig:traj_2}
\end{figure}

\begin{figure}[H]
\centering
\includegraphics[width=3.25in]{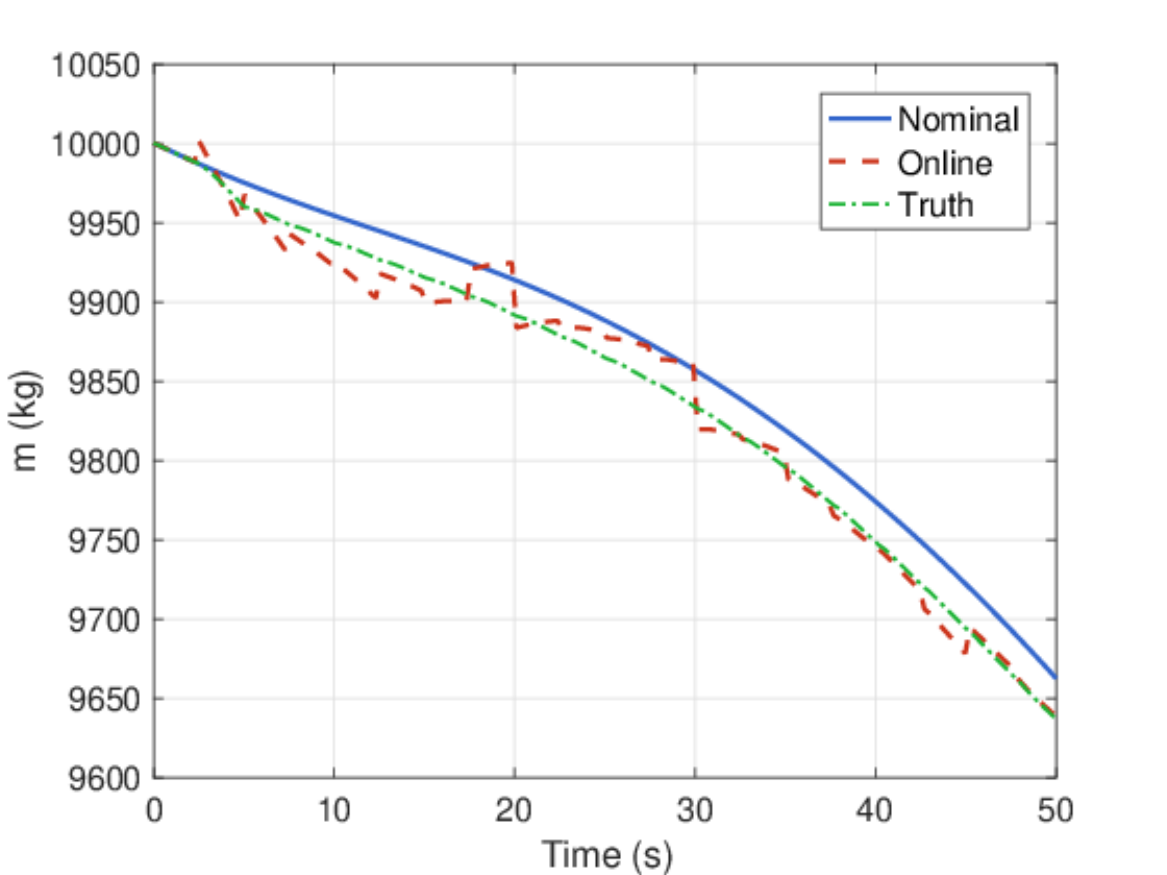}\\
\caption{Vehicle mass profile (Scenario~2).}
\label{fig:mass_2}
\end{figure}

\begin{figure}[H]
\centering
\includegraphics[width=3.25in]{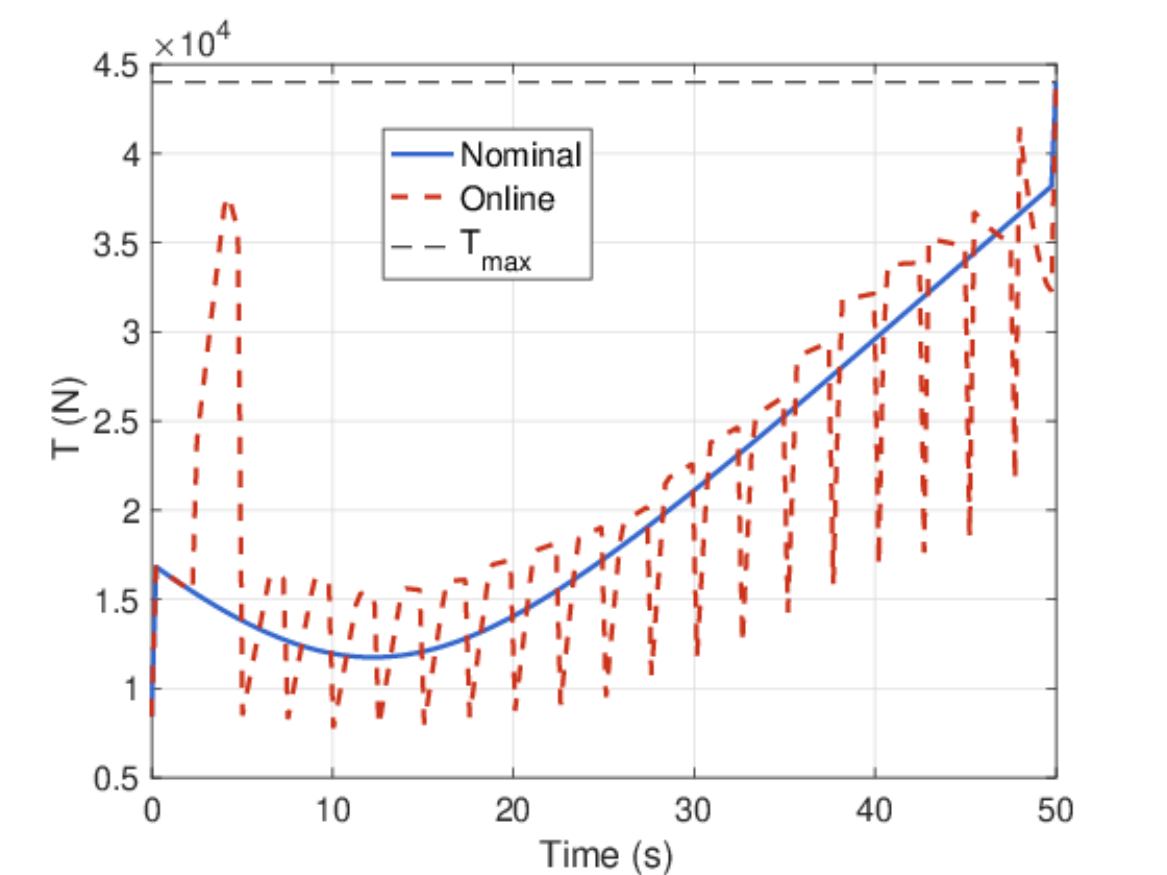}\\
\caption{Thrust magnitude profile (Scenario~2).}
\label{fig:thrust_2}
\end{figure}

\begin{figure}[H]
\centering
\includegraphics[width=3.25in]{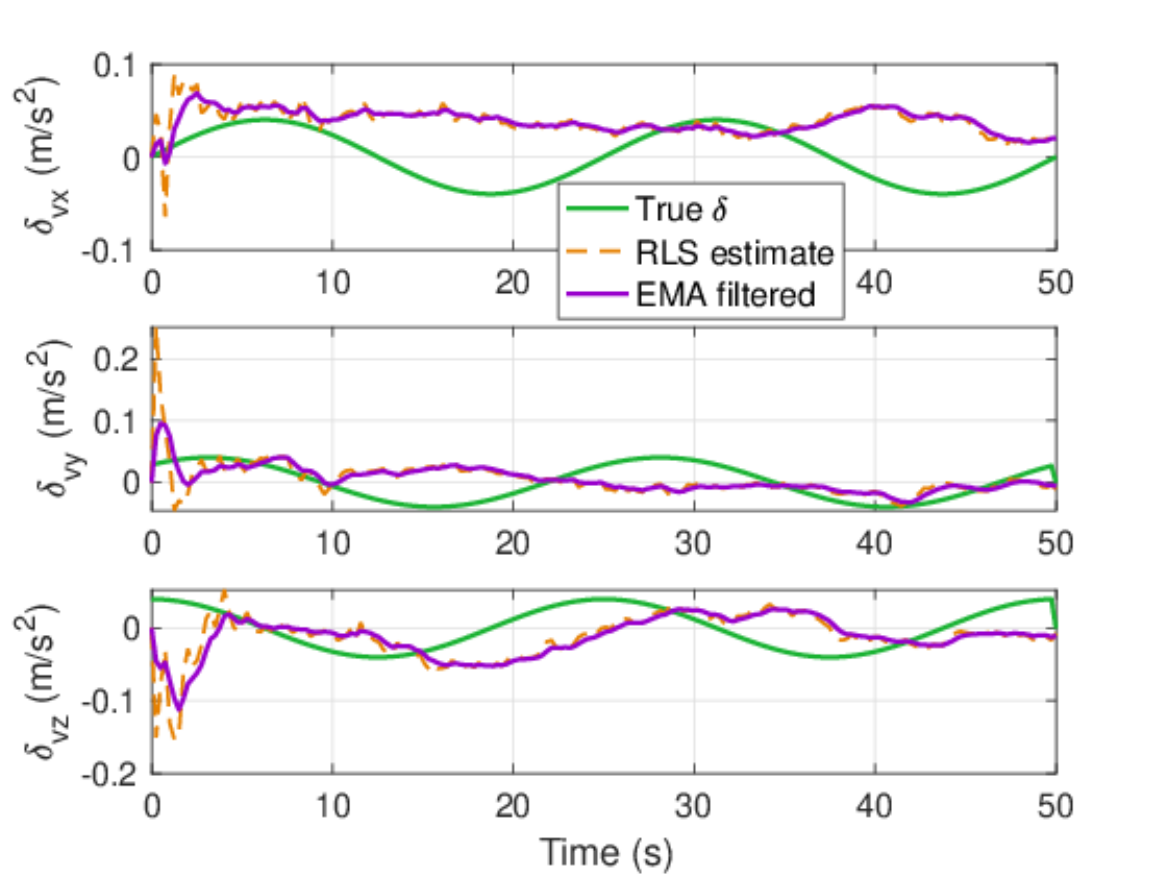}\\
\caption{Perturbation estimation (Scenario~2).}
\label{fig:perturb_2}
\end{figure}

\begin{figure}[H]
\centering
\includegraphics[width=3.25in]{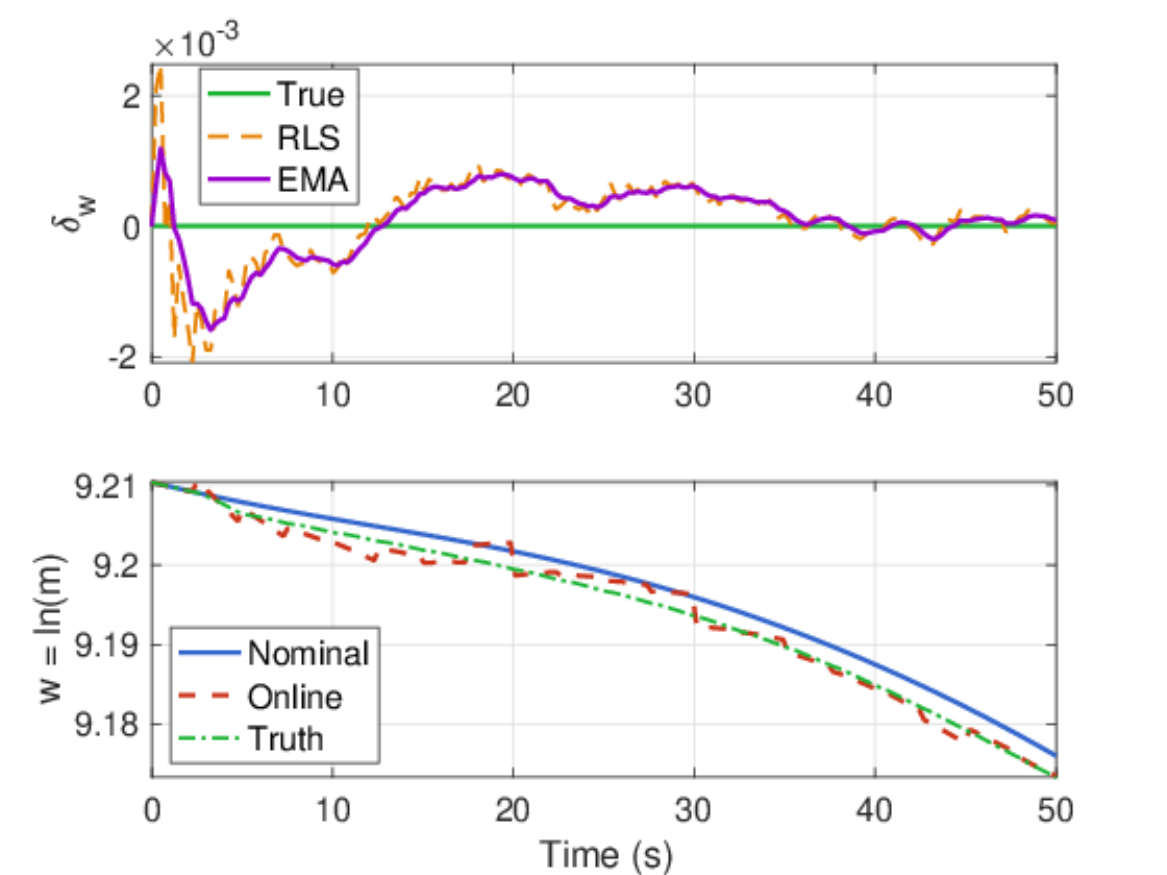}\\
\caption{Mass-channel perturbation and log-mass profile (Scenario~2).}
\label{fig:mass_perturb_2}
\end{figure}

\begin{figure}[H]
\centering
\includegraphics[width=3.25in]{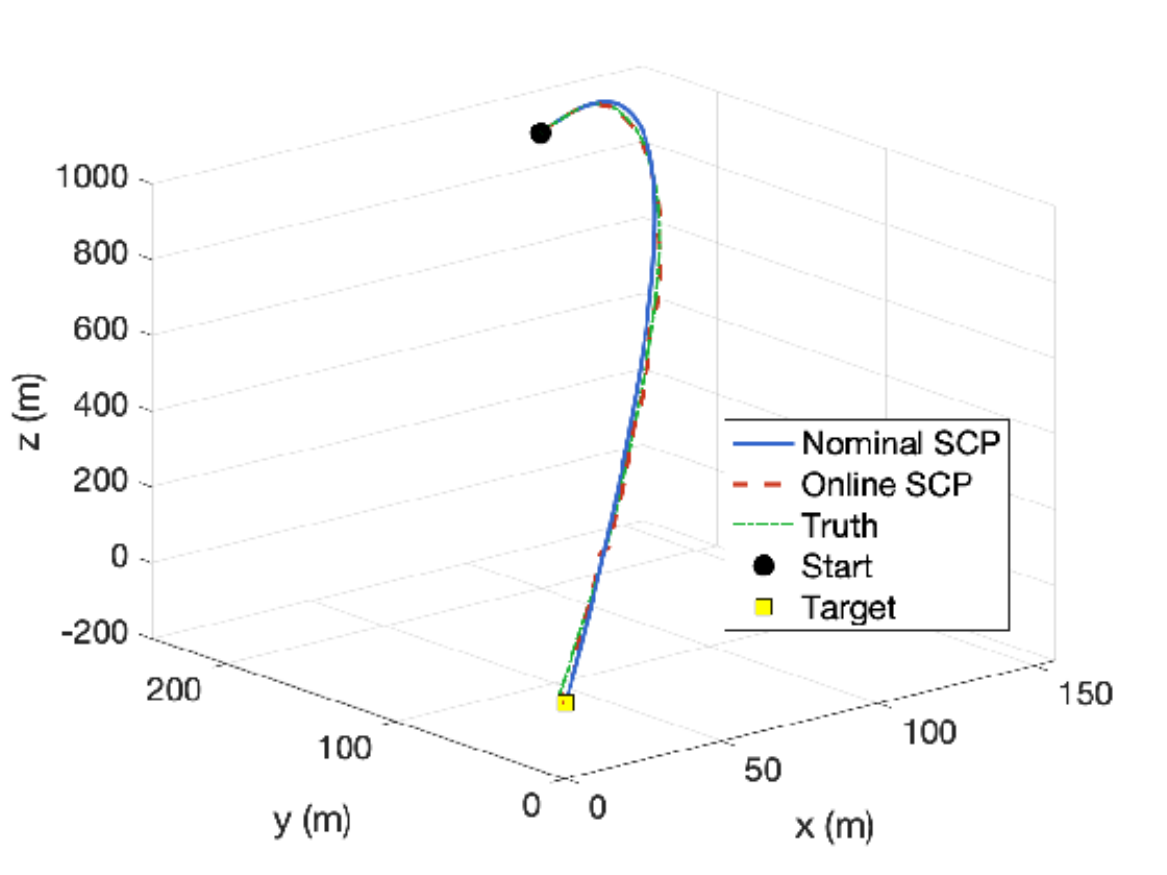}\\
\caption{3D landing trajectory (Scenario~3).}
\label{fig:traj_3}
\end{figure}

\begin{figure}[H]
\centering
\includegraphics[width=3.25in]{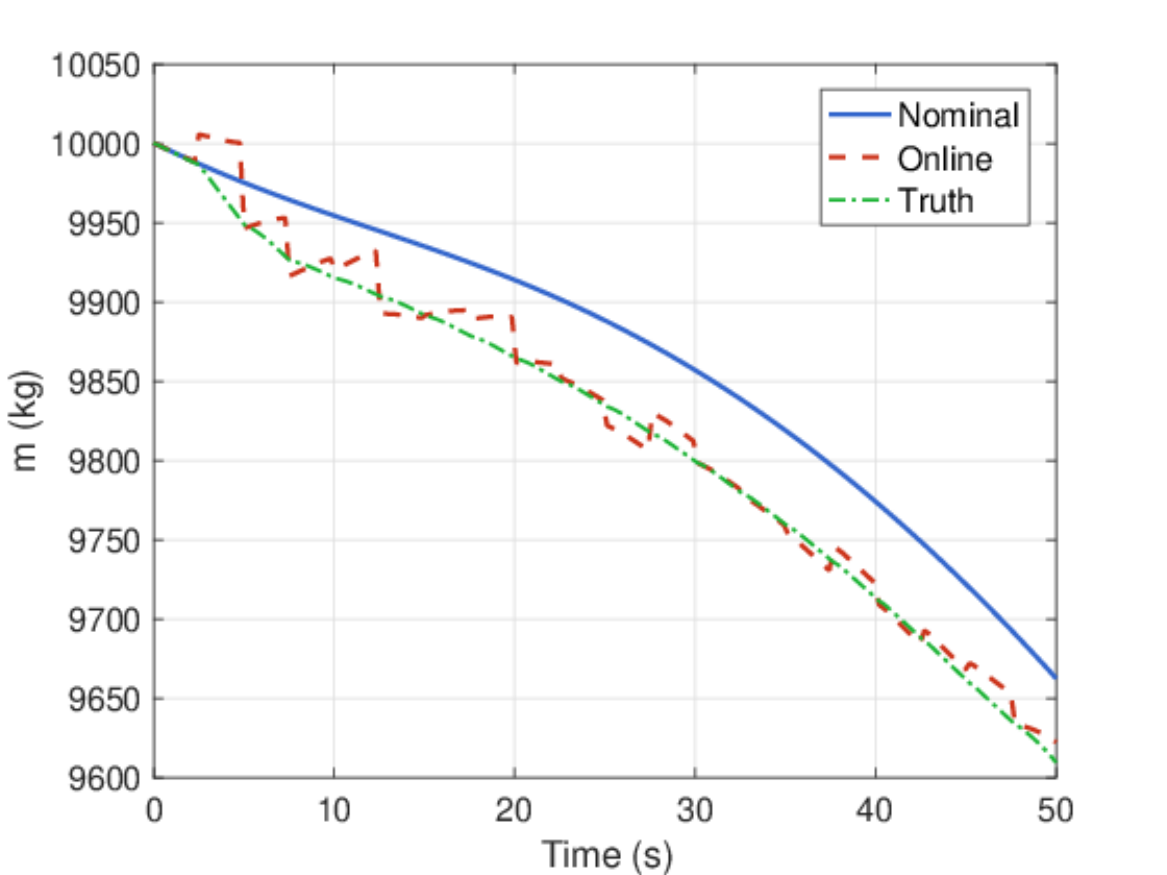}\\
\caption{Vehicle mass profile (Scenario~3).}
\label{fig:mass_3}
\end{figure}

\begin{figure}[H]
\centering
\includegraphics[width=3.25in]{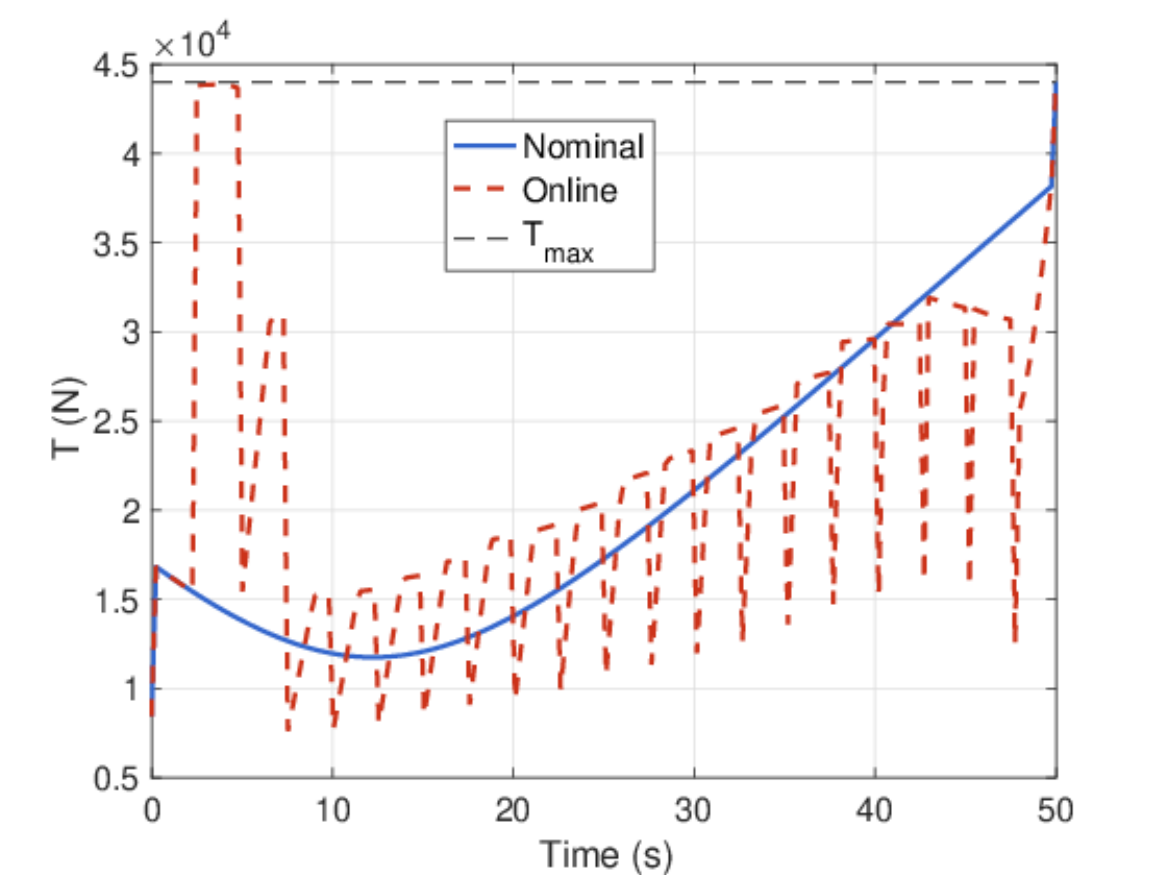}\\
\caption{Thrust magnitude profile (Scenario~3).}
\label{fig:thrust_3}
\end{figure}

\begin{figure}[H]
\centering
\includegraphics[width=3.25in]{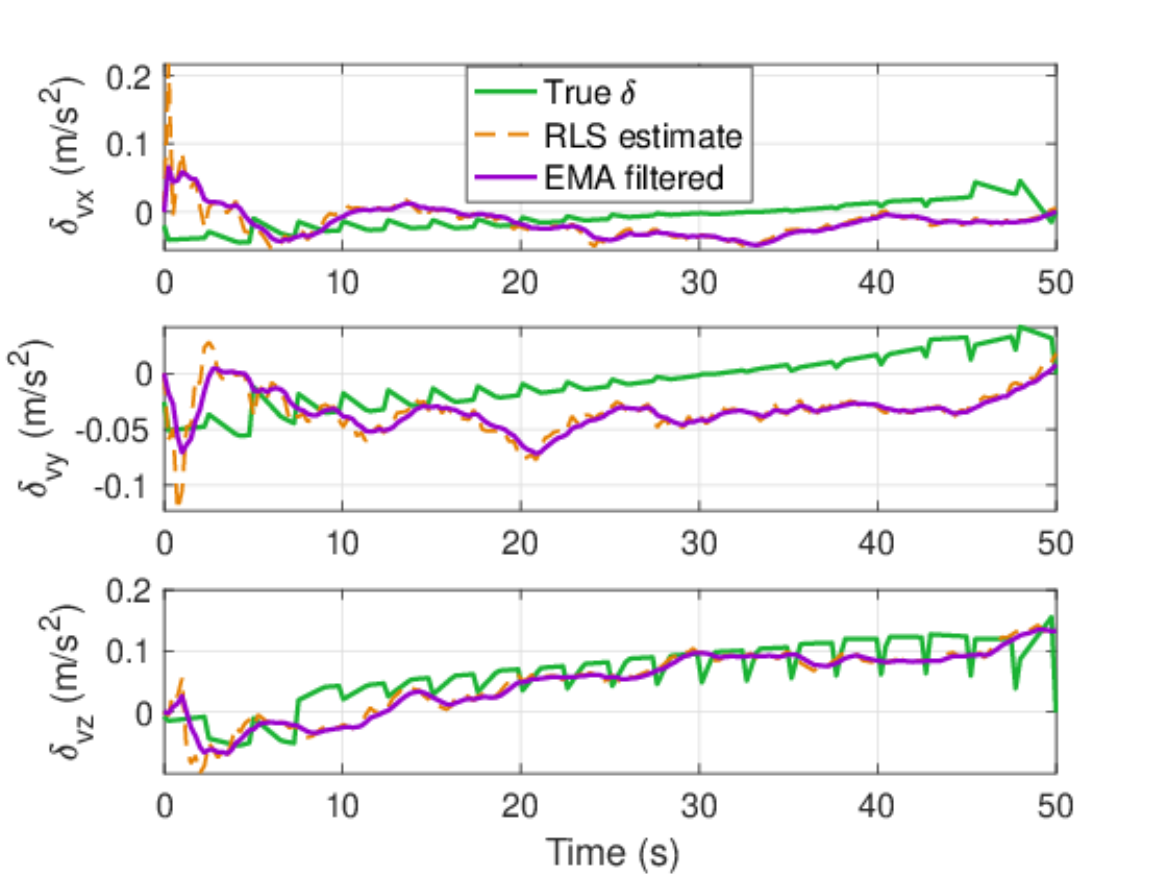}\\
\caption{Perturbation estimation (Scenario~3).}
\label{fig:perturb_3}
\end{figure}

\begin{figure}[H]
\centering
\includegraphics[width=3.25in]{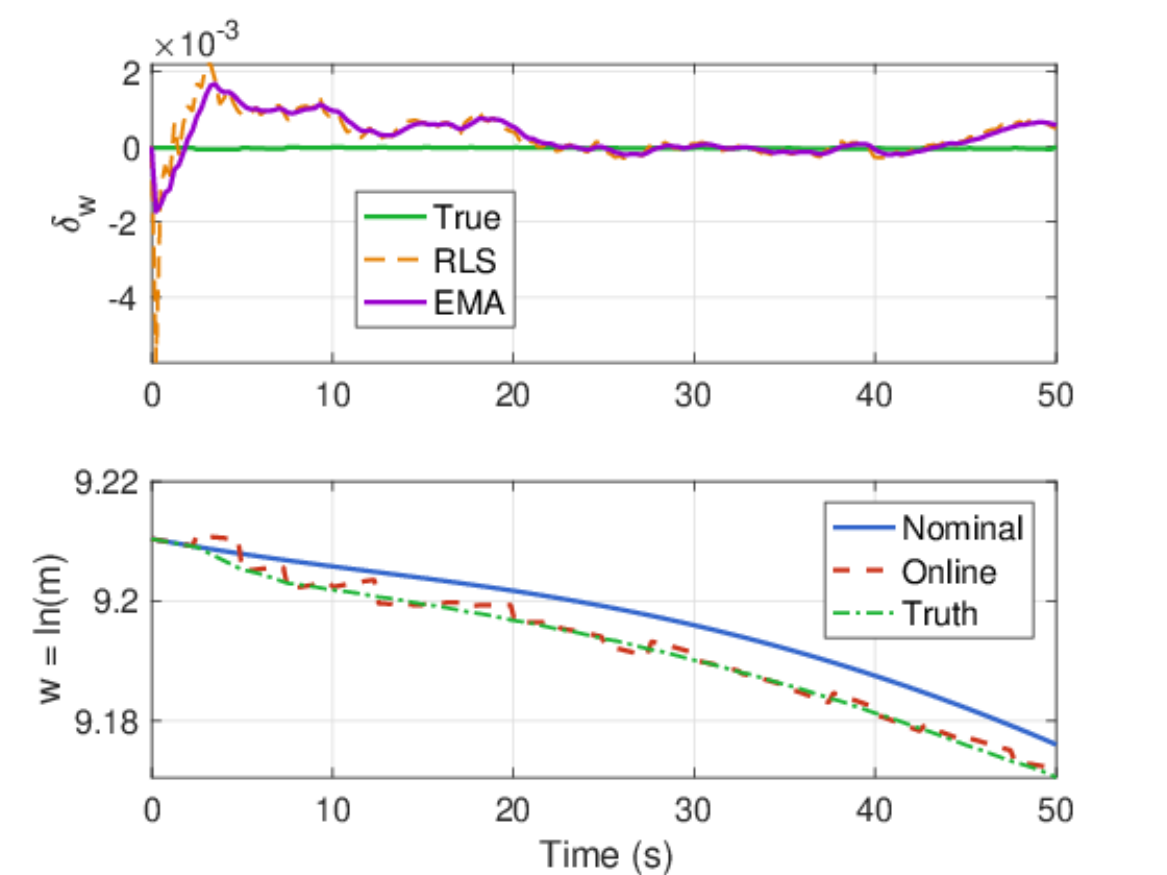}\\
\caption{Mass-channel perturbation and log-mass profile (Scenario~3).}
\label{fig:mass_perturb_3}
\end{figure}

\begin{figure}[H]
\centering
\includegraphics[width=3.25in]{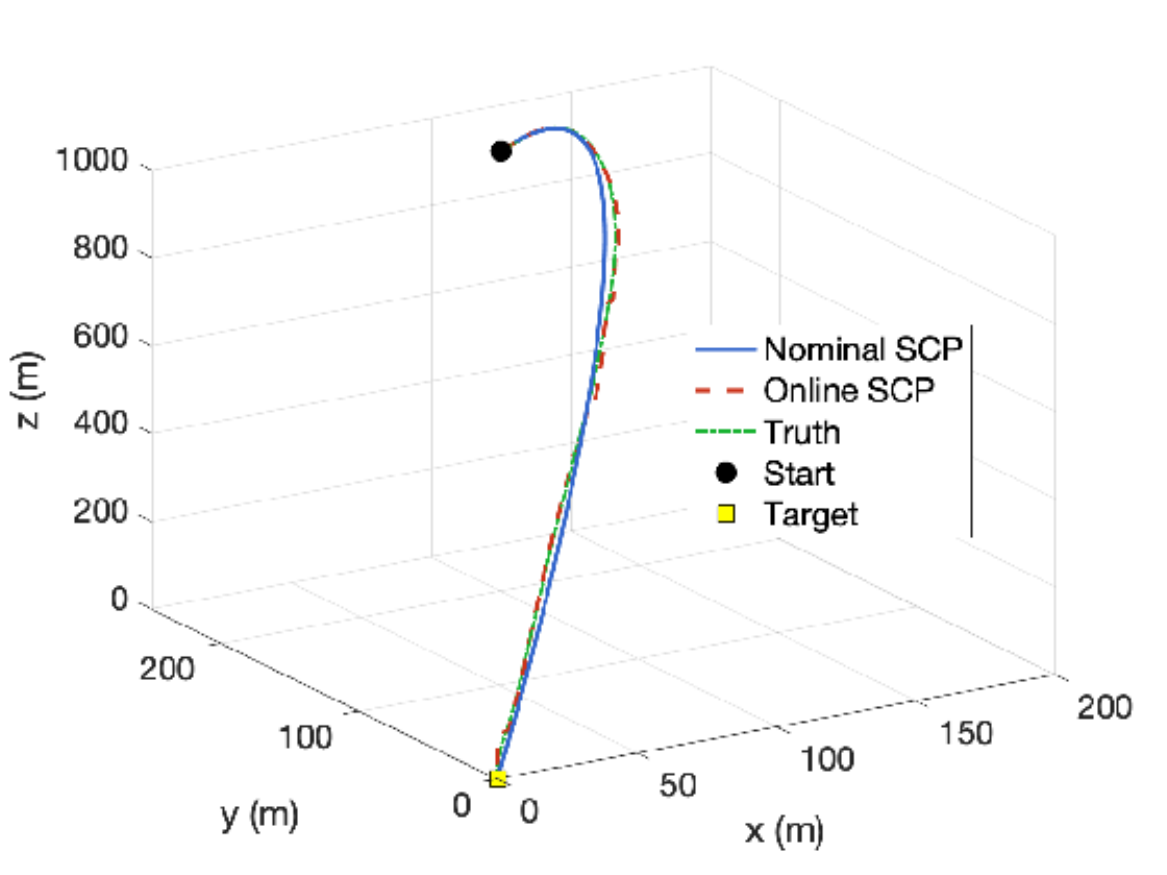}\\
\caption{3D landing trajectory (Scenario~4).}
\label{fig:traj_4}
\end{figure}

\begin{figure}[H]
\centering
\includegraphics[width=3.25in]{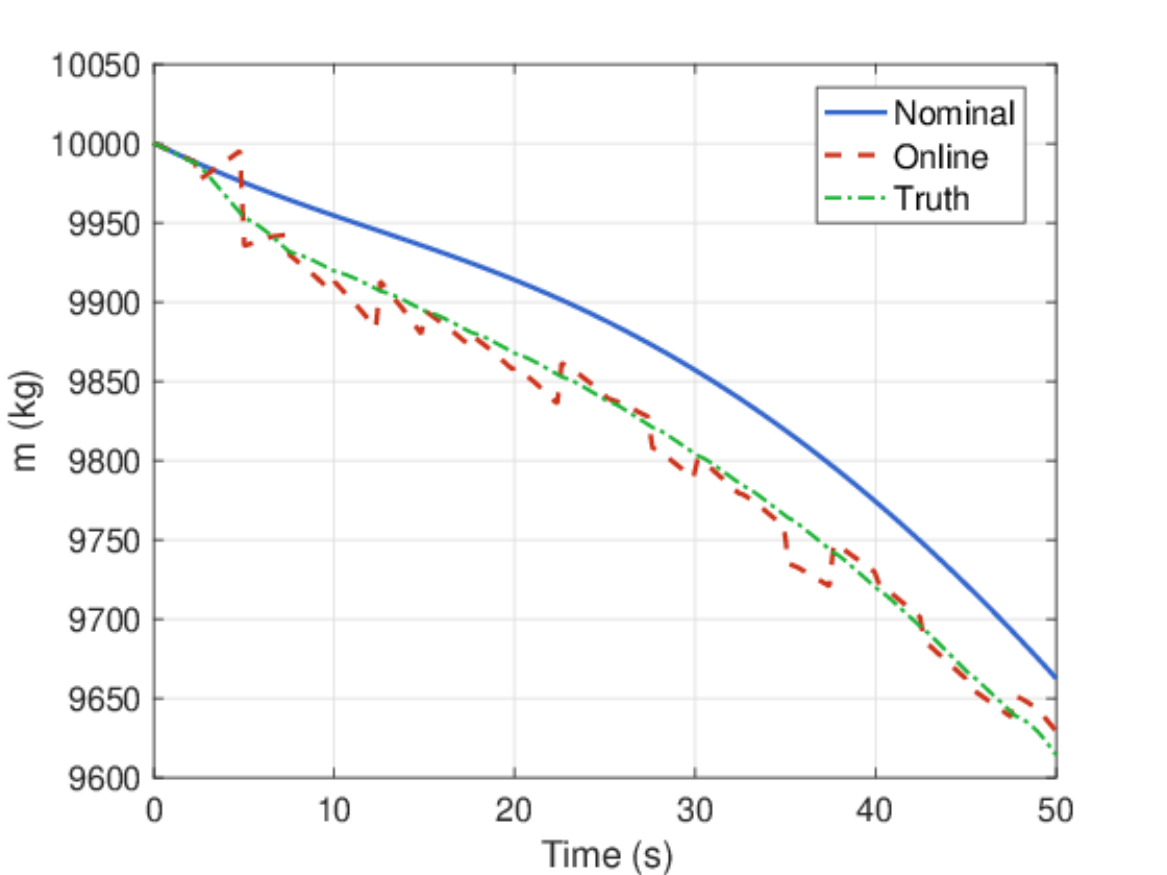}\\
\caption{Vehicle mass profile (Scenario~4).}
\label{fig:mass_4}
\end{figure}

\begin{figure}[H]
\centering
\includegraphics[width=3.25in]{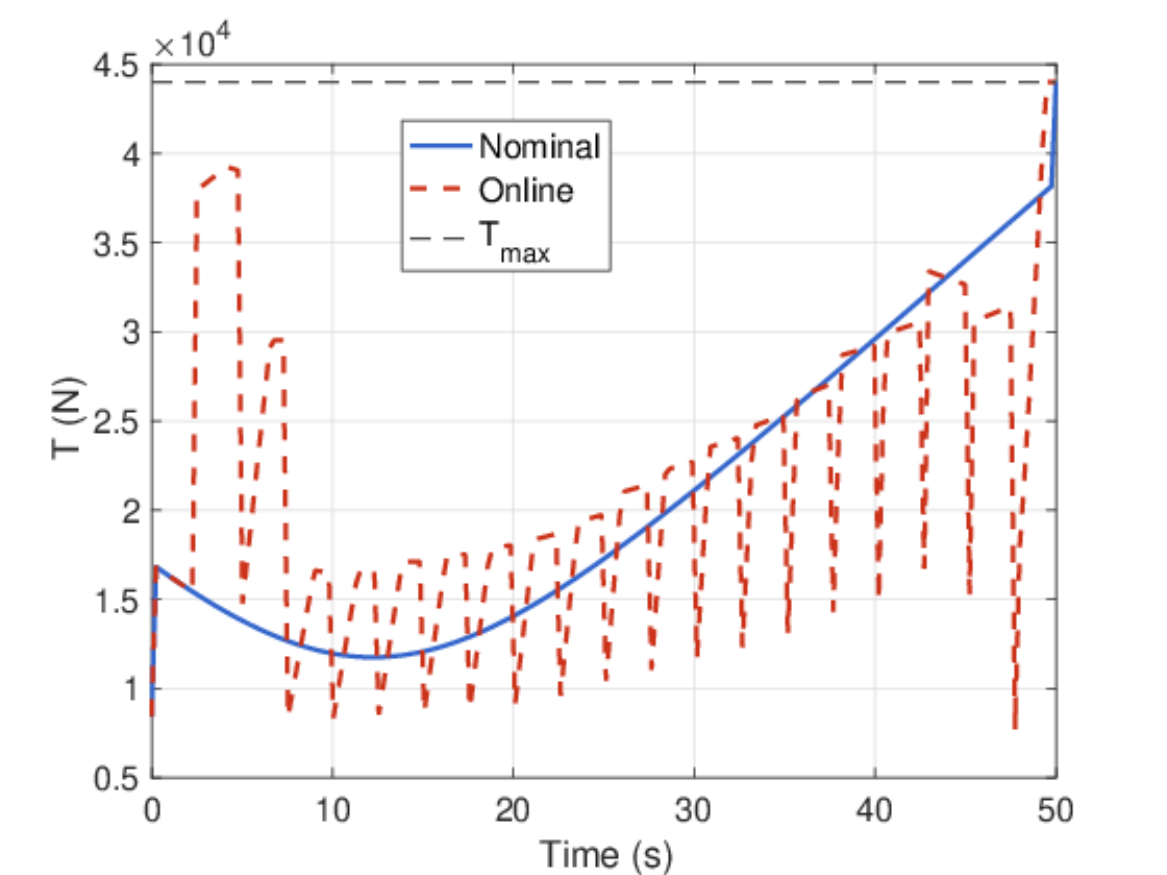}\\
\caption{Thrust magnitude profile (Scenario~4).}
\label{fig:thrust_4}
\end{figure}

\begin{figure}[H]
\centering
\includegraphics[width=3.25in]{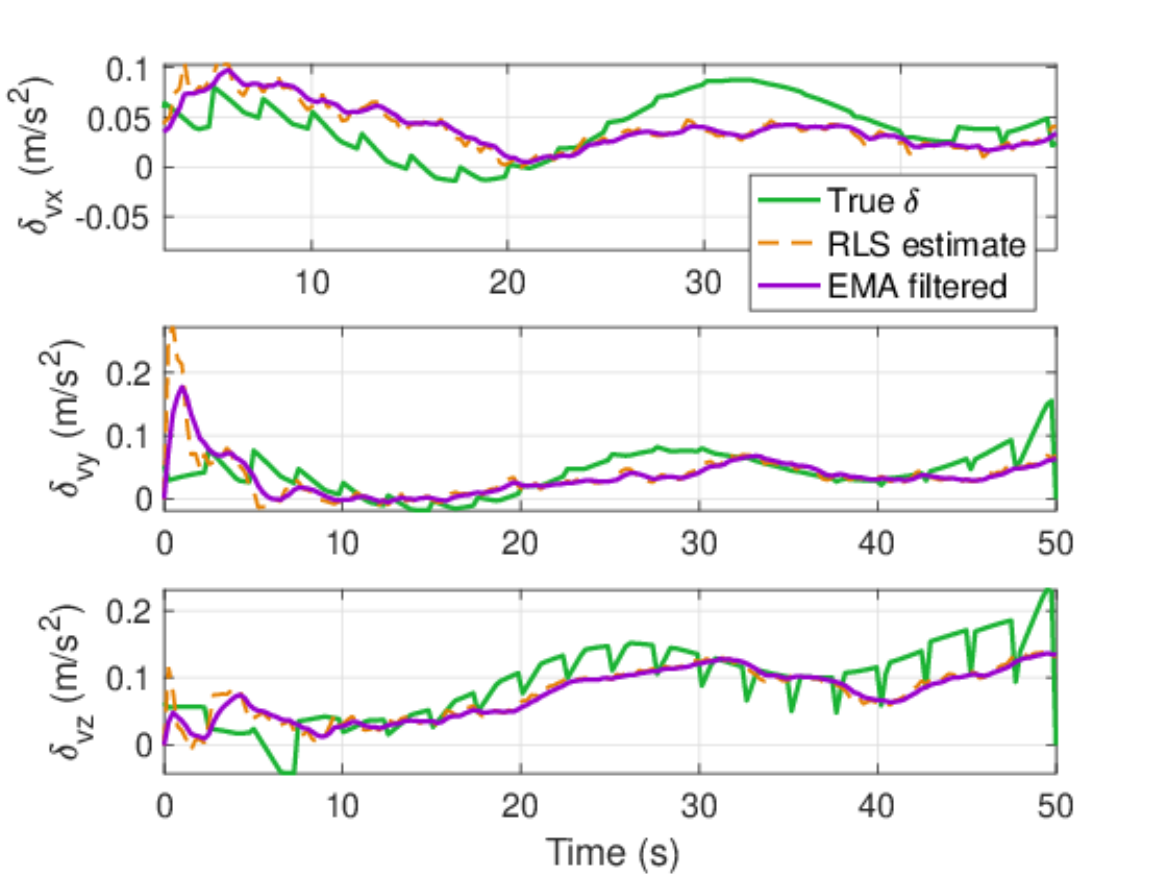}\\
\caption{Perturbation estimation (Scenario~4).}
\label{fig:perturb_4}
\end{figure}

\begin{figure}[H]
\centering
\includegraphics[width=3.25in]{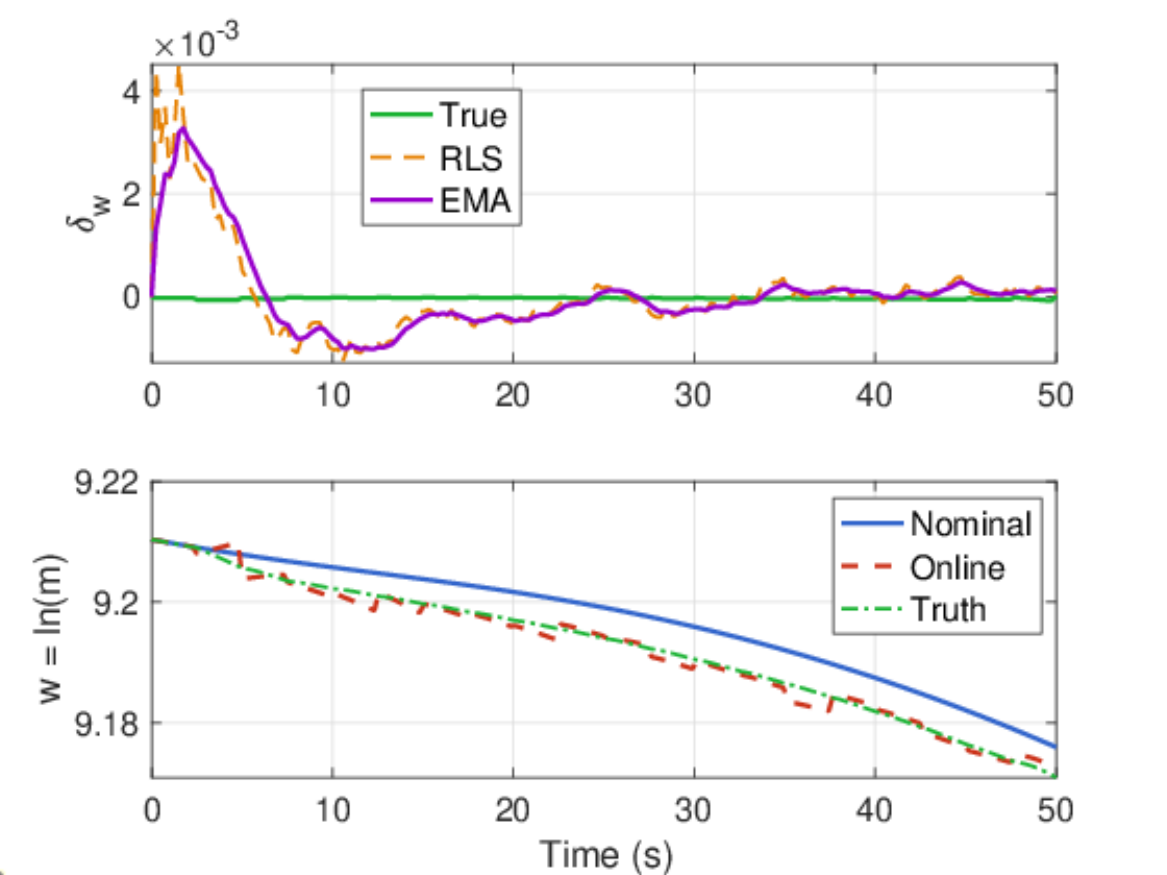}\\
\caption{Mass-channel perturbation and log-mass profile (Scenario~4).}
\label{fig:mass_perturb_4}
\end{figure}

\section{Conclusions}
\label{sec:conclusions}

This paper has developed and analyzed a guidance framework that integrates
online model identification with sequential convex programming for
three-dimensional lunar powered descent trajectory optimization. The
nonconvex landing problem was reformulated into a sequence of convex
second-order cone programs through a change of variables, successive
linearization of the thrust-to-mass constraint, and a second-order cone
relaxation of the thrust direction constraint that was shown to be lossless
under mild and physically reasonable conditions. An online identification
layer combining a recursive least squares filter with exponential
forgetting and an exponential moving average smoother was introduced to
estimate unknown gravitational, thrust-scale, and mass-gauging
perturbations from noisy flight measurements in real time, and the
resulting corrected bias was incorporated into the convex subproblem
solved at each guidance cycle without compromising its convexity or
feasibility. Both a baseline offline SCP algorithm and a receding-horizon
online extension with model identification were presented, and a
theoretical analysis established the losslessness of the convex
relaxation, the mean-square convergence and stability of the identification
filters, the guaranteed convergence of the SCP iteration to a stationary
point of the original nonconvex problem, and explicit characterizations of
the convergence radius and convergence rate, including the residual
neighborhood to which the online algorithm converges as a function of the
perturbation estimation error.
Numerical simulations across four
perturbation scenarios of increasing complexity confirmed these
theoretical predictions, demonstrating that the proposed online algorithm
consistently improves landing accuracy and propellant-budget fidelity
relative to an uncorrected nominal trajectory, while preserving the rapid,
predictable convergence behavior characteristic of convex optimization.
Collectively, these results indicate that online model identification and
sequential convex programming can be combined without sacrificing the
theoretical guarantees that make convex-optimization-based guidance
attractive for safety-critical, real-time deployment.
 

\acknowledgements 
Generative AI tools were used for language refinement during the preparation of this paper. All technical content and interpretations were developed by the author.

\bibliographystyle{IEEEtran}
\bibliography{MyBibFile}

\begin{thebibliography}{10}
\providecommand{\url}[1]{#1}
\csname url@samestyle\endcsname
\providecommand{\newblock}{\relax}
\providecommand{\bibinfo}[2]{#2}
\providecommand{\BIBentrySTDinterwordspacing}{\spaceskip=0pt\relax}
\providecommand{\BIBentryALTinterwordstretchfactor}{4}
\providecommand{\BIBentryALTinterwordspacing}{\spaceskip=\fontdimen2\font plus
\BIBentryALTinterwordstretchfactor\fontdimen3\font minus \fontdimen4\font\relax}
\providecommand{\BIBforeignlanguage}[2]{{%
\expandafter\ifx\csname l@#1\endcsname\relax
\typeout{** WARNING: IEEEtran.bst: No hyphenation pattern has been}%
\typeout{** loaded for the language `#1'. Using the pattern for}%
\typeout{** the default language instead.}%
\else
\language=\csname l@#1\endcsname
\fi
#2}}
\providecommand{\BIBdecl}{\relax}
\BIBdecl

\bibitem{wang2025a}
Z.~Wang, ``From new commercial moon landers to asteroid investigations, expect a slate of exciting space missions in 2025,'' 2025, \url{https://theconversation.com/from-new-commercial-moon-landers-to-asteroid-investigations-expect-a-slate-of-exciting-space-missions-in-2025-243645} [Accessed: 2026-06-25].

\bibitem{wang2025b}
------, ``Landing on the moon is an incredibly difficult feat -- 2025 has brought successes and shortfalls for companies and space agencies,'' 2025, \url{https://theconversation.com/landing-on-the-moon-is-an-incredibly-difficult-feat-2025-has-brought-successes-and-shortfalls-for-companies-and-space-agencies-256046} [Accessed: 2026-06-25].

\bibitem{song2020survey}
Z.~Song, C.~Wang, S.~Theil, D.~Seelbinder, M.~Sagliano, X.-f. Liu, and Z.-j. Shao, ``Survey of autonomous guidance methods for powered planetary landing,'' \emph{Frontiers of Information Technology \& Electronic Engineering}, vol.~21, no.~5, pp. 652--674, 2020.

\bibitem{lorenz2023planetary}
R.~D. Lorenz, ``Planetary landings with terrain sensing and hazard avoidance: A review,'' \emph{Advances in Space Research}, vol.~71, no.~1, pp. 1--15, 2023.

\bibitem{liu2026survey}
X.~Liu, S.~Li, and M.~Xin, ``Survey of trajectory optimization methods for mars entry and powered descent,'' \emph{Journal of Guidance, Control, and Dynamics}, vol.~49, no.~1, pp. 216--239, 2026.

\bibitem{acikmese2007convex}
B.~Acikmese and S.~R. Ploen, ``Convex programming approach to powered descent guidance for mars landing,'' \emph{Journal of Guidance, Control, and Dynamics}, vol.~30, no.~5, pp. 1353--1366, 2007.

\bibitem{wang2024survey}
Z.~Wang, ``A survey on convex optimization for guidance and control of vehicular systems,'' \emph{Annual Reviews in Control}, vol.~57, p. 100957, 2024.

\bibitem{cheng2019fast}
L.~Cheng, Z.~Wang, F.~Jiang, and J.~Li, ``Fast generation of optimal asteroid landing trajectories using deep neural networks,'' \emph{IEEE Transactions on Aerospace and Electronic Systems}, vol.~56, no.~4, pp. 2642--2655, 2019.

\bibitem{cheng2019real}
L.~Cheng, Z.~Wang, and F.~Jiang, ``Real-time control for fuel-optimal moon landing based on an interactive deep reinforcement learning algorithm,'' \emph{Astrodynamics}, vol.~3, no.~4, pp. 375--386, 2019.

\bibitem{cheng2020real}
L.~Cheng, Z.~Wang, Y.~Song, and F.~Jiang, ``Real-time optimal control for irregular asteroid landings using deep neural networks,'' \emph{Acta Astronautica}, vol. 170, pp. 66--79, 2020.

\bibitem{gaudet2020deep}
B.~Gaudet, R.~Linares, and R.~Furfaro, ``Deep reinforcement learning for six degree-of-freedom planetary landing,'' \emph{Advances in Space Research}, vol.~65, no.~7, pp. 1723--1741, 2020.

\bibitem{wang2018minimum}
Z.~Wang and M.~J. Grant, ``Minimum-fuel low-thrust transfers for spacecraft: A convex approach,'' \emph{IEEE Transactions on Aerospace and Electronic Systems}, vol.~54, no.~5, pp. 2274--2290, 2018.

\bibitem{wang2018optimization}
------, ``Optimization of minimum-time low-thrust transfers using convex programming,'' \emph{Journal of Spacecraft and Rockets}, vol.~55, no.~3, pp. 586--598, 2018.

\bibitem{wang2017constrained}
------, ``Constrained trajectory optimization for planetary entry via sequential convex programming,'' \emph{Journal of Guidance, Control, and Dynamics}, vol.~40, no.~10, pp. 2603--2615, 2017.

\bibitem{boyd2004convex}
S.~Boyd and L.~Vandenberghe, \emph{Convex optimization}.\hskip 1em plus 0.5em minus 0.4em\relax Cambridge University Press, 2004.

\bibitem{islam2019recursive}
S.~A.~U. Islam and D.~S. Bernstein, ``Recursive least squares for real-time implementation [lecture notes],'' \emph{IEEE Control Systems Magazine}, vol.~39, no.~3, pp. 82--85, 2019.

\bibitem{klinker2011exponential}
F.~Klinker, ``Exponential moving average versus moving exponential average,'' \emph{Mathematische Semesterberichte}, vol.~58, no.~1, pp. 97--107, 2011.

\bibitem{mao2018successive}
Y.~Mao, M.~Szmuk, X.~Xu, and B.~A{\c{c}}ikmese, ``Successive convexification: A superlinearly convergent algorithm for non-convex optimal control problems,'' \emph{arXiv preprint arXiv:1804.06539}, 2018.

\bibitem{haykin2008adaptive}
S.~S. Haykin, \emph{Adaptive filter theory}.\hskip 1em plus 0.5em minus 0.4em\relax Pearson Education India, 2008.

\bibitem{lofberg2004yalmip}
J.~Lofberg, ``Yalmip: A toolbox for modeling and optimization in matlab,'' in \emph{2004 IEEE International Conference on Robotics and Automation}, 2004, pp. 284--289.

\bibitem{domahidi2013ecos}
A.~Domahidi, E.~Chu, and S.~Boyd, ``Ecos: An socp solver for embedded systems,'' in \emph{2013 European control conference (ECC)}, 2013, pp. 3071--3076.

\end{thebibliography}

\thebiography
\begin{biographywithpic}
{Zhenbo Wang}{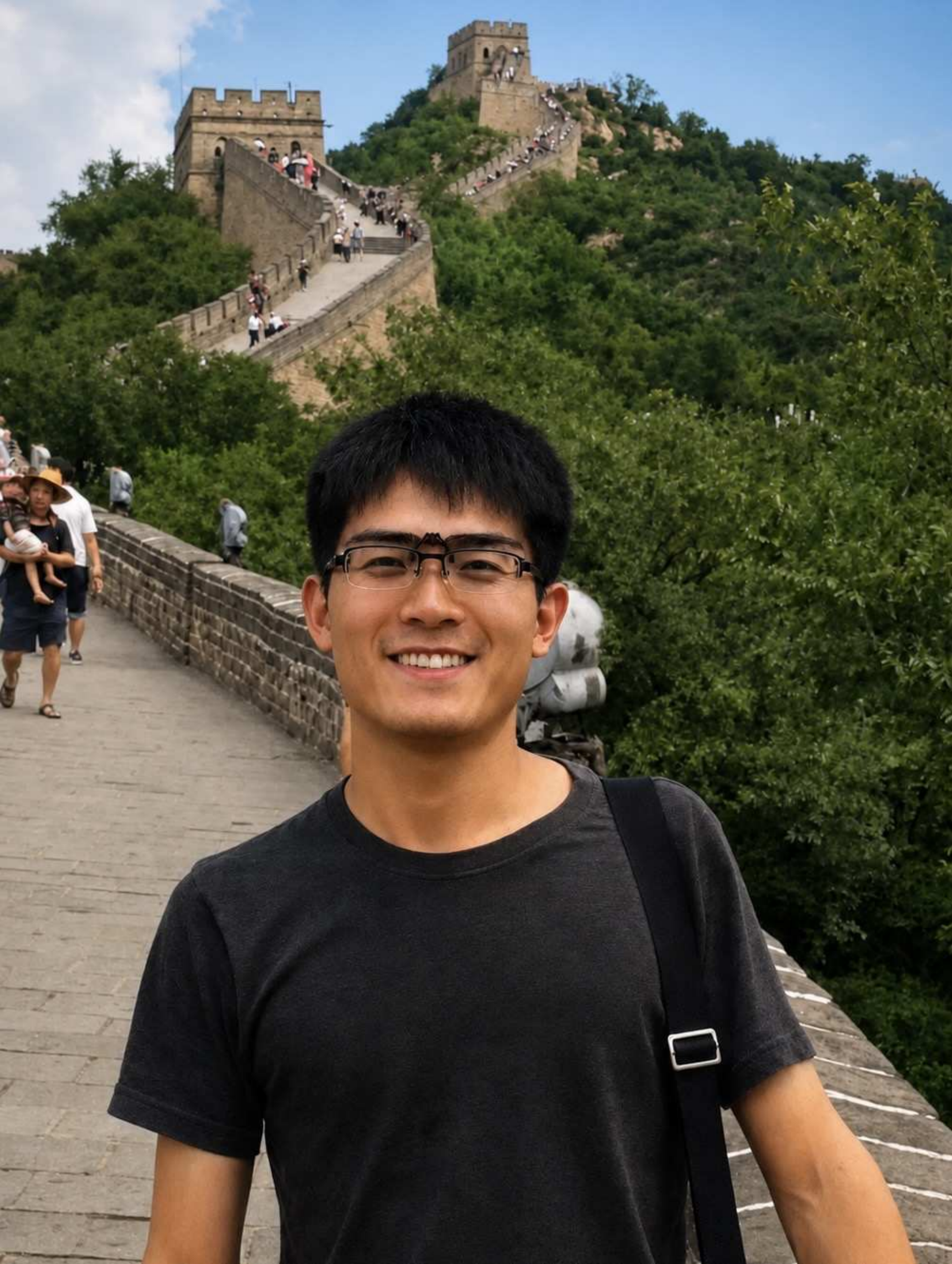} received his B.E. degree in Astronautics from Nanjing University of Aeronautics and Astronautics in 2010 and his M.E. degree in Control Engineering from Beihang University in 2013. In 2018, he received his Ph.D. degree in Aeronautics and Astronautics from Purdue University and joined the University of Tennessee Knoxville (UTK) as an Assistant Professor. He is now an Associate Professor in the Department of Mechanical and Aerospace Engineering and the director of the Autonomous Systems Laboratory at UTK. He is a recipient of the 2023 NSF Faculty Early Career Development Program (CAREER) Award, the 2023 Louis and Ann Hoffman Endowed Excellence in Research Award, and the 2024 Professional Promise in Research Award. His research interests are control, optimization, and machine learning for various engineering applications including space systems, air vehicles, connected and automated vehicles, and power and energy systems. He is a Senior Member of the American Institute of Aeronautics and Astronautics (AIAA) and a member of the AIAA Atmospheric Flight Mechanics (AFM) Technical Committee. He is a Senior Editor of \textit{IEEE Transactions on Aerospace and Electronic Systems}.
\end{biographywithpic}

\end{document}